\documentclass[11pt,a4paper]{article}

\usepackage[english]{babel}

\usepackage[letterpaper,top=2cm,bottom=2cm,left=3cm,right=3cm,marginparwidth=1.75cm]{geometry}

\usepackage{graphicx} 
\usepackage{amsfonts}
\usepackage{amssymb}
\usepackage{mathtools}
\usepackage[colorlinks=true, allcolors=blue]{hyperref}
\usepackage{amsthm}
\usepackage{mathrsfs}
\usepackage{enumerate}
\usepackage{cases}

\newtheorem{theorem}{Theorem}[section]

\newtheorem{lemma}{Lemma}[section]
\newtheorem{definition}{Definition}[section]

\newtheorem{corollary}{Corollary}[section]
\newtheorem{remark}{Remark}[section]

\allowdisplaybreaks

\title{A Parameterization of Small-Amplitude KP Finite-Gap Solutions via Classical Schottky Uniformization and Persistence of One- and Two-Gap Solutions}

\author{Haoxuan Liu  \thanks{The work was supported  by National Natural Science Foundation of China (Grant No. 12371189).}$^1$,  Zuhong You \footnotemark[1] $^1$ and Xiaoping Yuan \footnotemark[1] $^1$  \\
$^{1}$School of Mathematical Sciences, Fudan University, Shanghai 200433, P. R. China   }
\date{}

\begin{document}

\maketitle

\begin{abstract}
We develop a parameterization of small-amplitude finite-gap solutions to the KP equation with prescribed spatial wavenumber vectors using classical Schottky uniformization. This parameterization is suitable for a Lyapunov–Schmidt reduction.  Using this parameterization, we prove the persistence, under Hamiltonian perturbations, of periodic one-gap solutions for both KP-I and KP-II and of bi-periodic two-gap solutions for KP-I. 
\end{abstract}

\section{Introduction}
\label{sec:introduction}

The Kadomtsev--Petviashvili (KP) equation is a fundamental two-dimensional
nonlinear dispersive equation of the form
\begin{equation}
    u_t+u_{xxx}-\lambda\partial_x^{-1}u_{yy}+uu_x=0,
    \qquad \lambda=\pm1.
    \label{kp}
\end{equation}
Under the convention adopted in this paper, the cases $\lambda=1$ and
$\lambda=-1$ correspond to the KP-I and KP-II equations, respectively.
We consider \eqref{kp} on the two-dimensional torus
\[
    (x,y)\in\mathbb{T}_{\gamma}^{2},
    \quad
    \mathbb{T}_{\gamma}:=\mathbb{R}/\gamma\mathbb{Z},
    \quad
    \gamma>0,
\]
subject to the zero-$x$-mean condition
\begin{equation*}
    \int_{\mathbb{T}_{\gamma}}u(t,x,y)\,dx=0
    \qquad
    \text{for a.e. }y\in\mathbb{T}_{\gamma}.
\end{equation*}

The KP equation arises as an asymptotic model for weakly nonlinear long
waves with weak transverse effects; see
\cite{kadomtsev1970stability,ablowitz1981solitons,novikov1984theory}
for classical accounts. They also form one of the basic
classes of integrable equations in two spatial dimensions, admitting a Lax
representation, inverse-scattering formulations, and an associated
integrable hierarchy
\cite{ablowitz1981solitons,novikov1984theory,ablowitz1991solitons}. 
Their integrable structure gives rise to several distinguished families of
exact solutions, including line solitons, periodic traveling waves,
and algebro-geometric finite-gap solutions. The latter can be
represented in terms of Riemann theta functions associated with
compact Riemann surfaces and Abelian differentials; see, for example,
\cite{krichever1977methods,dubrovin1981theta,
ablowitz1991solitons,belokolos1994algebro}.

The reality and regularity of theta-functional KP solutions require
additional restrictions on the underlying spectral data. Dubrovin and
Natanzon obtained an algebro-geometric characterization of smooth
real finite-gap solutions of the KP equation \cite{dubrovin1989real}. More recently,
Ichikawa and Kodama constructed real Schottky groups uniformizing Riemann
surfaces associated with real finite-gap solutions and showed that KP
solitons arise through a corresponding degeneration \cite{ichikawa2025kp}.

A genus-$g$ finite-gap solution depends on $g$ phase variables and can be regarded
as a finite-dimensional invariant torus of the integrable KP flow.
The richness of the exact solution theory for the KP equation naturally
leads to a basic question: which of these structures persist beyond the
integrable regime? In the finite-gap setting, the invariant-torus
interpretation above turns this into a KAM-type problem, namely whether
the corresponding tori survive under Hamiltonian perturbations.
This problem was also raised explicitly by Bourgain and Kuksin in their
respective survey articles \cite{bourgain2004recent,kuksin2004fifteen}.

For the periodic KdV equation, the persistence of finite-gap invariant
tori under Hamiltonian perturbations was developed in the foundational
work of Kuksin \cite{Kuk89,Kuk98kdv}; see also \cite{Kuk00} for a systematic
account. A central ingredient in this theory is the nondegeneracy of the
unperturbed finite-gap family. Bobenko and Kuksin established the relevant
nondegeneracy for periodic finite-gap KdV solutions by means of Schottky
uniformization \cite{Kuk91}. Closely related to the parameterization problem
considered here, Bikbaev and Kuksin investigated frequency and wavenumber
vectors as coordinates on manifolds of real finite-gap solutions, emphasizing
their role in perturbation and averaging theories \cite{bikbaev1993parametrization}.

Motivated by this question, the present paper studies the persistence
of small-amplitude finite-gap solutions when the KP equation is
subjected to a Hamiltonian perturbation.

The first difficulty is to find a parameterization of the unperturbed
finite-gap solutions that is suitable for perturbative analysis. A
Lyapunov--Schmidt reduction requires an explicit finite-dimensional
family of unperturbed solutions, together with internal parameters for
which the dependence of the associated frequencies can be controlled.
For the periodic KdV equation,
global Birkhoff coordinates provide a canonical system of action
variables and hence a natural description of its finite-gap
invariant tori \cite{kappeler2003perturbed}.
In the present KP problem, however, one cannot rely on an analogous global
Birkhoff-coordinate framework.
Although the Its--Matveev formula gives an
explicit theta-functional representation of KP finite-gap solutions,
the associated Riemann-surface data do not directly 
yield a small-amplitude coordinate system in which the prescribed periodic
spatial wavenumber vectors are kept fixed.

Our construction is related to the preceding works, but differs in both its
purpose and its choice of parameters. In the KP setting, one must
simultaneously prescribe the two spatial wavenumber vectors associated
with the $x$- and $y$-directions. Rather than using the temporal frequency
and spatial wavenumber vectors themselves as local coordinates, we take
the square roots of the Schottky multipliers as amplitude parameters and
vary the fixed points of the Schottky generators such that the prescribed
spatial wavenumber vectors are retained. The period matrix and the temporal
frequency vector are then obtained as functions of these amplitudes.

Thus, the novelty of the present construction lies not in the use of Schottky uniformization by itself, but in a constructive KP-specific small-amplitude parameterization near the degeneration limit that preserves the prescribed periodic spatial wavenumber vectors and is directly adapted to the subsequent Lyapunov--Schmidt reduction.

Our parameterization result may be summarized as follows.

\begin{theorem} \label{intro}
Fix $g\in \mathbb{Z}_+$. If $g>1$, suppose that $\sqrt{3}\pi/\gamma\in\mathbb{R}\setminus\mathbb{Q}$. Set $\boldsymbol m = (m_1,\cdots, m_g), \boldsymbol n= (n_1,\cdots, n_g) \in \mathbb{Z}^g$, satisfying $m_1\cdots m_g\neq 0$ and $m_j n_k-m_k n_j\neq 0$, for $j,k\in\mathcal{G}$ and $j\neq k$.

For $i,j\in\mathcal{G}$, let
\begin{align*}
    &A_{j,0}^{(1)}=-\dfrac{n_j}{2\sqrt{3} m_j}-\dfrac{\pi}{\gamma}m_j, B_{j,0}^{(1)}=-\dfrac{n_j}{2\sqrt{3} m_j}+\dfrac{\pi}{\gamma}m_j,\\
    &A_{j,0}^{(2)}=\dfrac{n_j}{2\sqrt{3}m_j}+i\dfrac{\pi}{\gamma}m_j,\ B_{j,0}^{(2)}=\dfrac{n_j}{2\sqrt{3}m_j}-i\dfrac{\pi}{\gamma}m_j.
\end{align*}
and
\begin{align*}
    &\kappa_{ij,0}^{(\iota)} = \{ B_{j,0}^{(\iota)}, A_{j,0}^{(\iota)}, B_{i,0}^{(\iota)}, A_{i,0}^{(\iota)}\}, \quad \{a,b,c,d\} := \frac{(a-c)(b-d)}{(a-d)(b-c)},\quad i\neq j,\quad \iota =1,2,\\
    &\epsilon_{jj}:=0, \qquad \epsilon_{ij} :=
    \begin{cases}
        0,\textup{ if }\kappa_{ij,0}^{(1)}>0,\\
        \pi,\textup{ if }\kappa_{ij,0}^{(1)}<0,
    \end{cases}\qquad  E := \bigl( \epsilon_{ij} \bigr)_{i,j\in\mathcal G}.
\end{align*}

$(i)$ 
There exists a constant $\tilde\rho \in (0,1)$ and real-analytic functions $B_{real}^{(1)}(\boldsymbol r) =(B^{(1)}_{real,ij}(\boldsymbol r))_{i,j\in \mathcal{G}} \in \mathbb{R}^{g\times g},\ \boldsymbol \omega_0^{(1)}(\boldsymbol r)\in \mathbb{R}^g$ on $(0,\tilde\rho)^g$.
For every $\boldsymbol d\in \mathbb{R}^g$ and $\boldsymbol r \in (0,\tilde\rho)^g$, these functions determine a finite-gap solution
\begin{equation*}
    u_0^{(1)}(t,x,y;\boldsymbol r) = 12\frac{\partial^2}{\partial x^2} \log \theta\left( \frac{2\pi i}{\gamma}(\boldsymbol m x+\boldsymbol n y+ \boldsymbol \omega_0^{(1)}(\boldsymbol r)t +\boldsymbol d); B_{real}^{(1)}(\boldsymbol r)+i E \right)
\end{equation*}
to the KP-I equation \eqref{kp}.
Moreover, for $i,j\in\mathcal{G}$,
\begin{equation*}
    B^{(1)}_{real,jj}(\boldsymbol r)=\log r_j^2+O(\boldsymbol r^2),\quad B^{(1)}_{real,ij}(\boldsymbol r)=\log |\kappa_{ij,0}^{(1)}| + O(\boldsymbol r^2),\quad i\neq j.
\end{equation*}

$(ii)$ There exists a constant $\tilde\rho \in (0,1)$ and real-analytic functions $ B_{real}^{(2)}(\boldsymbol r) = (B^{(2)}_{real,ij}(\boldsymbol r))_{i,j\in \mathcal{G}}\in \mathbb{R}^{g\times g},\ \boldsymbol \omega_0^{(2)}(\boldsymbol r)\in \mathbb{R}^g$ on $(0,\tilde\rho)^g$.
For every $\boldsymbol d\in \mathbb{R}^g$ and $\boldsymbol r \in (0,\tilde\rho)^g$, these functions determine a finite-gap solution
\begin{equation*}
    u_0^{(2)}(t,x,y;\boldsymbol r) = 12\frac{\partial^2}{\partial x^2} \log \theta\left( \frac{2\pi i}{\gamma}(\boldsymbol m x+\boldsymbol n y+\boldsymbol \omega_0^{(2)}(\boldsymbol r)t +\boldsymbol d); B_{real}^{(2)}(\boldsymbol r) \right)
\end{equation*}
to the KP-II equation \eqref{kp}. Moreover, for $i,j\in \mathcal{G}$,
\begin{equation*}
    B^{(2)}_{real,jj}(\boldsymbol r)=\log r_j^2+O(\boldsymbol r^2),\quad B^{(2)}_{real,ij}(\boldsymbol r)=\log \kappa_{ij,0}^{(2)} + O(\boldsymbol r^2),\quad i\neq j.
\end{equation*}
\end{theorem}

\bigskip

With this parameterization at hand, we turn to the perturbed problem.
We now specify the classes of Hamiltonian differential perturbations
considered in this paper. Denote by
\[
J^s u
:=
\bigl(
\partial_x^{\alpha_1}\partial_y^{\alpha_2}u
\bigr)_{\alpha_1+\alpha_2\leq s}
\]
the spatial \(s\)-jet of \(u\).
 We consider the perturbed KP
equation
\begin{equation}
u_t+u_{xxx}-\lambda\partial_x^{-1}u_{yy}
   +uu_x+\mathcal P[u]=0,
\qquad (x,y)\in\mathbb T_\gamma^2,
\label{eq:perturbed-KP}
\end{equation}
where the perturbation vector field is of the Hamiltonian form
\begin{equation*}
\mathcal P[u]
=
\partial_x Q(J^su),
\qquad
 Q(J^su)
=
\frac{\delta\mathcal H_{\mathrm{pert}}}{\delta u}[u].
\end{equation*}
Here \(\mathcal H_{\mathrm{pert}}\) is a real-analytic Hamiltonian
functional whose Taylor expansion at the origin starts at degree
\(p+1\), with the integer \(p\geq3\). Equivalently, the real-analytic differential
expression \( Q\) vanishes to order at least \(p\) at the zero
jet. All Taylor coefficients are assumed to be real.

We use two classes of perturbations.
Assume that $Q_1,Q_2$ are variational gradients of real-analytic Hamiltonian functionals.
In the sequel, only the first two homogeneous terms of the perturbation contribute to the leading bifurcation equations.
For the one-gap persistence result, we therefore set
\begin{equation}
\mathcal P_1[u] := \partial_x Q_1(J^2u),\quad
 Q_1(J^2u) := u^p+u^{p+1} +Q_{1,>p+1} \bigl(u,u_x,u_y,u_{xx},u_{xy},u_{yy}\bigr).
\label{eq:intro-P1}
\end{equation}
The remainder $Q_{1,>p+1}(z_0,z_1,z_2,z_3,z_4,z_5)$ is real analytic on the polydisc
\[
\bigl\{
(z_0,\cdots,z_5)\in\mathbb C^6:
|(z_0,\cdots,z_5)|_\infty<R_1
\bigr\},
\]
and its Taylor expansion at the origin vanishes to order at least
\(p+2\). We further impose the symmetry condition
\begin{equation}
Q_{1,>p+1}
(z_0,-z_1,-z_2,z_3,z_4,z_5)
=
Q_{1,>p+1}
(z_0,z_1,z_2,z_3,z_4,z_5).
\label{eq:intro-P1-parity}
\end{equation}
Here the variables $(z_0,z_1,z_2,z_3,z_4,z_5)$ correspond respectively to
$ (u,u_x,u_y,u_{xx},u_{xy},u_{yy})$.

For the two-gap persistence result, we consider a more structured
Hamiltonian gradient. Set
\[
J_x^2u:=(u,u_x,u_{xx}),
\qquad
J_y^1u:=(u,u_y),
\]
and write
\begin{equation}
 Q_2(J^2u)
:=
 Q_{2,1}(J_x^2u)
+
 Q_{2,2}(J_y^1u),
\label{eq:two-gap-Q-decomposition}
\end{equation}
where
\begin{align}
 Q_{2,1}(u,u_x,u_{xx})
&:=
a_pu^p+a_{p+1}u^{p+1}
+Q_{2,1,>p+1}(u,u_x,u_{xx}),
\label{eq:intro-P21}
\\
 Q_{2,2}(u,u_y)
&:=
b_pu^p+b_{p+1}u^{p+1}
+Q_{2,2,>p+1}(u,u_y).
\label{eq:intro-P22}
\end{align}
Here $a_p,\ a_{p+1},\ b_p,\ b_{p+1}\in\mathbb R$ are fixed constants. The corresponding perturbation vector field
\begin{equation}
\mathcal P_2[u]
:=
\partial_x Q_2(J^2u)
=
\partial_x Q_{2,1}(J_x^2u)
+
\partial_x Q_{2,2}(J_y^1u).
\label{eq:two-gap-P}
\end{equation}
The remainders $Q_{2,1,>p+1}, Q_{2,2,>p+1}$ are real analytic on the respective polydiscs
\[
|(z_0,z_1,z_2)|_\infty<R_2,
\qquad
|(z_0,z_1)|_\infty<R_2,
\]
and their Taylor expansions at the origin vanish to order at least
\(p+2\). We assume that
\begin{align}
Q_{2,1,>p+1}(z_0,-z_1,z_2)
&=
Q_{2,1,>p+1}(z_0,z_1,z_2),
\label{eq:intro-P21-parity}
\\
Q_{2,2,>p+1}(z_0,-z_1)
&=
Q_{2,2,>p+1}(z_0,z_1).
\label{eq:intro-P22-parity}
\end{align}

The symmetry assumptions imposed on the perturbations ensure that the
bifurcation equations preserve the real subspace and, in particular,
that the resulting corrections to the frequencies remain real.

With the two perturbation classes specified above, we now state our
persistence results. The first result concerns one-gap solutions and
applies to both the KP-I and KP-II equations under perturbations of
the form \(\mathcal P_1\). The second result concerns two-gap
solutions of the KP-I equation under the more structured perturbation
class \(\mathcal P_2\).

For convenience, set
\[
\lambda_\iota:=(-1)^{\iota-1},\qquad \iota=1,2,
\]
so that \(\iota=1\) corresponds to the KP-I equation and \(\iota=2\) to the
KP-II equation.

\begin{theorem}[Persistence of one-gap solutions]
\label{thm:intro-one-gap}
For each $\iota=1,2$, consider the perturbed KP equation \eqref{eq:perturbed-KP} with $\mathcal P=\mathcal P_1$ and $\lambda = \lambda_\iota$,
where the perturbation term satisfies the assumptions stated in \eqref{eq:intro-P1}--\eqref{eq:intro-P1-parity}. Let $m_0,n_0\in\mathbb{Z}$ with $m_0\neq0$.

With Theorem \ref{intro}, there exist a constant $\rho>0$ and real-analytic functions $q^{(\iota)}(r),\ \omega_0^{(\iota)}(r)$ on $(0,\rho)$. For every $d^{(\iota)}\in\mathbb{R}$ and $r \in (0,\rho)$, these functions determine a periodic one-gap solution $u_0^{(\iota)}(t,x,y;r)$ of the unperturbed KP equation with $\lambda=\lambda_{\iota}$, where
\begin{align*}
    &u_0^{(\iota)}(t,x,y;r)
    =
    \tilde u_0^{(\iota)}(\zeta_0;r)
    =
    \sum_{k\in\mathbb{Z}\setminus\{0\}}
    \hat{\tilde u}_0^{(\iota)}(k;r)
    \exp
    \left(
        \frac{2\pi i}{\gamma}
        k\zeta_0
    \right),\
    \zeta_0 = \omega_0^{(\iota)}(r)t+m_0x+n_0y+d^{(\iota)},\\
    &\hat{\tilde u}_0^{(\iota)}(k;r)
    =
    \frac{48m_0^2\pi^2}{\gamma^2}
    (-1)^{|k|}
    \frac{
        |k|\bigl(q^{(\iota)}(r)\bigr)^{|k|}
    }{
        1-\bigl(q^{(\iota)}(r)\bigr)^{2|k|}
    },\quad
    q^{(\iota)}(r)=r+O(r^3).
\end{align*}

For sufficiently small $\rho' \in (0,\rho)$ and $r \in (0,\rho')$, there exists a real-analytic
frequency $\omega^{(\iota)}(r)$ such that $u^{(\iota)}(t,x,y;r)$ is a periodic traveling-wave solution of the perturbed equation, where
\begin{align*}
    &u^{(\iota)}(t,x,y;r)
    =
    \tilde u^{(\iota)}(\zeta;r)
    =
    \sum_{k\in\mathbb{Z}\setminus\{0\}}
    \hat{\tilde u}^{(\iota)}(k;r)
    \exp
    \left(
        \frac{2\pi i}{\gamma}
        k\zeta
    \right), \quad
    \zeta = \omega^{(\iota)}(r)t +m_0x+n_0y+d^{(\iota)}.
\end{align*}
Moreover, for some $\sigma>0$,
\begin{align*}
    &\sum_{|k|\neq1} \left| \hat{\tilde u}^{(\iota)}(k;r) - \hat{\tilde u}_0^{(\iota)}(k;r) \right|^2 \exp \left( \frac{4\pi\sigma}{\gamma}|k| \right) \lesssim |r|^{2p},\\
    &\hat{\tilde u}^{(\iota)}(1;r) = \hat{\tilde u}_0^{(\iota)}(1;r), \quad \hat{\tilde u}^{(\iota)}(-1;r) = \hat{\tilde u}_0^{(\iota)}(-1;r).
\end{align*}
The perturbed frequency satisfies
\begin{equation}
    \omega^{(\iota)}(r)-\omega_0^{(\iota)}(r)
    =
    \begin{cases}
        O(r^{p-1}), & p \text{ odd},
        \\[2mm]
        O(r^p), & p \text{ even}.
    \end{cases}
    \label{eq:intro-one-gap-frequency}
\end{equation}
The leading-order terms in \eqref{eq:intro-one-gap-frequency} are computed
explicitly in Section~\ref{sec:one-gap-2}.
\end{theorem}

We next consider the persistence of two-gap solutions. Our result in
this case is obtained for the KP-I equation under perturbations of the
form \(\mathcal P_2\).

\begin{theorem}[Persistence of two-gap solutions]
\label{thm:intro-two-gap}
Consider the perturbed KP-I equation \eqref{eq:perturbed-KP} with $\mathcal P=\mathcal P_2$,
where the perturbation term satisfies the assumptions stated in \eqref{eq:two-gap-Q-decomposition}--\eqref{eq:intro-P22-parity}.
Assume that $\frac{4\pi^2}{\gamma^2}\notin\mathbb{Q}$ and $(m_1,n_1),(m_2,n_2)\in\mathbb{Z}^2$
satisfy $m_1m_2\neq0,\ m_1n_2-m_2n_1\neq0$. In particular, $\gamma = 1$ satisfies the restrictions.
Set $S_0 = \{(k_1,k_2) \in \mathbb{Z}^2:m_1 k_1 + m_2 k_2 = 0\}$.

With Theorem \ref{intro}, there exist constants $\rho,\eta \in (0,1)$, and real-analytic functions $\hat{u}_0(k_1,k_2;\boldsymbol r),\allowbreak \omega_{j,0}(\boldsymbol r)$ on $(0,\rho)^2$ such that the following results hold:

For every $d_1,d_2\in\mathbb{R}$ and $\boldsymbol r \in (0,\rho)^2$, these functions determine a bi-periodic two-gap solution $u_0(t,x,y;\boldsymbol r)$ of the unperturbed KP-I equation, where
\begin{align*}
    &u_0(t,x,y;\boldsymbol r)
    =
    \tilde u_0(\boldsymbol \zeta_0;\boldsymbol r)
    =
    \sum_{(k_1,k_2)\in\mathbb{Z}^2\setminus S_0}
    \hat{\tilde u}_0(k_1,k_2;\boldsymbol r)
    \exp
    \left\{
        \frac{2\pi i}{\gamma}
        \left(
            k_1\zeta_{1,0}
            +
            k_2\zeta_{2,0}
        \right)
    \right\},\\
    &\zeta_{j,0} = m_jx+n_jy+\omega_{j,0}(\boldsymbol r)t+d_j, \quad j=1,2,\quad \boldsymbol\zeta_0 = (\zeta_{1,0},\zeta_{2,0}),\\
    &\left| \hat{\tilde u}_0(k_1,k_2;\boldsymbol r) \right| \leq C (k_1m_1+k_2m_2)^2 \left(\frac{|r_1|}{\eta}\right)^{|k_1|}\left(\frac{|r_2|}{\eta}\right)^{|k_2|}, \quad (k_1,k_2)\in\mathbb{Z}^2\setminus S_0.
\end{align*}
For sufficiently small $\rho' \in (0,\rho)$ and $\boldsymbol r\in(\rho'/2,\rho')^2$, there exist real-analytic frequencies $\omega_1(\boldsymbol r),\omega_2(\boldsymbol r)$ such that $u(t,x,y;\boldsymbol r)$ is a bi-periodic traveling-wave solution of the perturbed KP-I equation, where
\begin{align*}
    &u(t,x,y;\boldsymbol r) = \tilde u(\boldsymbol \zeta;\boldsymbol r) = \sum_{(k_1,k_2)\in\mathbb{Z}^2\setminus S_0} \hat{\tilde u}(k_1,k_2;\boldsymbol r) \exp \left\{ \frac{2\pi i}{\gamma} \left( k_1\zeta_1 + k_2\zeta_2 \right) \right\},\\
    &\zeta_j = m_jx+n_jy+\omega_j(\boldsymbol r)t+d_j, \quad j=1,2,\quad \boldsymbol\zeta = (\zeta_1,\zeta_2).
\end{align*}
Moreover, for some $\sigma>0$,
\begin{align*}
    \sum_{ (k_1,k_2)\notin (S_0\cup \{ \pm(1,0), \pm(0,1) \}) } &\left| \hat{\tilde u}(k_1,k_2;\boldsymbol r) - \hat{\tilde u}_0(k_1,k_2;\boldsymbol r) \right|^2 \exp \left( \frac{4\pi\sigma}{\gamma}|(k_1,k_2)|_1 \right) \lesssim |\boldsymbol r|_\infty^{2p},\\
    \hat{\tilde u}(\pm1,0;\boldsymbol r) &= \hat{\tilde u}_0(\pm1,0;\boldsymbol r),\quad
    \hat{\tilde u}(0,\pm1;\boldsymbol r) = \hat{\tilde u}_0(0,\pm1;\boldsymbol r).
\end{align*}
The perturbed frequencies are real analytic and satisfy
\begin{equation}
    \omega_j(\boldsymbol r)-\omega_{j,0}(\boldsymbol r)
    =
    \begin{cases}
        O(\boldsymbol r^{p-1}), & p \text{ odd},
        \\[2mm]
        O(\boldsymbol r^{p}), & p \text{ even}.
    \end{cases}
    \label{eq:intro-two-gap-frequency}
\end{equation}
The leading-order terms in \eqref{eq:intro-two-gap-frequency} are computed explicitly in Subsection~\ref{sec:two-gap-2} and Appendix \ref{Appendix}.
\end{theorem}
\begin{remark}
The restriction $\boldsymbol r\in(\rho'/2,\rho')^2$ in Theorem~\ref{thm:intro-two-gap} is only used to keep the two amplitude parameters comparable and can be removed by keeping track of the individual powers of $r_1$ and $r_2$.
With Lemma \ref{construction2}, we introduce the weighted spaces $H_{\mu}^{\sigma',s_1,s_2}$ with norm
\[
    \|\hat{u}\|_{\sigma',s_1,s_2,\mu}^2 = \sum\limits_{\boldsymbol k \in\mathbb{Z}^2} \left|\frac{\hat{u}(\boldsymbol k;\boldsymbol r)}{r_{\mu}(\boldsymbol k;\boldsymbol r)}\right|^2e^{2\sigma'\nu_2|(m(\boldsymbol k),n(\boldsymbol k))|_1}\langle \nu_2m(\boldsymbol k)\rangle^{2s_1}\langle \nu_2n(\boldsymbol k)\rangle^{2s_2},
\]
where $r_{\mu}(\boldsymbol k;\boldsymbol r) = r_1^{|k_1|}r_2^{|k_2|}|\boldsymbol r|_{\infty}^{\mu-|\boldsymbol k|_1}$. It is easy to obtain the following algebra property:
\[
    \|\hat{u}_1*\hat{u}_2\|_{\sigma',s_1,s_2,\mu+\nu} \leq C_{\sigma',s_1,s_2}\|\hat{u}_1\|_{\sigma',s_1,s_2,\mu}\|\hat{u}_2\|_{\sigma',s_1,s_2,\nu}.
\]
Moreover, $\|\hat{u}\|_{\sigma',s_1,s_2,\mu'} = |\boldsymbol r|_{\infty}^{\mu-\mu'} \|\hat{u}\|_{\sigma',s_1,s_2,\mu}$ for $\mu > \mu'$, whereas the small-divisor estimates for $L_0^{-1}$ preserve the
index $\mu$. Hence the range-equation estimates and the contraction argument of Section~4 carry over to these weighted spaces.
In particular,
\[
\widehat{P_2(u_0+v)}(1,0) = r_1\cdot O\!\left(\boldsymbol r^{p-1}\right), \qquad
\widehat{P_2(u_0+v)}(0,1) = r_2\cdot O\!\left(\boldsymbol r^{p-1}\right).
\]
Since $|\widehat u_0(1,0;\boldsymbol r)|\sim r_1,\ |\widehat u_0(0,1;\boldsymbol r)|\sim r_2$, division by the corresponding fundamental Fourier coefficients in
the bifurcation equations gives $O(\boldsymbol r^{p-1})$ uniformly for $\boldsymbol r \in (0,\rho')^2$.
\end{remark}
To the best of our knowledge, these results provide the first
persistence theorem for small-amplitude one-gap and two-gap solutions
of the KP equation under the class of Hamiltonian differential perturbations considered here.

A key point in solving the range equation is that, after the resonant modes have been separated,
the range equation does not exhibit an accumulating small-divisor
problem. Indeed, the specific structure of the KP dispersion relation
provides a coercive lower bound for the range divisors at high Fourier
modes, while the remaining finitely many modes are controlled by the
nonresonance assumptions in the statements (see also \cite{yuan2003}). Consequently, the range
equation can be solved directly by a contraction argument, without
any Cantor-type exclusion of the internal parameters. The persisting
solutions are therefore obtained for every sufficiently small
parameter in the prescribed domains.

\textbf{Notations.}

$a\lesssim b$ refers to $a\leq Cb$, for some constant which can be chosen depending on the context. $a\sim b$ refers to $b\lesssim a\lesssim b$.

For $x\in\mathbb{R}^d$, denote $|x|_1=\sum\limits_{j=1}^d|x_j|,|x|_{\infty} = \max\limits_{1\leq j\leq d} |x_j|, \langle x\rangle=\sqrt{1+|x|_1^2}$.

For a $d\times d$ matrix $A$, denote
$\|A\|_1=\max\limits_{1\leq j\leq d}\left(\sum\limits_{i=1}^d |A_{ij}|\right),\|A\|_{max}=\max\limits_{1\leq i,j\leq d} |A_{ij}|$.

Denote $\mathbb{Z}_+ = \{1,2,\cdots\}, \mathbb{N} = \{0,1,2,\cdots\}$. For $g\in \mathbb{Z}_+$, denote $\mathcal{G} = \{1,2,\cdots,g\}$.

For a fractional linear transformation $\sigma$, we say the matrix
$M_{\sigma}=\begin{pmatrix}
        \sigma_{11} &\sigma_{12}\\
        \sigma_{21} &\sigma_{22}
    \end{pmatrix}$
corresponds to $\sigma$ if $\sigma(z)=\dfrac{\sigma_{11}z+\sigma_{12}}{\sigma_{21}z+\sigma_{22}}$ and $\det M_{\sigma}=1$.

Let \(X\subset E\times F\). For each \(x\in E\), we denote the section of
\(X\) at \(x\) by $X(x):=\{y\in F:(x,y)\in X\}$.
Similarly, for each \(y\in F\), we set $X(y):=\{x\in E:(x,y)\in X\}$.

When $C_j$ and $C_j'$ denote a collection of circles in the complex plane, $\mathring{C_j}$ and $\mathring{C_j'}$ are understood to represent their respective open interiors, while their closures $\overline{C}_{j}$ and $\overline{C_j'}$ represent the corresponding closed disks.

A map $f=(f_1,\cdots,f_N): \mathcal{U}\subset\mathbb{R}^d \to \mathbb{R}^N$ is said to be real analytic on $\mathcal U$ if each component $f_j$ admits a holomorphic extension to a complex neighborhood $\tilde{\mathcal{U}} \subset \mathbb{C}^d$ that contains $\mathcal{U}$.

Set $\mathbb{T}_{\gamma}^d := (\mathbb{R}/\gamma\mathbb{Z})^d,\ \gamma>0,\ d\in\mathbb{Z}_+$. Denote $\hat{u}(\boldsymbol k) = \frac{1}{\gamma^d} \int_{\mathbb{T}_{\gamma}^d} u(\boldsymbol\zeta)e^{-\frac{2\pi i}{\gamma}\boldsymbol k\cdot\boldsymbol \zeta} d\boldsymbol\zeta$ and $\mathcal{F}u = \hat{u}$.

For Banach spaces $X,Y$, we denote by $\mathcal{L}(X,Y)$ the Banach space of bounded linear operators from $X$ to $Y$.
For a map $F$ from $X$ to $Y$, we denote its first and second Fr\'echet derivatives by $DF(x)[h],\ D^2F(x)[h_1,h_2]$, respectively.

Let $\mathcal{U}$ be a region in $\mathbb{R}^d$ or $\mathbb{C}^d$. For $k\in\mathbb{Z}_+,\ \boldsymbol r  \in \mathcal{U}$, we write the vector-valued function $f(\boldsymbol r) = O(\boldsymbol r^k)$ if $|f(\boldsymbol r)|_{\infty}\leq C|\boldsymbol r|_{\infty}^k$ for some constant $C>0$ independent of $\boldsymbol r \in \mathcal{U}$.

\textbf{Description of the paper.}
In Section~\ref{sec:parameterization}, we recall the theta-functional representation of finite-gap solutions of the KP equation and develop their small-amplitude parameterization by Schottky uniformization. We establish the real-analytic dependence on the degeneration parameters and prescribe the spatial wavenumber vectors by an implicit-function argument. In Section~\ref{sec:one-gap}, we prove the persistence of one-gap solutions for the KP-I and KP-II equations. We solve the range equation by a contraction argument and determine the perturbed traveling-wave frequency from the bifurcation equation. In Section~\ref{sec:two-gap}, we consider two-gap solutions of the KP-I equation. We analyze the resonant set and the anisotropic linear divisor, solve the infinite-dimensional range equation and the two-dimensional bifurcation system, and finally derive the leading asymptotic corrections to the two temporal frequencies.

\section{Parameterization of small-amplitude multi-periodic finite-gap solutions}\label{sec:parameterization}

In this section, we develop a parameterization of small-amplitude multi-periodic finite-gap solutions with prescribed spatial wavenumbers.
\subsection{Basic facts about finite-gap solutions of the KP equation}
For the reader's convenience, we briefly recall some basic facts about finite-gap solutions of the KP equation. For a more detailed account, we refer to \cite{belokolos1994algebro}.

Let \(X\) be a compact Riemann surface of genus \(g\), and set $\mathcal{G} = \{1, \dots, g\}$. Choose a canonical homology basis \(\{a_j,b_j\}_{j=1}^g\) and normalized holomorphic differentials \(d\omega_1,\cdots,d\omega_g\) satisfying
\[
\oint_{a_k}d\omega_j=2\pi i\delta_{jk},\qquad j,k\in\mathcal{G}.
\]
Define the period matrix $B = (B_{jk})_{j,k\in\mathcal{G}}$ by $B_{jk}=\oint_{b_k}d\omega_j$.
Then \(B\) is symmetric and its real part is negative definite.

Fix a point \(P_\infty\in X\), and let $k^{-1}$ be a local parameter near
\(P_\infty\), with \(k\to\infty\) as \(P\to P_\infty\). 
Let \(d\Omega_1,d\Omega_2,d\Omega_3\) be the normalized Abelian differentials of the second kind, each having its only pole at \(P_\infty\). They are characterized by the normalization conditions
\begin{equation*}
    \oint_{a_j}d\Omega_i=0,\qquad i=1,2,3,\ j\in \mathcal{G},
\end{equation*}
together with the asymptotic behavior of their Abelian integrals
\begin{equation*}
    \Omega_i(P)=k^{i}+o(1),\qquad i=1,2,3,\qquad P\to P_\infty.
\end{equation*}
For $j\in\mathcal{G}$, set
\[
U_j=\oint_{b_j}d\Omega_1,\qquad
V_j=\oint_{b_j}d\Omega_2,\qquad
W_j=\oint_{b_j}d\Omega_3.
\]

\begin{lemma}[Its--Matveev formula]
Let \(X\) be a compact Riemann surface of genus \(g\), let \(B\) be its period matrix, and let
\[
\boldsymbol U=(U_1,\cdots,U_g),\qquad
\boldsymbol V=(V_1,\cdots,V_g),\qquad
\boldsymbol W=(W_1,\cdots,W_g)
\]
be the vectors defined above.  Choose an arbitrary phase vector $\boldsymbol D\in \mathbb C^g$.
Then the KP-II equation
\begin{equation}
    \frac{\partial}{\partial x}
    \left(
u_t-\frac{1}{4}\bigl(6uu_x+u_{xxx}\bigr)
\right)
    -\frac{3}{4}u_{yy}=0
    \label{transformedKP2}
\end{equation}
admits the theta-functional expression
\begin{equation*}
    u(t,x,y)
    =
    2\frac{\partial^2}{\partial x^2}
    \log \theta\left(
    \boldsymbol U x+\boldsymbol V y+\boldsymbol W t+\boldsymbol D;B
    \right)
    +2c,
\end{equation*}
where \(c\) is the constant determined by the local expansion
\begin{equation*}
    \Omega_1(P)=k-\frac{c}{k}+O(k^{-2}),
    \quad P\to P_\infty.
\end{equation*}
Here, the Riemann theta function associated with the period matrix \(B\) is
defined by
\begin{equation*}
    \theta(\boldsymbol z;B)
    =
    \sum_{\boldsymbol n\in\mathbb Z^g}
    \exp\left(
    \frac{1}{2}\langle B \boldsymbol n,\boldsymbol n\rangle+\langle \boldsymbol z,\boldsymbol n\rangle
    \right),
    \qquad \boldsymbol z\in\mathbb C^g,
\end{equation*}
where \(\langle\cdot,\cdot\rangle\) denotes the standard bilinear form.
\end{lemma}

\begin{remark}\label{transformKP12}
    Let $u(t,x,y)$ be a solution of the KP-II equation \eqref{transformedKP2} that admits a holomorphic continuation to the required complex arguments.
    Then
    \begin{align*}
        \tilde{u}(t,x,y) = -u(it,-ix,iy)
    \end{align*}
    is a solution of the KP-I equation
    \begin{equation}\label{transformedKP1}
        \frac{\partial}{\partial x}
        \bigg(
        \tilde{u}_t-\frac{1}{4}(6\tilde{u}\tilde{u}_x+\tilde{u}_{xxx})
        \bigg)
        +\frac{3}{4}\tilde{u}_{yy}=0.
    \end{equation}
\end{remark}

\begin{corollary}\label{cor:finite-gap-KP}
Let \(X\) be a compact Riemann surface of genus \(g\), let \(B\) be its period matrix, and let
\[
\boldsymbol U=(U_1,\cdots,U_g),\qquad
\boldsymbol V=(V_1,\cdots,V_g),\qquad
\boldsymbol W=(W_1,\cdots,W_g)
\]
be the vectors defined above.  Choose an arbitrary phase vector $\boldsymbol D\in \mathbb C^g$.

The equation \eqref{kp} with $\lambda=-1$ admits the following theta-functional expression
\begin{equation*}
    \tilde u(t,x,y)
    =
    12\frac{\partial^2}{\partial x^2}
    \log\theta\left(
        \boldsymbol Ux+\sqrt3\boldsymbol Vy
        +\boldsymbol W't+\boldsymbol D;B
    \right),
\end{equation*}
where $\boldsymbol W'=-4\boldsymbol W+12c\,\boldsymbol U$.

Likewise, the equation \eqref{kp} with $\lambda = 1$ admits the following theta-functional expression
\begin{equation*}
    \tilde u(t,x,y)
    =
    12\frac{\partial^2}{\partial x^2}
    \log\theta\left(
        i\bigl(-\boldsymbol Ux+\sqrt3\boldsymbol Vy
        +\boldsymbol W't\bigr)+\boldsymbol D;B
    \right).
\end{equation*}
\end{corollary}

\subsection{Parameterization via the classical Schottky uniformization}

In the KdV reduction, in addition to the normalization $\oint_{a_k} d\Omega_2 = 0$, one has $\boldsymbol V = \boldsymbol 0$. Hence $\Omega_2$ is a single-valued meromorphic function with a unique pole of second order at the point $P_{\infty}$, which implies that the corresponding spectral curve $X$ is hyperelliptic.
However, this reduction does not hold for the KP case since $\boldsymbol V\neq 0$.
In what follows, we use classical Schottky uniformization to parameterize a class of finite-gap solutions and extract the internal parameters needed in the Lyapunov-Schmidt scheme for the KP
equations.

We first recall the basic construction of the classical Schottky uniformization.
Let \(F\) be a domain in \(\mathbb{CP}^{1}\) bounded by \(2g\) mutually disjoint circles \(C_1,C_1',\cdots,C_g,C_g'\). For each \(j=1,\cdots,g\), let \(\sigma_{j}\) be a fractional linear transformation which maps the exterior of \(C_{j}\) onto the interior of \(C_{j}^{\prime }\), with \(\sigma_j(C_j)=C_j'\). The group $G=\langle \sigma_1,\cdots,\sigma_g\rangle$ is called a classical Schottky group. Its limit set \(\Lambda(G)\) is the closure of the fixed points of the nontrivial elements of \(G\), and its domain of discontinuity is $\Omega(G)=\mathbb{CP}^1\setminus\Lambda(G).$
Then, the quotient space $X=\Omega(G)/G$ is a compact Riemann surface of genus \(g\).

Note that the fractional linear transformation \(\sigma _{j}\) can be expressed in terms of its two fixed points \(A_j, B_j\) and its multiplier \(\mu _{j}\), where $A_j\neq B_j$ and \(0 < \vert{}\mu_j\vert{} < 1\), as
\begin{equation}
    \frac{\sigma_j(z)-B_j}{\sigma_j(z)-A_j}
    =
    \mu_j\frac{z-B_j}{z-A_j},\textup{ for }z\in\mathbb{CP}^1.\label{definitionoffractions}
\end{equation}
Thus, we aim to find internal parameters through fixed points $A_j,B_j$ and multiplier $\mu_j$.

Given $\mu_{j}, A_{j}, B_{j}$, let $\sigma_{j}$ be the fractional linear transformation defined in \eqref{definitionoffractions}. We can rewrite $\sigma_j$ into
\begin{equation*}
    \sigma_j(z)=\frac{\sigma_{j,11}z+\sigma_{j,12}}{\sigma_{j,21}z+\sigma_{j,22}},\textup{ for }z\in\mathbb{CP}^1
\end{equation*}
where 
\begin{equation*}
   \begin{pmatrix}
        \sigma_{j,11}&\sigma_{j,12}\\
        \sigma_{j,21}&\sigma_{j,22}
    \end{pmatrix}
    =
    \frac{1}{A_j-B_j}
    \begin{pmatrix}
        A_j\sqrt{\mu_j}-B_j/\sqrt{\mu_j}
        &
        A_jB_j(1/\sqrt{\mu_j}-\sqrt{\mu_j})\\
        \sqrt{\mu_j}-1/\sqrt{\mu_j}
        &
        A_j/\sqrt{\mu_j}-B_j\sqrt{\mu_j}
    \end{pmatrix},
\end{equation*}
where a square root of \(\mu_j\) has been chosen. The other choice
multiplies the entire matrix by \(-1\) and therefore determines the
same fractional linear transformation.

For computational convenience, let $C_j,C_j'$ be the isometric circles of $\sigma_j$ and $\sigma_j^{-1}$, respectively.
Recall that the isometric circle of a fractional linear transformation $f$ is the set on which $|f'(z)|=1$.
By calculation,  the centers of \(C_{j}\) and \(C_{j}'\) are
\[
    -\frac{\sigma_{j,22}}{\sigma_{j,21}}
    =
    \frac{B_j\sqrt{\mu_j}-A_j/\sqrt{\mu_j}}
         {\sqrt{\mu_j}-1/\sqrt{\mu_j}},
\]
and
\[
    \frac{\sigma_{j,11}}{\sigma_{j,21}}
    =
    \frac{A_j\sqrt{\mu_j}-B_j/\sqrt{\mu_j}}
         {\sqrt{\mu_j}-1/\sqrt{\mu_j}},
\]
respectively. Moreover, their common radius is
\begin{equation}\label{radius}
    \frac{1}{|\sigma_{j,21}|}
    =
    \left|
    \frac{\sqrt{\mu_j}(A_j-B_j)}{1-\mu_j}
    \right|.
\end{equation}
We see that the radii of \(C_j\) and \(C_j'\) tend to zero, and the circles
\(C_j\), \(C_j'\) shrink to points \(A_j\), \(B_j\) as \(|\mu_j|\to0\), respectively. 
Thus, we can ensure $C_1,C_1',\cdots,C_g,C_g'$ mutually disjoint if $\max\limits_{j\in \mathcal{G}}|\mu_j|$ is sufficiently small and $A_1,B_1,\cdots,A_g,B_g$ are distinct from each other.
We choose the canonical basis $a_j$ to coincide with $C_j'$, oriented in a positive direction. And we choose $b_j$ to run on $F$ between the points $z_j\in C_j$ and $\sigma_j(z_j)\in C_j'$, such that $b_j$ do not mutually intersect. Let $P_{\infty} = \infty$ and the local parameter $k^{-1} = z^{-1}$. Then, the normalized Abelian differentials of the second kind $d\Omega_1, d\Omega_2, d\Omega_3$ are uniquely determined.

Now we are in the position to describe the corresponding period matrix $B$, the vectors $\boldsymbol U, \boldsymbol V, \boldsymbol W$, and the constant $c$.

For each $j \in \mathcal{G}$, denote by $G_j = \langle \sigma_j \rangle$ the cyclic group generated by $\sigma_j$.
Notice that every non-identity element $\sigma\in G$ admits a unique representation $\sigma=\sigma_{i_1}^{l_1}\cdots\sigma_{i_n}^{l_n}\in G$, where $i_1, \cdots, i_n\in\mathcal{G}, l_1, \cdots, l_n\in \{\pm 1\}$, and $l_k=l_{k+1}$ when $i_k=i_{k+1}$. In the sequel, we call such an expression a reduced word of length $n$. We call $n$ the length of $\sigma$ and denote it by $\textup{len}(\sigma)$. We also set $\textup{len}(id)=0$.

We use the following systems of reduced-word representatives for the right, left, and double cosets:
\begin{align*}
    G / G_j := \{ id \} \cup \Big\{& \sigma_{i_1}^{l_1} \cdots \sigma_{i_k}^{l_k} \in G \;\Big|\; 
\sigma_{i_1}^{l_1} \cdots \sigma_{i_k}^{l_k} \text{ is a reduced word with } \\
&\sigma_{i_k} \neq \sigma_j, i_1, \dots, i_k \in \mathcal{G},\; l_1, \dots, l_k \in \{\pm 1\}, k\in\mathbb{Z}_{+} \Big\},\\
    G_i \setminus G := \{ id \} \cup \Big\{& \sigma_{i_1}^{l_1} \cdots \sigma_{i_k}^{l_k} \in G \;\Big|\; 
\sigma_{i_1}^{l_1} \cdots \sigma_{i_k}^{l_k} \text{ is a reduced word with } \\
&\sigma_{i_1} \neq \sigma_i, i_1, \dots, i_k \in \mathcal{G},\; l_1, \dots, l_k \in \{\pm 1\}, k\in\mathbb{Z}_{+} \Big\},\\
    G_j \setminus G / G_i := \{ id \} \cup \Big\{& \sigma_{i_1}^{l_1} \cdots \sigma_{i_k}^{l_k} \in G \;\Big|\; 
\sigma_{i_1}^{l_1} \cdots \sigma_{i_k}^{l_k} \text{ is a reduced word with } \\
&\sigma_{i_1} \neq \sigma_j, \sigma_{i_k} \neq \sigma_i, i_1, \dots, i_k \in \mathcal{G},\; l_1, \dots, l_k \in \{\pm 1\}, k\in\mathbb{Z}_{+} \Big\}.
\end{align*}

\begin{lemma}\label{shoulianyinli}
    If for every $j$, the series
    \begin{equation}\label{differential1}
    \sum_{\sigma\in G/G_j}\bigg(\frac{1}{z-\sigma (B_j)}-\frac{1}{z-\sigma (A_j)}\bigg) dz
    \end{equation}
    converges absolutely, then it defines a holomorphic differential $d\omega_j(z)$, and the period matrix is given by
    \begin{align}
    &B_{ij}=\log\{B_j,A_j, B_i, A_i\} + \sum_{\sigma\in G_j\setminus G/G_i,\sigma\neq id}\log\{B_j,A_j,\sigma (B_i),\sigma (A_i)\}, i\neq j, \label{pmatrix1}\\
    &B_{jj}=\log\mu_j+\sum_{\sigma\in G_j\setminus G/G_j,\sigma\neq id}\log\{B_j,A_j,\sigma (B_j),\sigma (A_j)\},\label{pmatrix2}
    \end{align}
    where
    \[
        \{a,b,c,d\} := \frac{(a-c)(b-d)}{(a-d)(b-c)}.
    \]
    Moreover, we have 
    \begin{align}
        &U_j=\sum_{\sigma\in G/G_j}\sigma(A_j)-\sigma(B_j),\label{parameter1}\\
        &V_j=\sum_{\sigma\in G/G_j}\sigma(A_j)^2-\sigma(B_j)^2,\label{parameter2}\\
        &W_j=\sum_{\sigma\in G/G_j}\sigma(A_j)^3-\sigma(B_j)^3,\notag\\
        &c=\sum\limits_{\sigma\in G,\sigma\neq id} \sigma_{21}^{-2}\notag.
    \end{align}
\end{lemma}
(See \cite[Section 5.1,5.5]{belokolos1994algebro} for the proof.)

\bigskip

In order to circumvent issues related to the choice of branches of the square root, we regard $r_j$ as the basic parameter and set $\mu_j = r_j^2$.

To extract parameters, we shall apply the implicit function theorem to  equations involving \(\boldsymbol U\) and \(\boldsymbol V\). For this purpose, the functions in \eqref{parameter1} and \eqref{parameter2} need to be extended so as to be well-defined when some \(r_j\) vanish.
To extend the definition of the series, we first extend the generators $\sigma_{j}^{\pm 1}$ to the degenerate case.

Denote
\begin{align*}
    \Xi_{j,+}(z)=\begin{cases}
        A_j,\textup{ for }z=A_j,\\
        B_j,\textup{ otherwise }.
    \end{cases}\quad
    \Xi_{j,-}(z)=\begin{cases}
        B_j,\textup{ for }z=B_j,\\
        A_j,\textup{ otherwise},
    \end{cases}
\end{align*}
When $r_{j}=0$, define $\sigma_j(z) = \Xi_{j,+}(z), \sigma_j^{-1}(z) = \Xi_{j,-}(z)$. 
Note that the definitions of \(\sigma_j\) and \(\sigma_j^{-1}\) at \(r_j=0\) agree with the pointwise limits obtained from the case \(r_j\neq 0\). More precisely, as \(r_j\to 0\), one has
\[
\sigma_j(z) \to \Xi_{j,+}(z),
\qquad
\sigma_j^{-1}(z)\to \Xi_{j,-}(z),
\]
for all $z\in \mathbb{CP}^{1}$. Note that when $r_{j}=0$, $\sigma_{j}\sigma_{j}^{-1}\ne id$.
Allowing \(r_j=0\), compositions of the extended generators form a semigroup, which we denote by \(\widetilde G\). Notice that the symbols \(\sigma_j^{-1}\) are retained for notational convenience; when \(r_j=0\), they no longer denote the inverse maps of \(\sigma_j\).
When some of the parameters \(r_j\) vanish, the symbols
\[
    \widetilde G/\widetilde G_j,\qquad
    \widetilde G_i\setminus\widetilde G,\qquad
    \widetilde G_j\setminus\widetilde G/\widetilde G_i
\]
are not understood as coset spaces. More precisely, starting from the
fixed systems of reduced-word representatives of $G/G_j,\ G_i\setminus G,\ G_j\setminus G/G_i,$
we replace each generator by its extended counterpart and compose the
resulting maps in the same order. The above symbols denote the
corresponding families of composed maps. Each member of these families
is understood together with its inherited reduced-word representation,
so that its length and its first and last generators remain
well-defined in the degenerate case.
By abuse of notation, we retain the symbols $C_j$ and $C_j'$ for the degenerate limits of the corresponding isometric circles, namely $C_j=\{A_j\},\ C_j'=\{B_j\}$.

Now we are in the position to study the series
\begin{align}
    &f_{1,j}(\boldsymbol A,\boldsymbol B,\boldsymbol r)=\sum_{\sigma\in \tilde{G}/\tilde{G}_j} \sigma(A_j)-\sigma(B_j),\label{parameter1'}\\
    &f_{2,j}(\boldsymbol A,\boldsymbol B,\boldsymbol r)=\sum_{\sigma\in \tilde{G}/\tilde{G}_j}\sigma(A_j)^2-\sigma(B_j)^2,\label{parameter2'}\\
    &f_{3,j}(\boldsymbol A,\boldsymbol B,\boldsymbol r)=\sum_{\sigma\in \tilde{G}/\tilde{G}_j}\sigma(A_j)^3-\sigma(B_j)^3,\label{parameter3'}\\
    &f_{4,i,j}(\boldsymbol A,\boldsymbol B,\boldsymbol r)=\log\{B_j,A_j, B_i, A_i\} + \sum_{\sigma\in G_j\setminus G/G_i,\sigma\neq id}\log\{B_j,A_j,\sigma (B_i),\sigma (A_i)\},i\neq j,\label{parameter4'}\\
    &f_{5,j}(\boldsymbol A,\boldsymbol B,\boldsymbol r)=\log r_j^2+\sum_{\sigma\in G_j\setminus G/G_j,\sigma\neq id}\log\{B_j,A_j,\sigma (B_j),\sigma (A_j)\}.\label{parameter5'}
\end{align}

We now study the dependence of
\[
\sigma(A_j)-\sigma(B_j),\qquad
\sigma(A_j)^2-\sigma(B_j)^2,\qquad
\sigma(A_j)^3-\sigma(B_j)^3,
\]
on the parameters \(\boldsymbol r\), \(\boldsymbol A\), and \(\boldsymbol B\), and in particular establish their analyticity.

Recall that, for \(0<|r_j|<1\), we have
\begin{equation} 
\begin{pmatrix} 
\sigma_{j,11}&\sigma_{j,12}\\ 
\sigma_{j,21}&\sigma_{j,22} \end{pmatrix} 
= \frac{1}{A_j-B_j} 
\begin{pmatrix} 
A_j r_{j}-B_j/r_{j} & A_jB_j(1/r_{j}-r_{j})
\\ r_{j}-1/r_{j} & A_j/r_{j}-B_jr_{j} 
\end{pmatrix}, 
\label{jiexixingpanduan}
\end{equation}
When \(r_j=0\), by convention mentioned above, we have
\begin{align*} 
\sigma_j(z)=\begin{cases} A_j,\textup{ for }z=A_j,\\ 
B_j,\textup{ otherwise }, 
\end{cases} 
\end{align*} 
and
\begin{align*} 
\sigma_j^{-1}(z)=
\begin{cases} B_j,\textup{ for }z=B_j,\\ 
A_j,\textup{ otherwise}. 
\end{cases} 
\end{align*}

\begin{lemma}\label{lm6.4}
    Let \(\mathcal U\subset\mathbb C^{3g}\) be an open set such that the sets $C_1,C_1',\cdots,C_g,C_g'$ are pairwise disjoint for every \((\boldsymbol A,\boldsymbol B,\boldsymbol r)\in\mathcal U\).
    Then, for any $j\in\mathcal{G}$ and $\sigma\in \tilde{G}/\tilde{G}_{j}$, the maps 
    \begin{align*}
        &(\boldsymbol{A},\boldsymbol{B}, \boldsymbol{r})\mapsto \sigma(A_{j})-\sigma(B_{j}),\\
        &(\boldsymbol{A},\boldsymbol{B}, \boldsymbol{r}) \mapsto \sigma(A_{j})^{2}-\sigma(B_{j})^{2},\\
        &(\boldsymbol{A},\boldsymbol{B}, \boldsymbol{r})\mapsto \sigma(A_{j})^{3}-\sigma(B_{j})^{3},
    \end{align*}
    are holomorphic in $\mathcal{U}$.
\end{lemma}
\begin{proof}
Let \(\sigma\in \widetilde{G}/\widetilde{G}_{j}\). We first prove the
analyticity of $\sigma(A_{j})-\sigma(B_{j})$.

Fix \(\boldsymbol A,\boldsymbol B\) and
\[
    \hat{\boldsymbol r}_{i}
    =(r_{1},\cdots,r_{i-1},r_{i+1},\cdots,r_{g}).
\]
We consider the dependence on \(r_i\). If the reduced representation of
\(\sigma\) contains neither \(\sigma_i\) nor \(\sigma_i^{-1}\), then
\(\sigma(A_j)-\sigma(B_j)\) is independent of \(r_i\), and there is nothing
to prove. Otherwise, we may write
\[
    \sigma=\widetilde{\sigma}\sigma_i^{\nu}\widehat{\sigma},
    \qquad \nu\in\{\pm1\},
\]
where \(\widehat{\sigma}\in \widetilde{G}/\widetilde{G}_{j}\) contains
neither \(\sigma_i\) nor \(\sigma_i^{-1}\).

We first assume that \(r_l\ne0\) for all \(l\ne i\). 

By
\eqref{jiexixingpanduan} and the assumptions, the function
\(\sigma(A_j)-\sigma(B_j)\) is holomorphic in
\[
    r_i\in \mathcal U(\boldsymbol A,\boldsymbol B,\hat{\boldsymbol r}_i),
    \qquad r_i\ne0.
\]
Moreover, since \(\widehat{\sigma}\) contains neither \(\sigma_i\) nor
\(\sigma_i^{-1}\), we have
\[
    \widehat{\sigma}(A_j),\ \widehat{\sigma}(B_j)
    \in \mathbb{CP}^{1}\setminus (\overline{C_i}\cup \overline{C_i'}).
\]
Hence, by the definition of \(\sigma_i\) and \(\sigma_i^{-1}\) at
\(r_i=0\), the two points
\(\sigma_i^{\nu}\widehat{\sigma}(A_j)\) and
\(\sigma_i^{\nu}\widehat{\sigma}(B_j)\) have the same limit as
\(r_i\to0\). Therefore,
\[
    \sigma(A_j)-\sigma(B_j)\longrightarrow 0,
    \qquad r_i\to0.
\]
Since this function $\sigma(A_j)-\sigma(B_j)$ is holomorphic for \(r_i\neq0\) and is locally
bounded near \(r_i=0\), the Riemann removable singularity theorem shows that its value at \(r_i=0\), defined to be zero, gives a holomorphic extension with respect to
$r_i\in \mathcal U(\boldsymbol A,\boldsymbol B,\hat{\boldsymbol r}_i)$.

Next, assume that \(r_l=0\) for some \(l\ne i\). If the reduced
representation of \(\sigma\) contains neither \(\sigma_l\) nor
\(\sigma_l^{-1}\) for every such \(l\), then the previous case applies.
Otherwise, we may write
\[
    \sigma=\widetilde{\sigma}\sigma_l^{\nu}\widehat{\sigma},
    \qquad \nu\in\{\pm1\},
\]
where \(r_l=0\) and \(\widehat{\sigma}\) contains neither \(\sigma_l\) nor
\(\sigma_l^{-1}\). Then
\[
    \widehat{\sigma}(A_j),\ \widehat{\sigma}(B_j)
    \in \mathbb{CP}^{1}\setminus (\overline{C_l}\cup \overline{C_l'}).
\]
By the definition of \(\sigma_l\) and \(\sigma_l^{-1}\) at \(r_l=0\), this
implies
\[
    \sigma(A_j)-\sigma(B_j)=0.
\]
In particular, \(\sigma(A_j)-\sigma(B_j)\) is holomorphic with respect to
\(r_i\). 

We now fix
\[
    \boldsymbol A_{\hat{i}}=(A_1,\cdots,A_{i-1},A_{i+1},\cdots,A_g),
    \qquad \boldsymbol B,\qquad \boldsymbol r,
\]
and consider the dependence on \(A_i\). If \(r_l\ne0\) for all
\(l\in\mathcal G\), then the analyticity with respect to \(A_i\) follows
directly from \eqref{jiexixingpanduan} and the assumptions.

If \(r_l=0\) for some \(l\in\mathcal G\), then either the reduced
representation of \(\sigma\) contains no such generator, in which case the
previous argument applies, or we can write
\[
    \sigma=\widetilde{\sigma}\sigma_l^{\nu}\widehat{\sigma},
    \qquad \nu\in\{\pm1\},
\]
where \(\widehat{\sigma}\) contains neither \(\sigma_l\) nor
\(\sigma_l^{-1}\). As before,
\[
    \widehat{\sigma}(A_j),\ \widehat{\sigma}(B_j)
    \in \mathbb{CP}^{1}\setminus (\overline{C_l}\cup \overline{C_l'}),
\]
and hence
\[
    \sigma(A_j)-\sigma(B_j)=0.
\]
Therefore \(\sigma(A_j)-\sigma(B_j)\) is holomorphic with respect to \(A_i\).

The analyticity with respect to \(B_i\) is obtained in exactly the same
way. Since the function $\sigma(A_j)-\sigma(B_j)$ is separately holomorphic in all variables,
Hartogs' theorem implies that it is jointly holomorphic with respect to $(\boldsymbol A,\boldsymbol B,\boldsymbol r)\in\mathcal U$.

Finally, the same argument applies to
\[
    \sigma(A_j)^2-\sigma(B_j)^2
    \qquad\text{and}\qquad
    \sigma(A_j)^3-\sigma(B_j)^3.
\]
Indeed, away from the hypersurfaces \(r_l=0\), the claim follows from
\eqref{jiexixingpanduan}; while if some generator with \(r_l=0\) appears in
the reduced representation of \(\sigma\), the two values
\(\sigma(A_j)\) and \(\sigma(B_j)\) coincide by the above argument. This
completes the proof.
\end{proof}

We now discuss the convergence of the series in \eqref{differential1}, \eqref{parameter1'}, \eqref{parameter2'}, \eqref{parameter3'}, \eqref{parameter4'}, \eqref{parameter5'}.

\begin{lemma}\label{convergelm}
    Suppose $\mathcal{U}_{1,1}, \mathcal{U}_{2,1}, \cdots, \mathcal{U}_{1,g}, \mathcal{U}_{2,g}\subset \mathbb{C}$ are $2g$ pairwise disjoint domains such that
    \begin{align*}
        m=\inf\limits_{\substack{i,i'=1,j,j'\in\mathcal{G}\\ (i,j)\neq (i',j')}} \inf\limits_{\substack{z\in \mathcal{U}_{i,j},\\ w\in\mathcal{U}_{i',j'}}} |z-w|>0,\quad
        M=\sup\limits_{\substack{i,i'=1,j,j'\in\mathcal{G}\\ (i,j)\neq (i',j')}}  \sup\limits_{\substack{z\in \mathcal{U}_{i,j},\\ w\in\mathcal{U}_{i',j'}}}|z-w|<\infty.
    \end{align*}
    Let $\rho_0<\min\bigg\{\dfrac{1}{8},\dfrac{m^2}{16gM^2},\dfrac{m}{16(2g-1)M},\dfrac{1}{M}\bigg\}$.
    Denote
    \begin{align*}
        &\boldsymbol A = (A_1,\cdots,A_g),\qquad \boldsymbol B = (B_1,\cdots,B_g),\\
        &\mathcal{U}_1=\mathcal{U}_{1,1}\times\mathcal{U}_{1,2}\times\cdots\times \mathcal{U}_{1,g},\qquad\mathcal{U}_2=\mathcal{U}_{2,1}\times\mathcal{U}_{2,2}\times\cdots\times \mathcal{U}_{2,g},\\
        &\mathcal{N}_{\rho_0}=\{\boldsymbol r: |r_j|\leq \rho_0\textup{ for }j\in\mathcal{G}\},\qquad \mathcal{N}^{\circ}_{\rho_0} = \{\boldsymbol r: 0<|r_j|\leq \rho_0\textup{ for }j\in\mathcal{G}\}.
    \end{align*}
    Then, we have:
    \begin{enumerate}[$(i)$]
        \item For $(\boldsymbol A,\boldsymbol B,\boldsymbol r)\in\mathcal{U}_1 \times \mathcal{U}_2 \times \mathcal{N}^{\circ}_{\rho_0}, z\in F$, the series \eqref{differential1} converge absolutely, which implies that the holomorphic differential $d\omega_j(z)$ is well defined.
        \item For $(\boldsymbol A,\boldsymbol B,\boldsymbol r)\in\mathcal{U}_1 \times \mathcal{U}_2 \times \mathcal{N}_{\rho_0}$, the circles $C_j,C_j'$ together with their degenerate limits are disjoint from each other. The series \eqref{parameter1'}, \eqref{parameter2'}, \eqref{parameter3'} converge absolutely and uniformly with respect to $(\boldsymbol A,\boldsymbol B,\boldsymbol r)\in\mathcal{U}_1 \times \mathcal{U}_2 \times \mathcal{N}_{\rho_0}$.
        \item For any $i,j\in \mathcal{G}$, the series
        \[
            \sum_{\sigma\in G_j\setminus G/G_i,\sigma\neq id}\log\{B_j,A_j,\sigma (B_i),\sigma (A_i)\}
        \]
        converge absolutely and locally uniformly on $\mathcal{U}_1 \times \mathcal{U}_2 \times \mathcal{N}^{\circ}_{\rho_0}$.
        Moreover,
        \[
            f_{4,i,j}(\boldsymbol A,\boldsymbol B,\boldsymbol r)= \log\{B_j,A_j, B_i, A_i\} + O(\boldsymbol r^2)
        \]
        for all $i,j\in\mathcal{G},i\neq j$, and $f_{5,j}(\boldsymbol A,\boldsymbol B,\boldsymbol r)=\log r_j^2+O(\boldsymbol r^2)$ for all $j\in\mathcal{G}$.
    \end{enumerate}

\end{lemma}
\smallskip
\begin{proof}
    $(i)$ Given a sequence of reduced words $\{\sigma^{(n)}\}_{n=1}^{\infty}$, we now estimate $\bigg|\dfrac{\sigma_{n+1,21}}{\sigma_{n,21}}\bigg|$ for $n\ge 1$.   Denote $\sigma^{(0)}=id$.
Suppose
\[
\sigma^{(n+1)}=\sigma^{(n)}\sigma_j^{l_j},\quad l_j=\pm1.
\]
Then, we have
\[
M_{\sigma^{(n+1)}}=M_{\sigma^{(n)}} M_{\sigma_j}^{l_j}.
\]
A simple calculation yields
\begin{equation*}
    \frac{\sigma_{n+1,21}}{\sigma_{n,21}}=\dfrac{1}{r_j}\frac{B_j+\sigma_{n,22}/\sigma_{n,21}}{B_j-A_j}-r_j\frac{A_j+\sigma_{n,22}/\sigma_{n,21}}{B_j-A_j},
\end{equation*}
when $l_j=1$,
and
\begin{equation*}
    \frac{\sigma_{n+1,21}}{\sigma_{n,21}}=\dfrac{1}{r_j}\frac{A_j+\sigma_{n,22}/\sigma_{n,21}}{A_j-B_j}-r_j\frac{B_j+\sigma_{n,22}/\sigma_{n,21}}{A_j-B_j},
\end{equation*}
when $l_j=-1$.

Consider the case $l_{j}=1$.

Suppose $(\sigma^{(n)})^{-1}=\sigma_i^{-1}(\sigma^{(n-1)})^{-1}$. Then, $-\sigma_{n,22}/\sigma_{n,21}\in \mathring{C}_i$ since $(\sigma^{(n)})^{-1}$ maps $\infty$ to $-\sigma_{n,22}/\sigma_{n,21}$.

If $i=j$, with \eqref{radius}, we have
\begin{align}
    \bigg|\frac{\sigma_{n+1,21}}{\sigma_{n,21}}\bigg|
    &=\bigg|\bigg(\dfrac{1}{r_j}-r_j\bigg)\frac{B_j+\sigma_{n,22}/\sigma_{n,21}}{B_j-A_j}+r_j\bigg|\notag\\
    &>\bigg|\bigg(\dfrac{1}{r_j}-r_j\bigg)\bigg(1-\bigg|\frac{A_{j}+\sigma_{n,22}/\sigma_{n,21}}{B_j-A_j}\bigg|\bigg)\bigg|-\left|r_j\right|\notag\\
    &>\bigg|\bigg(\dfrac{1}{r_j}-r_j\bigg)\bigg(1-\bigg|\frac{2r_j}{1-r_j^2}\bigg|\bigg)\bigg|-\left|r_j\right|\label{yuandetaolun}\\
    &>\frac{1}{8|r_j|},\label{yuandetaolun1}
\end{align}
for $|r_j|<\dfrac{1}{8}$,
where the inequality \eqref{yuandetaolun} uses the facts that $A_{j}, -\sigma_{n,22}/\sigma_{n,21}\in \mathring{C}_i$.

If $i\neq j$, we have
\begin{align}
    \bigg|\frac{\sigma_{n+1,21}}{\sigma_{n,21}}\bigg|&=\bigg|\bigg(\dfrac{1}{r_j}-r_j\bigg)\frac{B_j+\sigma_{n,22}/\sigma_{n,21}}{B_j-A_j}+r_j\bigg|\notag\\
    &>\bigg|\bigg(\dfrac{1}{r_j}-r_j\bigg)\bigg(\bigg|\frac{B_j-A_i}{B_j-A_j}\bigg|-\bigg|\frac{2r_i(B_i-A_i)}{(1-r_i^2)(B_j-A_j)}\bigg|\bigg)\bigg|-|r_j|\notag\\
    &>\bigg(\frac{1}{|r_j|}-|r_j|\bigg)(\frac{m}{M}-\frac{4|r_i|M}{m})>\frac{c'}{|r_j|}\label{yuandetaolun2},
\end{align}
for $|r_i|,|r_j|< \dfrac{m^2}{16M^2}$, where $c'=\min \left\{\dfrac{m}{4M},\dfrac{1}{8} \right\}$.

Suppose $(\sigma^{(n)})^{-1}=\sigma_i(\sigma^{(n-1)})^{-1}$. Then, $-\sigma_{n,22}/\sigma_{n,21}\in \mathring{C}_i'$ since $(\sigma^{(n)})^{-1}$ maps $\infty$ to $-\sigma_{n,22}/\sigma_{n,21}$. 
The requirement for $\sigma^{(n+1)}$ to be of length $n+1$ implies that $i\neq j$. Thus,
\begin{align}
    \bigg|\frac{\sigma_{n+1,21}}{\sigma_{n,21}}\bigg|
    &=\bigg|\dfrac{1}{r_j}\bigg(\frac{B_j-B_i}{B_j-A_j}+\frac{B_i+\sigma_{n,22}/\sigma_{n,21}}{B_j-A_j}\bigg)-r_j\bigg(\frac{A_j-B_i}{B_j-A_j}+\frac{B_i+\sigma_{n,22}/\sigma_{n,21}}{B_j-A_j}\bigg)\bigg| \notag\\
    &>\dfrac{1}{|r_j|}\bigg(\bigg|\frac{B_j-B_i}{B_j-A_j}\bigg|-\bigg|\frac{2r_i(B_i-A_i)}{(1-r_i^2)(B_j-A_j)}\bigg|\bigg)-|r_j|\bigg(\bigg|\frac{A_j-B_i}{B_j-A_j}\bigg|+\bigg|\frac{2r_i(B_i-A_i)}{(1-r_i^2)(B_j-A_j)}\bigg|\bigg) \notag\\
    &>\dfrac{1}{|r_j|}\bigg(\frac{m}{M}-\frac{4|r_i|M}{m}\bigg)-|r_j|\bigg(\frac{M}{m}+\frac{4|r_i|M}{m}\bigg) >\frac{c'}{|r_j|},\label{yuandetaolun3}
\end{align}
for $|r_i|,|r_j|< \dfrac{m^2}{16M^2}$.
The discussion for $l_j=-1$  is similar.

Denote $\mathfrak{c}_{*}=\min\{|\sigma_{j,21}|:j\in\mathcal{G}\},\ r_* = \max\limits_{j\in \mathcal{G}}|r_j|$.
Since there are at most $2g-1$ choices of $\sigma_{j},\sigma_{j}^{-1}$ from $\sigma^{(n)}$ to $\sigma^{(n+1)}$ at every step $n$, we obtain the convergence of $\sum\limits_{\sigma\in G,\sigma \neq id} \dfrac{1}{|\sigma_{21}|^2}$ if $\dfrac{(2g-1)r_* }{c'} \leq \dfrac{1}{2}$. Moreover, we obtain the estimate
\begin{align}
    \sum\limits_{\sigma\in G,\sigma \neq id} \dfrac{1}{|\sigma_{21}|^2}&\leq \mathfrak{c}_{*}^{-2}\bigg(2g+2g(2g-1)\bigg(\frac{r_*}{c'}\bigg)^2+2g(2g-1)^2\bigg(\frac{r_*}{c'}\bigg)^4+\cdots\bigg)\label{poincareseries1}\\
    &\leq \frac{4}{3}\cdot \left(\frac{8}{7}\right)^2\cdot 2g \cdot \max_{j\in\mathcal{G}}|r_j(A_j-B_j)|^2 \leq 4g \cdot \max_{j\in\mathcal{G}}|r_j(A_j-B_j)|^2\label{poincareseries2}.
\end{align}

Note that for any $\sigma = \sigma_{i_1}^{l_1} \cdots \sigma_{i_k}^{l_k}$ with $k \ge 2$, the complex number $-\dfrac{\sigma_{22}}{\sigma_{21}}$ associated with $\sigma$ lies in the interior of at least one of the disks $\sigma_{i_k}( \overline{C_{i_1}} )$, $\sigma_{i_k}( \overline{C_{i_1}'} )$, $\sigma_{i_k}^{-1}( \overline{C_{i_1}} )$ and $\sigma_{i_k}^{-1}( \overline{C_{i_1}'} )$, which are contained in $\mathring{C}_{i_k}$ or $\mathring{C}_{i_k}'$. 
It follows that the quantity $\bigl| z + \frac{\sigma_{22}}{\sigma_{21}} \bigr|$  has a positive lower bound for all non-identity $\sigma \in G,$ and $z\in F$. 
Hence, for $(\boldsymbol A,\boldsymbol B,\boldsymbol r)\in\mathcal{U}_1 \times \mathcal{U}_2 \times \mathcal{N}^{\circ}_{\rho_0}$,
\[
    \tilde{m}(\boldsymbol A,\boldsymbol B,\boldsymbol r) := \inf_{\sigma \in G,\sigma \neq id, z\in F} \Bigl| z + \frac{\sigma_{22}}{\sigma_{21}} \Bigr| > 0.
\]

Thus, we have
\begin{equation}
    \sum\limits_{\sigma \in G,\sigma \neq id} \dfrac{1}{|\sigma_{21}z+\sigma_{22}|^2}\leq \tilde{m}(\boldsymbol A,\boldsymbol B,\boldsymbol r)^{-2} \sum\limits_{\sigma \in G,\sigma \neq id} \dfrac{1}{|\sigma_{21}|^2}.
    \label{panduan1}
\end{equation}
According to \cite[Section 5.2]{ablowitz1991solitons},
\begin{equation*}
    \sum_{\sigma\in G/G_j}\bigg(\frac{1}{z-\sigma (B_j)}-\frac{1}{z-\sigma (A_j)}\bigg) dz=\sum_{\sigma\in G_j\setminus G}\bigg(\frac{1}{\sigma z- (B_j)}-\frac{1}{\sigma z- (A_j)}\bigg) \sigma'(z)\, dz,
\end{equation*}
where $\sigma'(z)=\dfrac{1}{(\sigma_{21} z+ \sigma_{22})^2}$. The function $f(z)=\dfrac{1}{z- B_j}-\dfrac{1}{z- A_j}$ is bounded on the fundamental region. By the convergence of $\sum\limits_{\sigma \in G,\sigma \neq id} \dfrac{1}{|\sigma_{21}|^2}$ and \eqref{panduan1}, the series \eqref{differential1} converge absolutely, which implies that
$d\omega_j$ is well defined.

\bigskip
$(ii)$ 
Fix $j\in\mathcal{G}$. Denote
\begin{align*}
    h_{1,n}(\boldsymbol A,\boldsymbol B,\boldsymbol r) = \sum\limits_{\substack{\sigma\in \tilde{G}/\tilde{G}_j,\sigma \neq id,\\ \textup{len} (\sigma)=n }} |\sigma(A_j)-\sigma(B_j)|.
\end{align*}
Fix $\boldsymbol A\in \mathcal{U}_1, \boldsymbol B\in\mathcal{U}_2, \boldsymbol r\in \mathcal{N}_{\rho_0}$.
Decompose $\mathcal{G}=L(\boldsymbol r)\cup L^c(\boldsymbol r)$, where $L$ denotes the set of index $l$ such that $r_l\neq 0$, and $ L^c$ denotes  index $l$ such that $r_{l}= 0$. In the sequel, let 
\begin{align*}
    &\tilde{L}_j(\boldsymbol r)=\{\sigma=\sigma_{i_1}^{l_1}\cdots\sigma_{i_k}^{l_k}:\sigma\in \tilde{G}/\tilde{G}_j,\sigma\neq id, i_1,\cdots,i_k\in L(\boldsymbol r)\},\\
    &\tilde{L}^c_j(\boldsymbol r)=\{ \sigma=\sigma_{i_1}^{l_1}\cdots\sigma_{i_k}^{l_k}:\sigma\in \tilde{G}/\tilde{G}_j,\sigma\neq id,  \exists\  l\in\{1,\cdots,k\} \textup{ such that } i_l\in L^c(\boldsymbol r)\}.
\end{align*}
Denote 
\begin{align*}
    h_{1,n}^{(1)}(\boldsymbol A,\boldsymbol B, \boldsymbol r)=\sum\limits_{\sigma\in \tilde{L}_j(\boldsymbol r)\textup{ with length n}} |\sigma(A_j)-\sigma(B_j)|,\\
    h_{1,n}^{(2)}(\boldsymbol A,\boldsymbol B, \boldsymbol r)=\sum\limits_{\sigma\in \tilde{L}_j^c(\boldsymbol r)\textup{ with length n}} |\sigma(A_j)-\sigma(B_j)|.
\end{align*}
Then $h_{1,n}$ can be rewritten as $h_{1,n}^{(1)}+h_{1,n}^{(2)}$.
Notice that any term in $h_{1,n}^{(2)}$ equals $0$. As a consequence, the convergence of $\sum_{n}h_{1,n}$ is equivalent to the convergence of $\sum_{n}h_{1,n}^{(1)}$.

As discussed before, for $\sigma = \sigma_{j_1}^{l_1} \cdots \sigma_{j_n}^{l_n}$ belonging to $h_{1,n}^{(1)}$ with $j_n\neq j$, the complex number $-\dfrac{\sigma_{22}}{\sigma_{21}}$ lies in $\mathring{C}_{j_n}$ or $\mathring{C}_{j_n}'$. Set
\begin{align*}
    M':=\sup\limits_{\substack{i=1,2,\\ j\in\mathcal{G}}} \sup\limits_{z\in \mathcal{U}_{i,j}}|z|,\quad \bar{m}:= m-4M\rho_0,\quad
    \bar{M}:= M'+4M\rho_0.
\end{align*}
Then, for every nondegenerate reduced word contributing to \(h_{1,n}^{(1)}\),
\[
    \left| A_j+\frac{\sigma_{22}}{\sigma_{21}} \right| \geq\bar m,\qquad
    \left| B_j+\frac{\sigma_{22}}{\sigma_{21}} \right| \geq\bar m,
\]
and
\[
    |\sigma(A_j)|,\ |\sigma(B_j)|\leq\bar M.
\]
Notice that
\begin{align*}
    \sigma(A_j)-\sigma(B_j)=\frac{A_j-B_j}{(\sigma_{21} A_j +\sigma_{22})(\sigma_{21} B_j +\sigma_{22})}.
\end{align*}
Then the estimates \eqref{yuandetaolun1}, \eqref{yuandetaolun2}, \eqref{yuandetaolun3} provide that for any $\varepsilon>0$, there exists $N$ such that 
\begin{equation}\label{tobeused1}
    \sum_{n\geq N} h_{1,n}^{(1)}\lesssim \frac{2g (2g-1)^{N-1} M}{\bar{m}^2}\left(\frac{\rho_0M}{c'}\right)^{2N}<\varepsilon.
\end{equation}
Moreover, the estimate is independent of the choice of $\boldsymbol A \in \mathcal{U}_1, \boldsymbol B\in \mathcal{U}_2, \boldsymbol r \in \mathcal{N}_{\rho_0}$, which implies that \eqref{parameter1'} converge absolutely and uniformly with respect to $(\boldsymbol A,\boldsymbol B,\boldsymbol r)\in\mathcal{U}_1 \times \mathcal{U}_2 \times \mathcal{N}_{\rho_0}$. The discussions for \eqref{parameter2'} and \eqref{parameter3'} are similar, with slightly different estimates
\begin{align}
    &\sum_{n\geq N}h_{2,n}^{(1)}\lesssim \frac{4g (2g-1)^{N-1} M\bar{M}}{\bar{m}^2}\left(\frac{\rho_0M}{c'}\right)^{2N},\label{tobeused2}\\
    &\sum_{n\geq N}h_{3,n}^{(1)}\lesssim \frac{6g (2g-1)^{N-1} M\bar{M}^2}{\bar{m}^2}\left(\frac{\rho_0M}{c'}\right)^{2N} \notag,
\end{align}
for $h_{2,n}^{(1)}, h_{3,n}^{(1)}$ defined similarly to $h_{1,n}^{(1)}$.

\bigskip
$(iii)$
Fix $i,j\in\mathcal{G}$ with $i\neq j$. Denote 
\begin{align*}
    h_{4,n}(\boldsymbol A,\boldsymbol B,\boldsymbol r) &= \sum\limits_{\substack{\sigma\in G_j\setminus G/G_i,\sigma\neq id,\\ \textup{len} (\sigma)=n }} |\log\{B_j,A_j,\sigma(B_i),\sigma(A_i)\}|.
\end{align*}
Fix $\boldsymbol A\in \mathcal{U}_1, \boldsymbol B\in\mathcal{U}_2, \boldsymbol r\in \mathcal{N}^{\circ}_{\rho_0}$.
Notice that
\begin{align*}
    \{B_j,A_j,\sigma(B_i),\sigma(A_i)\}-1=\frac{(A_j-B_j)(\sigma(A_i)-\sigma(B_i))}{(B_j-\sigma(A_i))(A_j-\sigma(B_i))}.
\end{align*}
Thus, with $\rho_0 \leq \frac{7m^2}{64gM^2} \leq \frac{7\bar{m}^2}{16gM^2}$, we have
\begin{align}
    \left| \frac{(A_j-B_j)(\sigma(A_i)-\sigma(B_i))}{(B_j-\sigma(A_i))(A_j-\sigma(B_i))} \right| &\leq \frac{M}{\bar{m}^2} |\sigma(A_i)-\sigma(B_i)| \leq \frac{M^2}{\bar{m}^4}|\sigma_{21}|^{-2}\label{logbranch1}\\
    &\leq \frac{M^2}{\bar{m}^4}\cdot 2g\cdot \left(\frac{8}{7}M\rho_0\right)^2 \leq \frac{1}{2}.\label{logbranch2}
\end{align}
By \eqref{logbranch1}--\eqref{logbranch2}, every non-identity
cross-ratio lies in the disk \(|z-1|\leq1/2\). We use the holomorphic
branch of the logarithm on this disk normalized by \(\log1=0\). This branch satisfies
$\overline{\log(z)} = \log(\overline{z})$.
Hence, we obtain
\begin{align*}
    \sum_{\sigma\in G_j\setminus G/G_i,\sigma\neq id} |\log\{B_j,A_j,\sigma (B_i),\sigma (A_i)\} |&\leq \sum_{\sigma\in G_j\setminus G/G_i,\sigma\neq id} 2\left|\frac{(A_j-B_j)(\sigma(A_i)-\sigma(B_i))}{(B_j-\sigma(A_i))(A_j-\sigma(B_i))}\right|\\
    &\leq \frac{2M}{\bar{m}^2}\sum_{\sigma\in G_j\setminus G/G_i,\sigma\neq id} |\sigma(A_i)-\sigma(B_i)|.
\end{align*}
Thus, as discussed before, for any $\varepsilon>0$, there exists $N$ such that 
\begin{align*}
    \sum_{n\geq N} h_{4,n} \lesssim \frac{2g (2g-1)^{N-1} M^2}{\bar{m}^4}\left(\frac{\rho_0M}{c'}\right)^{2N}<\varepsilon.
\end{align*}
Moreover, the estimate is independent of the choice of $\boldsymbol A \in \mathcal{U}_1, \boldsymbol B\in \mathcal{U}_2, \boldsymbol r \in \mathcal{N}^{\circ}_{\rho_0}$, which implies that \eqref{parameter4'} converge absolutely and locally uniformly on $\mathcal{U}_1 \times \mathcal{U}_2 \times \mathcal{N}^{\circ}_{\rho_0}$. Meanwhile, the summation estimate \eqref{poincareseries1}--\eqref{poincareseries2} implies that $|f_{4,i,j} - \log\{B_j,A_j, B_i, A_i\}|\lesssim |\boldsymbol r|_{\infty}^2$. The proof for \eqref{parameter5'} is omitted as it follows a similar argument.
\end{proof}

\smallskip

\begin{remark}
    The main scheme of the proof follows \cite{burnside1891glass}, wherein it was proved that the Poincaré series
    \[
        \sum_{\sigma\in G, \sigma\neq id}
        \frac{1}{|\sigma_{21}z+\sigma_{22}|^2},
        \qquad
        \sigma(z)=
        \frac{\sigma_{11}z+\sigma_{12}}
        {\sigma_{21}z+\sigma_{22}},
    \]
    converges absolutely if the pairs of Schottky circles are sufficiently
    distant from one another and the multipliers are sufficiently small.
    Since we are also interested in the uniform convergence of the
    parameter series, the qualitative
    result of Burnside cannot be used directly. We need to investigate the
    dependence of the convergence estimates on \(A_j,B_j\) and \(r_j\).

    As a result, this argument establishes not only the absolute convergence of the Poincaré series, but also the uniform convergence of the series \eqref{parameter1'}--\eqref{parameter5'} with quantitative estimates.
\end{remark}

\smallskip

\begin{lemma}\label{reallm}
Assume that the parameters belong to the domains of
Lemma~\ref{convergelm}, with \(\rho_0\) sufficiently small that
\eqref{logbranch1}--\eqref{logbranch2} hold.

\begin{itemize}
    \item[$(i)$]
    Suppose that
    \[
        A_j=\overline{B_j}, \qquad r_j\in(-1,1), \qquad j\in\mathcal G.
    \]
    Then
    \[
        f_{1,j},f_{2,j},f_{3,j}\in i\mathbb R.
    \]
    If, in addition, \(r_j\neq0\) for every \(j\), then
    \[
        f_{4,i,j},f_{5,j},c\in\mathbb R, \qquad i\neq j.
    \]
    \item[$(ii)$]
    Suppose that
    \[
        A_j,B_j,r_j\in\mathbb R, \qquad r_j\in(-1,1), \qquad j\in\mathcal G.
    \]
    Then
    \[
        f_{1,j},f_{2,j},f_{3,j}\in\mathbb R.
    \]
    If \(r_j\neq0\) for every \(j\), then
    \[
        f_{5,j},c\in\mathbb R.
    \]
    Moreover, for \(i\neq j\), the quantity $\{B_j,A_j,B_i,A_i\} \in \mathbb{R}\setminus\{0\}$,
    and
    \begin{equation*}
        \sum_{\substack{ \sigma\in G_j\setminus G/G_i\\ \sigma\neq\mathrm{id} }}\log\{B_j,A_j,\sigma(B_i),\sigma(A_i)\} \in \mathbb{R},
    \end{equation*}
    in the equation \eqref{parameter4'}.
\end{itemize}
\end{lemma}
\begin{proof}
$(i)$
Notice that when $A_j=\bar{B}_j,r_j\in (-1,0)\cup(0,1)$, we have $\overline{\sigma_j(\bar{z})}=\sigma_j^{-1}(z)$.
More precisely, recall
\begin{equation*}
    \sigma_j(z)=\frac{\sigma_{j,11}z+\sigma_{j,12}}{\sigma_{j,21}z+\sigma_{j,22}},\textup{ for }z\in\mathbb{CP}^1
\end{equation*}
where 
\begin{equation*}
   \begin{pmatrix}
        \sigma_{j,11}&\sigma_{j,12}\\
        \sigma_{j,21}&\sigma_{j,22}
    \end{pmatrix}
    =
    \frac{1}{A_j-B_j}
    \begin{pmatrix}
        A_jr_j-B_j/r_j
        &
        A_jB_j(1/r_j-r_j)\\
        r_j-1/r_j
        &
        A_j/r_j-B_jr_j
    \end{pmatrix},
\end{equation*}
Thus we have
\begin{equation*}
    \overline{\sigma_j(\bar{z})}=\frac{\overline{\sigma_{j,11}}z+\overline{\sigma_{j,12}}}{\overline{\sigma_{j,21}}z+\overline{\sigma_{j,22}}},\textup{ for }z\in\mathbb{CP}^1
\end{equation*}
where 
\begin{equation*}
   \begin{pmatrix}
        \overline{\sigma_{j,11}} & \overline{\sigma_{j,12}}\\
        \overline{\sigma_{j,21}} & \overline{\sigma_{j,22}}
    \end{pmatrix}
    =
    \frac{1}{B_j-A_j}
    \begin{pmatrix}
        B_jr_j-A_j/r_j
        &
        A_jB_j(1/r_j-r_j)\\
        r_j-1/r_j
        &
        B_j/r_j-A_jr_j
    \end{pmatrix}.
\end{equation*}
Notice that the right hand side of the preceding equation is $M_{\sigma_j^{-1}}$, which implies that $\overline{\sigma_j(\bar{z})}=\sigma_j^{-1}(z)$. When $r_j = 0$, the definition of $\sigma_j$ also implies that $\overline{\sigma_j(\bar{z})}=\sigma_j^{-1}(z)$.

Denote $\sigma=\sigma_{i_1}^{l_1}\cdots\sigma_{i_k}^{l_k},\iota=\sigma_{i_1}^{-l_1}\cdots\sigma_{i_k}^{-l_k}$. This correspondence is a bijection of the chosen reduced-word representative set. Thus, for $r_j \in (-1,1)$, we have
\begin{align}
    &\bar{f}_{1,j}=\sum_{\iota\in \tilde{G}/\tilde{G}_j}\iota(B_j)-\iota(A_j)=-f_{1,j}\in i\mathbb{R},\label{implicitcondition1}\\
    &\bar{f}_{2,j}=\sum_{\iota\in \tilde{G}/\tilde{G}_j}\iota(B_j)^2-\iota(A_j)^2=-f_{2,j}\in i\mathbb{R},\label{implicitcondition2}\\
    &\bar{f}_{3,j}=\sum_{\iota\in \tilde{G}/\tilde{G}_j}\iota(B_j)^3-\iota(A_j)^3=-f_{3,j}\in i\mathbb{R} \notag.
\end{align}
Moreover,
\[
    \{B_j,A_j,B_i,A_i\} = \frac{|A_j-A_i|^2}{|A_j-B_i|^2}  >0, \qquad i\neq j.
\]
Thus, we have
\begin{align}
    &\bar{f}_{4,i,j}=\log\{B_j,A_j, B_i, A_i\} + \sum_{\iota\in G_j\setminus G/G_i,\iota\neq id}\log\{A_j,B_j,\iota A_i,\iota B_i\} = f_{4,i,j} \in \mathbb{R}, i\neq j, \label{realmatrix1}\\
    &\bar{f}_{5,j}=\log r_j^2+\sum_{\iota\in G_j\setminus G/G_j,\iota\neq id}\log\{A_j,B_j,\iota A_j,\iota B_j\} = f_{5,j} \in \mathbb{R}, \label{realmatrix2}
\end{align}
with the cross-ratio identity
\[
    \{A,B,C,D\} = \{B,A,D,C\}.
\]
All fractional linear transformations are represented by the determinant-one matrices fixed above. Changing the representative by a sign does not affect $\sigma_{21}^{-2}$. Similarly, we have
\begin{equation}
    \bar{c} = \sum_{\sigma \in G,\sigma\neq id} \overline{\sigma_{21}^{-2}} = \sum_{\iota \in G,\iota\neq id} \iota_{21}^{-2} = c\in \mathbb{R} \label{implicitcondition3}.
\end{equation}
\smallskip
$(ii)$
When $A_j,B_j\in\mathbb{R},r_j\in(-1,1)$, we have $\overline{\sigma_j(\bar{z})}=\sigma_j(z)$. Thus, for $r_j \in (-1,1)$, we have

\begin{align}
    &\bar{f}_{1,j}=\sum_{\sigma\in \tilde{G}/\tilde{G}_j}\sigma(A_j)-\sigma(B_j)=f_{1,j}\in\mathbb{R}, \label{implicitcondition4}\\
    &\bar{f}_{2,j}=\sum_{\sigma\in \tilde{G}/\tilde{G}_j}\sigma(A_j)^2-\sigma(B_j)^2=f_{2,j}\in\mathbb{R}, \label{implicitcondition5}\\
    &\bar{f}_{3,j}=\sum_{\sigma\in \tilde{G}/\tilde{G}_j}\sigma(A_j)^3-\sigma(B_j)^3=f_{3,j}\in\mathbb{R} \notag.
\end{align}
In addition, if $r_j \neq 0$ for all $j\in \mathcal{G}$, we have
\begin{align*}
    \sum_{\sigma\in G_j\setminus G/G_i,\sigma\neq id} &\overline{\log\{B_j,A_j,\sigma (B_i),\sigma (A_i)\}} = \sum_{\sigma\in G_j\setminus G/G_i,\sigma\neq id}\log\{B_j,A_j,\sigma (B_i),\sigma (A_i)\} \in \mathbb{R}, i\neq j,\\
    &\bar{f}_{5,j} =\log r_j^2+\sum_{\sigma\in G_j\setminus G/G_j,\sigma\neq id}\log\{B_j,A_j,\sigma (B_j),\sigma (A_j)\} = f_{5,j} \in \mathbb{R}.
\end{align*}
All fractional linear transformations are represented by the determinant-one matrices fixed above. Changing the representative by a sign does not affect $\sigma_{21}^{-2}$. Thus
\begin{equation}\label{implicitcondition6}
    \bar{c}=\sum_{\sigma \in G,\sigma\neq id} \sigma_{21}^{-2} = c\in \mathbb{R}.
\end{equation}
Moreover, since $A_1,B_1,\cdots,A_g,B_g$ belong to disjoint domains, $\{B_j,A_j, B_i, A_i\} \in \mathbb{R}\setminus \{0\}$.
\end{proof}
For KP-II, condition $(ii)$ in Lemma \ref{reallm} does not yield time-multi-periodic solutions considered here, thus we impose condition $(i)$. Conversely, for the KP-I equation, condition $(i)$ in Lemma \ref{reallm} does not yield time-multi-periodic solutions considered here, and we instead impose condition $(ii)$.

\begin{lemma}\label{solveuv}
    If $g>1$, suppose $\dfrac{\sqrt{3}\pi}{\gamma}\in\mathbb{R}\setminus\mathbb{Q}$. 
    Let $\boldsymbol m = (m_1,\cdots,m_g),\ \boldsymbol n = (n_1,\cdots,n_g)\in\mathbb{Z}^g$ satisfy $m_j\neq 0$ and $m_j n_k-m_k n_j\neq 0, \textup{ for }j,k\in\mathcal{G},\ j\neq k.$
Let
\begin{align*}
    &A_{j,0}^{(1)}=-\dfrac{n_j}{2\sqrt{3} m_j}-\dfrac{\pi}{\gamma}m_j,\quad B_{j,0}^{(1)}=-\dfrac{n_j}{2\sqrt{3} m_j}+\dfrac{\pi}{\gamma}m_j,\\ 
    &A_{j,0}^{(2)}=\dfrac{n_j}{2\sqrt{3}m_j}+i\dfrac{\pi}{\gamma}m_j,\quad B_{j,0}^{(2)}=\dfrac{n_j}{2\sqrt{3}m_j}-i\dfrac{\pi}{\gamma}m_j,\quad j\in\mathcal{G}.
\end{align*}
Denote
\[
    \boldsymbol A_{0}^{(\iota)} = (A_{1,0}^{(\iota)},\cdots,A_{g,0}^{(\iota)}), \boldsymbol B_{0}^{(\iota)} = (B_{1,0}^{(\iota)},\cdots,B_{g,0}^{(\iota)}),\quad \iota =1,2.
\]
\begin{itemize}
    \item[$(i)$] There exist real-analytic functions $\boldsymbol A^{(1)}(\boldsymbol r),\ \boldsymbol B^{(1)}(\boldsymbol r)$ and a neighborhood $\mathcal{U}_{\boldsymbol 0}^{(1)} \subset \mathbb{R}^g$ of $\boldsymbol 0$ such that for $\boldsymbol{r}\in \mathcal{U}^{(1)}_{\boldsymbol 0}$,
    \begin{equation}
        \boldsymbol A^{(1)}(\boldsymbol r) = \boldsymbol A_{0}^{(1)} + O(\boldsymbol r^2),\quad \boldsymbol B^{(1)}(\boldsymbol r) = \boldsymbol B_{0}^{(1)} + O(\boldsymbol r^2),\label{distancewithAj0Bj02}
    \end{equation}
    and
    \begin{equation}\label{restriction1}
        -f_{1,j}(\boldsymbol A^{(1)}(\boldsymbol{r}),\boldsymbol B^{(1)}(\boldsymbol{r}),\boldsymbol r)=\dfrac{2\pi}{\gamma}m_j,\quad  \sqrt{3}f_{2,j}(\boldsymbol A^{(1)}(\boldsymbol{r}) ,\boldsymbol B^{(1)}(\boldsymbol{r}), \boldsymbol r)=\frac{2\pi}{\gamma}n_j,\quad j\in \mathcal{G}.
    \end{equation}
    
    \item[$(ii)$] There exist real-analytic functions $\boldsymbol\alpha^{(2)}(\boldsymbol r),\boldsymbol\beta^{(2)}(\boldsymbol r)$ well defined on a neighborhood $ \mathcal{U}_{\boldsymbol 0}^{(2)} \allowbreak \subset \mathbb{R}^g$ of $\boldsymbol 0$. Denote $\boldsymbol A^{(2)}(\boldsymbol r) = \boldsymbol\alpha^{(2)}(\boldsymbol r) + i\boldsymbol\beta^{(2)}(\boldsymbol r),\ \boldsymbol B^{(2)}(\boldsymbol r) = \boldsymbol\alpha^{(2)}(\boldsymbol r) - i\boldsymbol\beta^{(2)}(\boldsymbol r)$. For $\boldsymbol{r}\in \mathcal{U}^{(2)}_{\boldsymbol 0}$,
    \begin{equation}
        \boldsymbol A^{(2)}(\boldsymbol r) = \boldsymbol A_{0}^{(2)} + O(\boldsymbol r^2),\quad \boldsymbol B^{(2)}(\boldsymbol r) = \boldsymbol B_{0}^{(2)} + O(\boldsymbol r^2),\label{distancewithAj0Bj0}
    \end{equation}
    and
    \begin{equation}\label{restriction2}
        f_{1,j}(\boldsymbol A^{(2)}(\boldsymbol{r}),\boldsymbol B^{(2)}(\boldsymbol{r}),\boldsymbol r)=\dfrac{2\pi i}{\gamma}m_j, \quad \sqrt{3}f_{2,j}(\boldsymbol A^{(2)}(\boldsymbol{r}),\boldsymbol B^{(2)}(\boldsymbol{r}),\boldsymbol r)=\dfrac{2\pi i}{\gamma}n_j,\quad j\in\mathcal{G}.
    \end{equation}
\end{itemize}
\end{lemma}

\begin{proof}
    $(i)$ 
    Denote 
    \begin{align*}
         &\tilde{m}=\min\{ |P-Q|: P,Q\in \{A_{1,0}^{(1)}, B_{1,0}^{(1)}, \cdots, A_{g,0}^{(1)}, B_{g,0}^{(1)}\},P\neq Q\},\\ 
         &\tilde{M}=\max\{ |P-Q|: P,Q\in \{A_{1,0}^{(1)}, B_{1,0}^{(1)}, \cdots, A_{g,0}^{(1)}, B_{g,0}^{(1)}\},P\neq Q\},\\
         &\tilde{M}'=\max\{ |P|: P\in \{A_{1,0}^{(1)}, B_{1,0}^{(1)}, \cdots, A_{g,0}^{(1)}, B_{g,0}^{(1)}\}\}.
    \end{align*}
    Since $\frac{\sqrt{3}\pi}{\gamma} \notin \mathbb{Q}$ for $g>1$ and $\frac{n_j}{m_j} \neq \frac{n_k}{m_k}$ for $j\neq k$, the points $A_{1,0}^{(1)}, B_{1,0}^{(1)}, \cdots, A_{g,0}^{(1)}, B_{g,0}^{(1)}$ are pairwise distinct. Hence the constant $\tilde{m}>0$.

    Denote $\boldsymbol A = (A_{1},\cdots,A_{g})\in \mathbb{C}^g, \boldsymbol B = (B_{1},\cdots,B_{g})\in \mathbb{C}^g$.
    Now we introduce the neighborhoods for $\boldsymbol A,\boldsymbol B$ as follows:
    \begin{align*}
        &\mathcal{U}_{\boldsymbol A_0}^{(1)}=\{\boldsymbol A:|A_j-A_{j,0}^{(1)}|<\frac{\tilde{m}}{3},\textup{for every } j\in\mathcal{G} \}, \\
        &\mathcal{U}_{\boldsymbol B_0}^{(1)}=\{\boldsymbol B:|B_j-B_{j,0}^{(1)}|<\frac{\tilde{m}}{3},\textup{for every } j\in\mathcal{G}  \},
    \end{align*}
    Then for $(\boldsymbol A,\boldsymbol B)\in \mathcal{U}_{\boldsymbol A_0}^{(1)}\times \mathcal{U}_{\boldsymbol B_0}^{(1)} $, the corresponding constants in Lemma \ref{convergelm} are
    \begin{align*}
        m=\dfrac{\tilde{m}}{3},\quad M=\tilde{M}+\dfrac{2\tilde{m}}{3},\quad M'=\tilde{M}'+\dfrac{\tilde{m}}{3},\quad
        \rho_0^{(1)}<\min\bigg\{\dfrac{1}{8},\dfrac{m^2}{16gM^2},\dfrac{m}{16(2g-1)M},\dfrac{1}{M}\bigg\}.
    \end{align*}
    And $\bar{m},\bar{M}$ are controlled by the same estimates with respect to $m,M,\rho_0^{(1)}$.
    
    Lemma \ref{lm6.4} implies that, for $k=1,2,3$, $\sigma(A_j)^k-\sigma(B_j)^k$ is holomorphic with respect to $(\boldsymbol A, \boldsymbol B,\boldsymbol r)\in \mathcal{U}_{\boldsymbol A_0}^{(1)} \times \mathcal{U}_{\boldsymbol B_0}^{(1)} \times \mathcal{N}_{\rho_0^{(1)}}$. Since the series \eqref{parameter1'}, \eqref{parameter2'}, \eqref{parameter3'} converge absolutely and locally uniformly in $\mathcal{U}_{\boldsymbol A_0}^{(1)} \times \mathcal{U}_{\boldsymbol B_0}^{(1)} \times \mathcal{N}_{\rho_0^{(1)}}$, with the Weierstrass theorem, functions $f_{1,j},f_{2,j},f_{3,j}$ are holomorphic in $\mathcal{U}_{\boldsymbol A_0}^{(1)} \times \mathcal{U}_{\boldsymbol B_0}^{(1)} \times \mathcal{N}_{\rho_0^{(1)}}$.

    Moreover, with \eqref{poincareseries1},\eqref{poincareseries2}, \eqref{tobeused1}, \eqref{tobeused2}, we have
    \begin{align*}
        &|f_{1,j}(\boldsymbol A,\boldsymbol B,\boldsymbol r)-(A_j-B_j)|\lesssim {\bar{m}}^{-2}(\tilde{M}+\frac{2\tilde{m}}{3})^3|\boldsymbol r|_{\infty}^2,\\
        &|f_{2,j}(\boldsymbol A,\boldsymbol B,\boldsymbol r)-(A_j^2-B_j^2)| \lesssim \frac{2\bar{M}(\tilde{M}+\frac{2\tilde{m}}{3})^3}{\bar{m}^2}|\boldsymbol r|_{\infty}^2,
    \end{align*}
    for $(\boldsymbol A,\boldsymbol B,\boldsymbol r)\in \mathcal{U}_{\boldsymbol A_0}^{(1)} \times \mathcal{U}_{\boldsymbol B_0}^{(1)} \times \mathcal{N}_{\rho_0^{(1)}}$.
    Thus, we have
    \begin{equation}\label{donotneediteration}
        f_{1,j}(\boldsymbol A,\boldsymbol B,\boldsymbol r) = A_j-B_j+O(\boldsymbol r^2),\quad f_{2,j}(\boldsymbol A,\boldsymbol B,\boldsymbol r) = A_j^2-B_j^2+O(\boldsymbol r^2).
    \end{equation}
    Evaluating the expansions at $\boldsymbol r = \boldsymbol 0$, for $j,k\in\mathcal{G}$, we have
    \begin{align*}
        &\dfrac{\partial f_{1,j}}{\partial A_k}(\boldsymbol A_{0}^{(1)},\boldsymbol B_0^{(1)},\boldsymbol 0) = \delta_{j,k},\quad
        &\dfrac{\partial f_{2,j}}{\partial A_k}(\boldsymbol A_{0}^{(1)},\boldsymbol B_0^{(1)},\boldsymbol 0) = 2A_{j,0}^{(1)}\delta_{j,k},\\
        &\dfrac{\partial f_{1,j}}{\partial  B_k}(\boldsymbol A_{0}^{(1)},\boldsymbol B_0^{(1)},\boldsymbol 0) = -\delta_{j,k},\quad
        &\dfrac{\partial f_{2,j}}{\partial  B_k}(\boldsymbol A_{0}^{(1)},\boldsymbol B_0^{(1)},\boldsymbol 0) = -2B_{j,0}^{(1)}\delta_{j,k}.    
    \end{align*}
    By restricting the variables to the real space, i.e., \(\boldsymbol A,\boldsymbol  B,\boldsymbol r\in \mathbb{R}^{g}\), equations \eqref{implicitcondition4} and \eqref{implicitcondition5} imply that \(f_{1,j},f_{2,j}\in \mathbb{R}\).
    Thus the equation 
    \begin{equation*}
        -f_{1,j}(\boldsymbol A,\boldsymbol B,\boldsymbol r)=\dfrac{2\pi}{\gamma}m_j,\  \sqrt{3}f_{2,j}(\boldsymbol A ,\boldsymbol B, \boldsymbol r)=\frac{2\pi}{\gamma}n_j,\quad j=1,\cdots,g,
    \end{equation*}
    reduces to
    \begin{equation}\label{reducedrelation2}
        -\textup{Re }f_{1,j}(\boldsymbol A,\boldsymbol B,\boldsymbol r)=\frac{2\pi}{\gamma}m_j,\quad
        \sqrt{3}\textup{Re }f_{2,j}(\boldsymbol A,\boldsymbol B,\boldsymbol r)=\frac{2\pi}{\gamma}n_j,\quad j=1,\cdots,g.
    \end{equation}
    Notice that $\boldsymbol A_{0}^{(1)},\boldsymbol B_{0}^{(1)},\boldsymbol r=0$ is a solution for \eqref{reducedrelation2}, and
    \begin{align*}
        &-\dfrac{\partial \textup{ Re}f_{1,j}}{\partial A_k}(\boldsymbol A_{0}^{(1)},\boldsymbol B_0^{(1)},\boldsymbol 0) = -\delta_{j,k},
        \quad \sqrt{3} \dfrac{\partial \textup{ Re}f_{2,j}}{\partial A_k}(\boldsymbol A_{0}^{(1)},\boldsymbol B_0^{(1)},\boldsymbol 0) = 2\sqrt{3}  A_{j,0}^{(1)}\delta_{j,k} ,\\
        &-\dfrac{\partial \textup{ Re}f_{1,j}}{\partial  B_k}(\boldsymbol A_{0}^{(1)},\boldsymbol B_0^{(1)},\boldsymbol 0) = \delta_{j,k},
        \quad \sqrt{3}\dfrac{\partial \textup{ Re}f_{2,j}}{\partial  B_k}(\boldsymbol A_{0}^{(1)},\boldsymbol B_0^{(1)},\boldsymbol 0) = -2\sqrt{3}B_{j,0}^{(1)}\delta_{j,k}.
    \end{align*}
    Denote
    \[
        \mathcal F_1(\boldsymbol A,\boldsymbol B,\boldsymbol r)= \left( -\textup{Re} f_{1,1}-\frac{2\pi m_1}{\gamma}, \sqrt3\,\textup{Re} f_{2,1}-\frac{2\pi n_1}{\gamma},\cdots, -\textup{Re} f_{1,g}-\frac{2\pi m_g}{\gamma}, \sqrt3\,\textup{Re} f_{2,g}-\frac{2\pi n_g}{\gamma}\right).
    \]
    After ordering the variables in the form $(A_1, B_1, \cdots, A_g, B_g)$, the Jacobian matrix of $\mathcal F_1$ at $(\boldsymbol A_{0}^{(1)}, \boldsymbol B_0^{(1)}, \boldsymbol 0)$ is block diagonal. Its \(j\)-th diagonal block has determinant $ 2\sqrt{3}(B_{j,0}^{(1)}-A_{j,0}^{(1)}) \neq0$. Hence the Jacobian matrix is invertible.
    
    By the real-analytic implicit function theorem (see, e.g., \cite[Chapter 2]{krantz2002implicit}), there exist real-analytic functions $\boldsymbol A^{(1)}(\boldsymbol r), \boldsymbol B^{(1)}(\boldsymbol r)\in \mathbb{R}^g$ defined for \(\boldsymbol r\) in a neighborhood \(\mathcal U_{\boldsymbol0}^{(1)}\subset \mathcal{N}_{\rho_0^{(1)}} \cap \mathbb{R}^g\), which solve \eqref{restriction1}.
    
    Set $x_1(\boldsymbol r):=\bigl(\boldsymbol A^{(1)}(\boldsymbol r), \boldsymbol B^{(1)}(\boldsymbol r) \bigr)$ and $\Phi_1(\boldsymbol A,\boldsymbol B) := \mathcal F_1( \boldsymbol A,\boldsymbol B,\boldsymbol0 )$.
    The estimates in \eqref{donotneediteration} give, locally uniformly,
    \[
        \mathcal F_1(x,\boldsymbol r) = \Phi_1(x) + O(\boldsymbol r^2).
    \]
    Since $\mathcal F_1(x_1(\boldsymbol r),\boldsymbol r) = \mathcal F_1(x_1(\boldsymbol0),\boldsymbol0)$, we have
    \[
        |\Phi_1(x_1(\boldsymbol r)) - \Phi_1(x_1(\boldsymbol0))|_{\infty} \lesssim |\boldsymbol r|_{\infty}^2.
    \]
    Since the Jacobian matrix \(\partial_x\Phi_1(x_1(\boldsymbol0))\) is invertible, \(\Phi_1\) has a locally Lipschitz inverse near \(x_1(\boldsymbol0)\). Consequently,
    \[
        |x_1(\boldsymbol r)-x_1(\boldsymbol0)|_{\infty} \lesssim |\boldsymbol r|_{\infty}^2,
    \]
    which implies that
    \begin{equation*}
        A_j^{(1)}(\boldsymbol r) = A_{j,0}^{(1)} + O(\boldsymbol r^2),\quad B_j^{(1)}(\boldsymbol r) = B_{j,0}^{(1)} + O(\boldsymbol r^2).
    \end{equation*}
    
    \bigskip
    $(ii)$
    Denote 
    \begin{align*}
         &\tilde{m}=\min\{ |P-Q|: P,Q\in \{A_{1,0}^{(2)}, B_{1,0}^{(2)}, \cdots, A_{g,0}^{(2)}, B_{g,0}^{(2)}\},P\neq Q\},\\ 
         &\tilde{M}=\max\{ |P-Q|: P,Q\in \{A_{1,0}^{(2)}, B_{1,0}^{(2)}, \cdots, A_{g,0}^{(2)}, B_{g,0}^{(2)}\},P\neq Q\},\\
         &\tilde{M}'=\max\{ |P|: P\in \{A_{1,0}^{(2)}, B_{1,0}^{(2)}, \cdots, A_{g,0}^{(2)}, B_{g,0}^{(2)}\}\}.
    \end{align*}
    Since $m_j\neq 0$ and $\frac{n_j}{m_j} \neq \frac{n_k}{m_k}$ for $j\neq k$, the points $A_{1,0}^{(2)}, B_{1,0}^{(2)}, \cdots, A_{g,0}^{(2)}, B_{g,0}^{(2)}$ are pairwise distinct. Hence the constant $\tilde{m}>0$. Define $\mathcal{U}_{\boldsymbol A_0}^{(2)}, \mathcal{U}_{\boldsymbol B_0}^{(2)}, \mathcal{N}_{\rho_0^{(2)}}$ as above.
    Similarly, $f_{1,j},f_{2,j}$ are holomorphic with respect to $(\boldsymbol A,\boldsymbol  B,\boldsymbol r)\in \mathcal{U}_{\boldsymbol  A_0}^{(2)} \times \mathcal{U}_{\boldsymbol  B_0}^{(2)} \times \mathcal{N}_{\rho_0^{(2)}}$.

    Now we apply a variable transformation
    \begin{equation}\label{transform}
        \boldsymbol A=\boldsymbol\alpha+i\boldsymbol\beta,\quad
        \boldsymbol B=\boldsymbol\alpha-i\boldsymbol\beta.
    \end{equation}
    Denote $\boldsymbol \alpha_0^{(2)}= \frac{\boldsymbol A_{0}^{(2)}+\boldsymbol B_{0}^{(2)}}{2},\allowbreak \boldsymbol \beta_0^{(2)}= \frac{\boldsymbol A_{0}^{(2)}-\boldsymbol B_{0}^{(2)}}{2i}$. Then 
    \begin{align*}
        f_{1,j}(\boldsymbol\alpha,\boldsymbol \beta,\boldsymbol r) = 2i\beta_j + O(\boldsymbol r^2),\quad
        f_{2,j}(\boldsymbol\alpha,\boldsymbol \beta,\boldsymbol r) = 4i\alpha_j\beta_j + O(\boldsymbol r^2)
    \end{align*}    
    are holomorphic with respect to $(\boldsymbol\alpha,\boldsymbol \beta,\boldsymbol r)$.
    
    Evaluating the expansions at $\boldsymbol r = \boldsymbol 0$, for $j,k\in\mathcal{G}$, we have
    \begin{align*}
        &\dfrac{\partial f_{1,j}}{\partial \alpha_k}(\boldsymbol \alpha_{0}^{(2)},\boldsymbol\beta_0^{(2)},\boldsymbol 0) =0,\quad
        &\dfrac{\partial f_{2,j}}{\partial \alpha_k}(\boldsymbol \alpha_{0}^{(2)},\boldsymbol\beta_0^{(2)},\boldsymbol 0) = 4i\beta_{j,0}^{(2)}\delta_{j,k},\\
        &\dfrac{\partial f_{1,j}}{\partial \beta_k}(\boldsymbol \alpha_{0}^{(2)},\boldsymbol\beta_0^{(2)},\boldsymbol 0) = 2i\delta_{j,k},\quad
        &\dfrac{\partial f_{2,j}}{\partial \beta_k}(\boldsymbol \alpha_{0}^{(2)},\boldsymbol\beta_0^{(2)},\boldsymbol 0) = 4i\alpha_{j,0}^{(2)}\delta_{j,k}.    
    \end{align*}
    By restricting the variables to the real space, i.e., \(\boldsymbol\alpha,\boldsymbol \beta,\boldsymbol r\in \mathbb{R}^{g}\), equations \eqref{implicitcondition1} and \eqref{implicitcondition2} imply that \(f_{1,j},f_{2,j}\in i\mathbb{R}\). Thus, the system 
    \begin{equation*}
        f_{1,j}(\boldsymbol A,\boldsymbol B,\boldsymbol r)=\dfrac{2\pi i}{\gamma}m_j, \qquad \sqrt{3}f_{2,j}(\boldsymbol A,\boldsymbol B,\boldsymbol r)=\dfrac{2\pi i}{\gamma}n_j ,\quad j=1,\cdots,g,
    \end{equation*}
    reduces to
    \begin{equation}\label{reducedrelation1}
        \textup{Im}f_{1,j}(\boldsymbol\alpha,\boldsymbol\beta,\boldsymbol r)=\frac{2\pi}{\gamma}m_j,\quad
        \sqrt{3}\textup{ Im}f_{2,j}(\boldsymbol \alpha,\boldsymbol\beta,\boldsymbol r)=\frac{2\pi}{\gamma}n_j,\quad j=1,\cdots,g.
    \end{equation}
    Notice that $\boldsymbol \alpha_{0}^{(2)},\boldsymbol \beta_{0}^{(2)},\boldsymbol r=0$ is a solution for \eqref{reducedrelation1}, and
    \begin{align*}
        &\dfrac{\partial \textup{ Im}f_{1,j}}{\partial \alpha_k}(\boldsymbol \alpha_{0}^{(2)},\boldsymbol\beta_0^{(2)},\boldsymbol 0) =0,
        \quad \sqrt{3} \dfrac{\partial \textup{ Im}f_{2,j}}{\partial \alpha_k}(\boldsymbol \alpha_{0}^{(2)},\boldsymbol\beta_0^{(2)},\boldsymbol 0) = 4\sqrt{3} \beta_{j,0}^{(2)}\delta_{j,k},\\
        &\dfrac{\partial \textup{ Im}f_{1,j}}{\partial \beta_k}(\boldsymbol \alpha_{0}^{(2)},\boldsymbol\beta_0^{(2)},\boldsymbol 0) = 2\delta_{j,k},
        \quad \sqrt{3}\dfrac{\partial \textup{ Im}f_{2,j}}{\partial \beta_k}(\boldsymbol \alpha_{0}^{(2)},\boldsymbol\beta_0^{(2)},\boldsymbol 0) = 4\sqrt{3}\alpha_{j,0}^{(2)}\delta_{j,k}.    
    \end{align*}
    Denote
    \[
        \mathcal F_2(\boldsymbol\alpha,\boldsymbol\beta,\boldsymbol r)= \left( \textup{Im} f_{1,1}-\frac{2\pi m_1}{\gamma}, \sqrt3\,\textup{Im} f_{2,1}-\frac{2\pi n_1}{\gamma},\cdots, \textup{Im} f_{1,g}-\frac{2\pi m_g}{\gamma}, \sqrt3\,\textup{Im} f_{2,g}-\frac{2\pi n_g}{\gamma}\right).
    \]
    After ordering the variables in the form $(\alpha_1, \beta_1, \cdots, \alpha_g, \beta_g)$, the Jacobian matrix of $\mathcal F_2$ at $(\boldsymbol \alpha_{0}^{(2)}, \boldsymbol\beta_0^{(2)}, \boldsymbol 0)$ is block diagonal. Its \(j\)-th diagonal block has determinant $ -8\sqrt{3}\,\beta_{j,0}^{(2)} = -\frac{8\sqrt{3}\pi}{\gamma}m_j \neq0$. Hence the Jacobian matrix is invertible.
    
    By the real-analytic implicit function theorem, there exist real-analytic functions $\boldsymbol\alpha^{(2)}(\boldsymbol r), \boldsymbol\beta^{(2)}(\boldsymbol r)$ defined for \(\boldsymbol r\) in a neighborhood \(\mathcal U_{\boldsymbol0}^{(2)}\subset \mathcal{N}_{\rho_0^{(2)}} \cap \mathbb{R}^g\), which solve \eqref{restriction2}.
    
    Set $x_2(\boldsymbol r):=\bigl(\boldsymbol\alpha^{(2)}(\boldsymbol r), \boldsymbol\beta^{(2)}(\boldsymbol r) \bigr)$ and $\Phi_2(\boldsymbol\alpha,\boldsymbol\beta) := \mathcal F_2( \boldsymbol\alpha,\boldsymbol\beta,\boldsymbol0 )$.
    The estimates in \eqref{donotneediteration} give, locally uniformly, 
    \[
        \mathcal F_2(x,\boldsymbol r) = \Phi_2(x) + O(\boldsymbol r^2).
    \]
    Since $\mathcal F_2(x_2(\boldsymbol r),\boldsymbol r) = \mathcal F_2(x_2(\boldsymbol0),\boldsymbol0) $, we have
    \[
        |\Phi_2(x_2(\boldsymbol r)) - \Phi_2(x_2(\boldsymbol0))|_{\infty} \lesssim |\boldsymbol r|_{\infty}^2.
    \]
    Because \(D\Phi_2(x_2(\boldsymbol0))\) is invertible, \(\Phi_2\) has a
    locally Lipschitz inverse near \(x_2(\boldsymbol0)\). Consequently,
    \[
        |x_2(\boldsymbol r)-x_2(\boldsymbol0)|_{\infty} \lesssim |\boldsymbol r|_{\infty}^2.
    \]
    After the variable transformation \eqref{transform}, the functions $A_j^{(2)}(\boldsymbol r), B_j^{(2)}(\boldsymbol r) \in \mathbb{C}$ defined on $\boldsymbol r\in \mathcal{U}_{\boldsymbol 0}^{(2)}\subset \mathcal{N}_{\rho_0^{(2)}} \cap \mathbb{R}^g$ satisfy \eqref{restriction2} and
    \begin{equation*}
        A_j^{(2)}(\boldsymbol r) = \overline{B_j^{(2)}(\boldsymbol r)},\quad A_j^{(2)}(\boldsymbol r) = A_{j,0}^{(2)} + O(\boldsymbol r^2),\quad B_j^{(2)}(\boldsymbol r) = B_{j,0}^{(2)} + O(\boldsymbol r^2).
    \end{equation*}
\end{proof}
For \(\iota=1,2\), define the nondegenerate parameter region
\[
    \mathcal{N}_{\boldsymbol 0}^{(\iota)} := \left\{ \boldsymbol r\in\mathcal{U}_{\boldsymbol 0}^{(\iota)}: r_j \neq0 \textup{ for every } j\in\mathcal G \right\}.
\]
Then \(\mathcal N_{\boldsymbol0}^{(\iota)}\) is disconnected, open in \(\mathbb R^g\), and its closure contains \(\boldsymbol0\).
For each $\boldsymbol\delta = (\delta_1,\ldots,\delta_g)  \in \{\pm1\}^g$,
define the corresponding connected component
\[
    \mathcal N_{\boldsymbol0,\boldsymbol\delta}^{(\iota)} := \left\{ \boldsymbol r\in\mathcal N_{\boldsymbol0}^{(\iota)} : \delta_jr_j>0 \text{ for every }j\in\mathcal G \right\}.
\]
For $\iota=1,2$, let
\[
    \boldsymbol A^{(\iota)}(\boldsymbol r) = \bigl(
    A_1^{(\iota)}(\boldsymbol r), \cdots, A_g^{(\iota)}(\boldsymbol r) \bigr), \qquad \boldsymbol B^{(\iota)}(\boldsymbol r) = \bigl( B_1^{(\iota)}(\boldsymbol r), \cdots, B_g^{(\iota)}(\boldsymbol r) \bigr)
\]
be the functions obtained in
Lemma~\ref{solveuv}. For \(\boldsymbol r\in\mathcal N_{\boldsymbol0}^{(\iota)}\), define the parameterized period matrix
\[
    B^{(\iota)}(\boldsymbol r) := \bigg(B_{ij}\left( \boldsymbol A^{(\iota)}(\boldsymbol r), \boldsymbol B^{(\iota)}(\boldsymbol r), \boldsymbol r \right)\bigg)_{i,j\in\mathcal{G}},
\]
where its entries are defined by the extensions of \eqref{pmatrix1}, \eqref{pmatrix2} given in \eqref{parameter4'}, \eqref{parameter5'}. We also define
\[
    W_j^{(\iota)}(\boldsymbol r) := f_{3,j}\left( \boldsymbol A^{(\iota)}(\boldsymbol r), \boldsymbol B^{(\iota)}(\boldsymbol r), \boldsymbol r \right),\quad
    c^{(\iota)}(\boldsymbol r):=c\left( \boldsymbol A^{(\iota)}(\boldsymbol r), \boldsymbol B^{(\iota)}(\boldsymbol r), \boldsymbol r \right),\quad j\in\mathcal G.
\]
Set
\[
    \boldsymbol W^{(\iota)}(\boldsymbol r):=\bigl( W_1^{(\iota)}(\boldsymbol r), \cdots, W_g^{(\iota)}(\boldsymbol r) \bigr).
\]
Fix $\boldsymbol\delta\in \{\pm1\}^g$. Let
\[
    \mathscr D_+ := \mathbb C\setminus(-\infty,0], \qquad \mathscr D_- := \mathbb C\setminus[0,\infty),\qquad  \mathscr D_{\boldsymbol\delta} := \mathscr D_{\delta_1} \times\cdots\times \mathscr D_{\delta_g}.
\]
On each $\mathscr D_{\delta_j}$, one may choose a
holomorphic branch of \(\log r_j^2\) that agrees with the real-valued
function \(\log r_j^2\) on the half-axis \(\delta_jr_j>0\). Consequently, the diagonal logarithmic terms in \eqref{pmatrix2} admit single-valued holomorphic branches on a complex neighborhood $\mathcal V_{\boldsymbol\delta}^{(\iota)} \subset \mathscr D_{\boldsymbol\delta}$ of \(\mathcal N_{\boldsymbol0,\boldsymbol\delta}^{(\iota)}\).

After shrinking this complex neighborhood if necessary, the
real-analytic functions \(\boldsymbol A^{(\iota)}\) and \(\boldsymbol B^{(\iota)}\) admit holomorphic extensions to it, and the leading cross-ratios in \eqref{pmatrix1} remain nonzero.
For \(i\neq j\), define
\[
    \kappa_{ij}^{(\iota)}(\boldsymbol r) := \left\{ B_j^{(\iota)}(\boldsymbol r), A_j^{(\iota)}(\boldsymbol r), B_i^{(\iota)}(\boldsymbol r), A_i^{(\iota)}(\boldsymbol r) \right\},\quad
    \kappa_{ij,0}^{(\iota)} := \left\{ B_{j,0}^{(\iota)}, A_{j,0}^{(\iota)}, B_{i,0}^{(\iota)}, A_{i,0}^{(\iota)} \right\}.
\]
The base points in Lemma~\ref{solveuv} are pairwise distinct, hence $\kappa_{ij,0}^{(\iota)}\neq0$.
Moreover, by \eqref{distancewithAj0Bj02}, \eqref{distancewithAj0Bj0},
\[
    \kappa_{ij}^{(\iota)}(\boldsymbol r) = \kappa_{ij,0}^{(\iota)} + O\bigl(\boldsymbol r^2\bigr).
\]
Therefore, after reducing the parameter neighborhood once more, we
may assume that
\[
\kappa_{ij}^{(\iota)}
\bigl(
\mathcal V_{\boldsymbol\delta}^{(\iota)}
\bigr)
\subset
\left\{
z\in\mathbb C:
|z-\kappa_{ij,0}^{(\iota)}|
<
\frac12|\kappa_{ij,0}^{(\iota)}|
\right\}.
\]
This disk is simply connected and does not contain \(0\).
Consequently, a single-valued holomorphic branch may be chosen on it, and the leading logarithmic term in \eqref{pmatrix1} is holomorphic in \(\mathcal V_{\boldsymbol\delta}^{(\iota)}\). In particular, the same branch is used in the definition of the period matrix and in the constants introduced below.

The branches of the logarithms in the non-identity terms are those fixed in the proof of Lemma~\ref{convergelm}. Hence, by the locally uniform convergence established in that lemma and the Weierstrass theorem, each entry of \(B^{(\iota)}(\boldsymbol r)\) admits a holomorphic extension on \(\mathcal V_{\boldsymbol\delta}^{(\iota)}\).

We first consider the case $\iota=1$. 
For real parameters, Lemma~\ref{reallm} implies all non-identity logarithmic terms in \eqref{pmatrix1} are real valued. Therefore, the imaginary part of \(B_{ij}^{(1)}(\boldsymbol r)\) comes entirely from the leading term $\log \bigl( \kappa_{ij}^{(1)}(\boldsymbol r) \bigr)$.

Since $\kappa_{ij}^{(1)}(\boldsymbol r) = \kappa_{ij,0}^{(1)} + O\bigl(\boldsymbol r^2\bigr)$ and \(\kappa_{ij,0}^{(1)}\neq0\), the sign of \(\kappa_{ij}^{(1)}(\boldsymbol r)\) is constant after shrinking \(\mathcal U_{\boldsymbol0}^{(1)}\). Hence
$\textup{Im}\log \bigl( \kappa_{ij}^{(1)}(\boldsymbol r) \bigr)$ is constant on the real parameter region. Define
\[
    \epsilon_{jj}:=0, \qquad \epsilon_{ij} := \textup{Im}\log \kappa_{ij,0}^{(1)} ,\quad  E := \bigl( \epsilon_{ij} \bigr)_{i,j\in\mathcal G}.
\]
and
\[
    B^{(1)}_{real}(\boldsymbol r) := B^{(1)}(\boldsymbol r)-iE.
\]
Then $B^{(1)}_{real}(\boldsymbol r)$ is a negative definite matrix whose entries are real analytic on \(\mathcal N_{\boldsymbol0}^{(1)}\). Moreover,
\[
    B^{(1)}_{real,ij}(\boldsymbol r) = \log | \kappa_{ij,0}^{(1)} | + O\bigl(\boldsymbol r^2\bigr), \quad i\neq j,\quad 
    B_{real,jj}^{(1)}(\boldsymbol r) = \log r_j^2 + O(\boldsymbol r^2)\in \mathbb{R}.
\]
Similarly, the series \eqref{parameter3'} defining \(\boldsymbol{W}^{(1)}\) converges locally uniformly by Lemma~\ref{convergelm}. The same estimates apply to the series \eqref{implicitcondition6} defining \(c^{(1)}\). Hence the Weierstrass theorem, together with Lemma~\ref{reallm} and Lemma \ref{solveuv}, implies that $\boldsymbol W^{(1)}(\boldsymbol r), \ c^{(1)}(\boldsymbol r)$ are real analytic on  \(\boldsymbol r\in\mathcal N_{\boldsymbol0}^{(1)}\).

We next consider the case \(\iota=2\).
By \eqref{realmatrix1}--\eqref{realmatrix2}, \(B^{(2)}(\boldsymbol r)\) is real valued for real \(\boldsymbol r\). Set $B_{real}^{(2)}(\boldsymbol r) := B^{(2)}(\boldsymbol r)$. Since \(\boldsymbol\delta\in\{\pm1\}^g\) was arbitrary, the entries of \(B_{real}^{(2)}(\boldsymbol r)\) are real analytic on \(\mathcal N_{\boldsymbol0}^{(2)}\). Since \(B_{real}^{(2)}(\boldsymbol r)\) is a period matrix, it is symmetric and negative definite. Moreover, using \eqref{distancewithAj0Bj0} and the estimates in Lemma~\ref{convergelm}, we obtain
\[
    B_{real,ij}^{(2)}(\boldsymbol r)=\log\left\{B_{j,0}^{(2)}, A_{j,0}^{(2)}, B_{i,0}^{(2)}, A_{i,0}^{(2)} \right\} + O(\boldsymbol r^2), \quad i\neq j, \quad
    B_{real,jj}^{(2)}(\boldsymbol r) = \log r_j^2 + O(\boldsymbol r^2).
\]
Similarly, Lemma~\ref{reallm} and Lemma \ref{solveuv} yield that \( i\boldsymbol W^{(2)}(\boldsymbol r)\) and \(c^{(2)}(\boldsymbol r)\) are real analytic on \(\mathcal N_{\boldsymbol0}^{(2)}\).

Combining the restrictions \eqref{restriction1},\eqref{restriction2}, the definitions and
properties established in the preceding argument, and the theta-functional formulas in Corollary~\ref{cor:finite-gap-KP}, we obtain the following parameterization theorem.

\begin{theorem}\label{thm6.1}
Fix $g\in \mathbb{Z}_+$. If $g>1$, suppose that $\sqrt{3}\pi/\gamma\in\mathbb{R}\setminus\mathbb{Q}$. Set $\boldsymbol m = (m_1,\cdots, m_g), \boldsymbol n= (n_1,\cdots, n_g) \in \mathbb{Z}^g$, satisfying $m_1\cdots m_g\neq 0$ and $m_j n_k-m_k n_j\neq 0$, for $j,k\in\mathcal{G}$ and $j\neq k$.

Let $\kappa_{ij,0}^{(\iota)}, E, \ B^{(\iota)}_{real}(\boldsymbol r), \ B^{(\iota)}(\boldsymbol r),\ \boldsymbol W^{(\iota)}(\boldsymbol r),\ c^{(\iota)}(\boldsymbol r)$ be as defined in the preceding argument.
Set
\[
    \boldsymbol\omega_{0}^{(1)}(\boldsymbol r) = -12c^{(1)}(\boldsymbol r)\cdot\boldsymbol m - \frac{2\gamma}{\pi}\boldsymbol W^{(1)}(\boldsymbol r),\quad \boldsymbol\omega_{0}^{(2)}(\boldsymbol r) = 12c^{(2)}(\boldsymbol r)\cdot \boldsymbol m + \frac{2i\gamma}{\pi}\boldsymbol W^{(2)}(\boldsymbol r).
\]
For $\iota = 1,2$ and $\boldsymbol d \in \mathbb{R}^g$, denote $\boldsymbol\zeta^{(\iota)} = \boldsymbol m x+\boldsymbol ny + \boldsymbol \omega_0^{(\iota)}(\boldsymbol r)t + \boldsymbol d \in \mathbb{R}^g$.

$(i)$ There exists an open set $\mathcal{N}_{\boldsymbol 0}^{(1)}$ whose closure contains $\boldsymbol 0$ such that
\[
    B_{real}^{(1)}(\boldsymbol r)\in \mathbb{R}^{g\times g},
    \qquad
    \boldsymbol\omega_{0}^{(1)}(\boldsymbol r) \in \mathbb{R}^g
\]
depend real analytically on $\boldsymbol r\in \mathcal{N}_{\boldsymbol 0}^{(1)}$.
For every $\boldsymbol d\in \mathbb{R}^g$, the theta-functional expression is
\begin{equation}\label{onegap1}
    u_0^{(1)}(t,x,y ; \boldsymbol r)
    =
    12\frac{\partial^2}{\partial x^2}
    \log \theta\left(
        \frac{2\pi i}{\gamma}\boldsymbol\zeta^{(1)}; B^{(1)}_{real}(\boldsymbol r) + iE
    \right).
\end{equation}
Wherever it is well defined, it is a finite-gap solution of the KP-I equation \eqref{kp}.
Moreover, $B^{(1)}_{real}(\boldsymbol r)$ is negative definite, and for $i,j \in \mathcal{G}$,
\begin{equation}\label{B1property}
    B_{real,jj}^{(1)}(\boldsymbol r) = \log r_j^2 + O(\boldsymbol r^2),\quad B^{(1)}_{real,ij}(\boldsymbol r) = \log | \kappa_{ij,0}^{(1)} | + O\bigl(\boldsymbol r^2\bigr), \quad i\neq j.
\end{equation}

$(ii)$ There exists an open set $\mathcal{N}_{\boldsymbol 0}^{(2)}$ whose closure contains $\boldsymbol 0$ such that
\[
    B_{real}^{(2)}(\boldsymbol r)\in \mathbb{R}^{g\times g},
    \qquad
    \boldsymbol \omega_{0}^{(2)}(\boldsymbol r)\in \mathbb{R}^g
\]
depend real analytically on $\boldsymbol r\in \mathcal{N}_{\boldsymbol 0}^{(2)}$.
For every $\boldsymbol d\in \mathbb{R}^g$, the theta-functional expression is
\begin{equation}\label{onegap2}
    u_0^{(2)}(t,x,y ; \boldsymbol r)
    =
    12\frac{\partial^2}{\partial x^2}
    \log \theta\left(
        \frac{2\pi i}{\gamma}\boldsymbol\zeta^{(2)}; B_{real}^{(2)}(\boldsymbol r)
    \right).
\end{equation}
Wherever it is well defined, it is a finite-gap solution of the KP-II equation \eqref{kp}.
Moreover, $B_{real}^{(2)}(\boldsymbol r)$ is negative definite, and for $i,j \in \mathcal{G}$,
\begin{equation}\label{B2property}
    B^{(2)}_{real,jj}(\boldsymbol r)=\log r_j^2+O(\boldsymbol r^2),\quad B^{(2)}_{real,ij}(\boldsymbol r)=\log\kappa_{ij,0}^{(2)} + O(\boldsymbol r^2),\quad i\neq j.
\end{equation}
\end{theorem}

\begin{corollary}\label{thm6.1corollary}
    Assume the hypotheses of Theorem \ref{thm6.1}. Denote $\delta_+ = (1,\cdots,1) \in \mathbb{Z}^g$.
    For sufficiently small $\tilde\rho_0\in(0,1)$, the rectangle $(0,\tilde\rho_0)^g \subset \mathcal N_{\boldsymbol 0,\delta_+}^{(\iota)},\ \iota=1,2$.
    Let $\mathcal U_{\tilde\rho_0}^{(\iota)}$ be a symmetric complex
    neighborhood of $(0,\tilde\rho_0)^g$ such that
    \[
         \boldsymbol r \in \mathcal{U}_{\tilde\rho_0}^{(\iota)} \Leftrightarrow \overline{\boldsymbol r} \in \mathcal{U}_{\tilde\rho_0}^{(\iota)},\quad
         \mathcal U_{\tilde\rho_0}^{(\iota)} \subset \mathcal V_{\delta_+}^{(\iota)} \cap \left\{ \boldsymbol r = (r_1,\cdots,r_g) \in\mathbb C^g: 0<|r_j|<\tilde\rho_0,\ j\in\mathcal G \right\}.
    \]
    Hence functions $B_{real,ij}^{(\iota)}(\boldsymbol r),\ B_{ij}^{(\iota)}(\boldsymbol r),\ \omega_{j,0}^{(\iota)}(\boldsymbol r),\ i,j\in\mathcal{G}$ defined in Theorem \ref{thm6.1} are holomorphic in $\mathcal{U}_{\tilde\rho_0}^{(\iota)}$. For $\rho \in (0,\tilde\rho_0)$, denote $\mathcal{U}_{\rho} := \mathcal{U}_{\tilde\rho_0} \cap \{\boldsymbol r:0<|r_j|< \rho,j \in \mathcal{G}\}$.
    
    $(i)$ Fix $g = 1$. Then $\boldsymbol m, \boldsymbol n, \boldsymbol r, \boldsymbol\omega_{0}(\boldsymbol r), B^{(\iota)}(\boldsymbol r)$ are scalar-valued and we denote them by $m_0, n_0, r, \omega_{0}(r), B^{(\iota)}(r)$, respectively. Define $q^{(\iota)}(r):=\exp{\frac{B^{(\iota)}(r)}{2}},\ r \in \mathcal{U}_{\tilde\rho_0}^{(\iota)}$. Then $\omega_{0}^{(\iota)}(r), q^{(\iota)}(r)$ are holomorphic functions in $\mathcal{U}_{\tilde\rho_0}^{(\iota)}$ with real-analytic restrictions to $(0,\tilde\rho_0)$. Moreover, $q^{(\iota)}(r) = r + O(r^3)$ for $r \in \mathcal{U}_{\tilde\rho_0}^{(\iota)}$. Hence, there exists a constant $\rho_0 \in (0,\tilde \rho_0)$ such that $|q^{(\iota)}(r)| \in (0,1)$ for $r \in \mathcal{U}_{\rho_0}^{(\iota)}$.

    As a consequence, for $\iota = 1,2$, the theta-functional expression in Theorem \ref{thm6.1} is well defined and admits the following Fourier expansion on $(0,\rho_0):$
    \begin{align}
        u_{0}^{(\iota)}(t,x,y;r)= \tilde{u}_{0}^{(\iota)}(\zeta^{(\iota)};r)
        = \sum_{k\in\mathbb{Z}\setminus\{0\}} \hat{\tilde{u}}_{0}^{(\iota)}(k;r)e^{\frac{2\pi i}{\gamma}k\zeta^{(\iota)}}. \label{finite-gap}
    \end{align}
    For any $k \in \mathbb{Z}_+$, $\hat{\tilde{u}}_{0}^{(\iota)}(k;r)$ is real analytic on $(0,\rho_0)$ and
    \begin{equation*}
        \hat{\tilde{u}}_{0}^{(\iota)}(-k;r)=\hat{\tilde{u}}_{0}^{(\iota)}(k;r)=\frac{48m_0^2\pi^2}{\gamma^2}\cdot \frac{(-1)^{k}k(q^{(\iota)}(r))^{k}}{1-(q^{(\iota)}(r))^{2k}}.
    \end{equation*}
    \smallskip
    
    $(ii)$ Fix $g>1$.
    For $\boldsymbol k \in \mathbb{Z}^g \setminus \{0\}$, define
    \begin{align*}
        &q_j^{(\iota)}(\boldsymbol r) := \exp\left(\frac{1}{2}B_{real,jj}^{(\iota)}(\boldsymbol r)\right),\quad \Lambda^{(\iota)}(\boldsymbol k;\boldsymbol r) :=
        \exp{\left(\frac{1}{2} \boldsymbol k^{T} B^{(\iota)}(\boldsymbol r) \boldsymbol k \right)}, \quad \boldsymbol r\in \mathcal{U}_{\tilde\rho_0}^{(\iota)}.
    \end{align*}
    Hence $\omega_{j,0}^{(\iota)}(\boldsymbol r), q_j^{(\iota)}(\boldsymbol r), \Lambda^{(\iota)}(\boldsymbol k;\boldsymbol r)$ are holomorphic functions on $\mathcal{U}_{\tilde\rho_0}^{(\iota)}$ with real-analytic restrictions to $(0,\tilde\rho_0)^g$.
    There exist constants $\eta\in(0,1)$ and $\rho_0\in(0,\tilde\rho_0)$
    such that
    \begin{itemize}
        \item $|q_j^{(\iota)}(\boldsymbol r)| \leq \eta$ and $q_j^{(\iota)}(\boldsymbol r) = r_j + O(\boldsymbol r^3)$, for $\boldsymbol r\in \mathcal{U}_{\rho_0}^{(\iota)} := \mathcal{U}_{\tilde\rho_0}^{(\iota)} \cap \{\boldsymbol r:0<|r_j|< \rho_0, j\in\mathcal{G}\}$.
        \item For every
            $\boldsymbol k=(k_1,\cdots,k_g)\in\mathbb Z^g\setminus\{\boldsymbol0\}$,
            the series
            \begin{equation}
                \mathcal{L}_{\boldsymbol k}^{(\iota)}(\boldsymbol r) := \left(q_1^{(\iota)}\right)^{-|k_1|}\cdots \left(q_g^{(\iota)}\right)^{-|k_g|}\sum_{l=1}^{\infty} \frac{(-1)^{l}}{l} \left( \sum_{\substack{\boldsymbol n^{(1)} + \dots + \boldsymbol n^{(l)} = \boldsymbol k \\ \boldsymbol n^{(i)} \in \mathbb{Z}^g \setminus \{0\}}} \prod_{i=1}^l \Lambda^{(\iota)}(\boldsymbol n^{(i)};\boldsymbol r) \right), \label{mathcalLdefinition}
            \end{equation}
            converges absolutely and locally uniformly on $\mathcal{U}_{\rho_0}^{(\iota)}$, and determines a holomorphic function satisfying
            \[
                |\mathcal{L}_{\boldsymbol k}^{(\iota)}(\boldsymbol r)|
                \leq
                C_\eta\eta^{-|\boldsymbol k|_1},
                \qquad
                \boldsymbol r\in \mathcal{U}_{\rho_0}^{(\iota)},
            \]
            where $C_\eta$ is independent of $\boldsymbol k$ and
            $\boldsymbol r$. The restriction of $\mathcal{L}_{\boldsymbol k}^{(\iota)}(\boldsymbol r)$ to $(0,\rho_0)^g$ is real analytic.
    \end{itemize}
    As a consequence, for $\iota = 1,2$, the theta-functional expression in Theorem \ref{thm6.1} is well defined and admits the following Fourier expansion on $(0,\rho_0)^g:$
    \begin{equation*}
        u_0^{(\iota)}(t,x,y;\boldsymbol r)
        =
        \tilde {u}_0^{(\iota)}(\boldsymbol\zeta^{(\iota)};\boldsymbol r)
        =
        \sum_{\boldsymbol k\in\mathbb Z^g\setminus\{\boldsymbol0\}}
        \widehat{\widetilde u}_0^{(\iota)}(\boldsymbol k;\boldsymbol r)
        e^{\frac{2\pi i}{\gamma}
           \boldsymbol k\cdot \boldsymbol\zeta^{(\iota)}}.
    \end{equation*}
    For any $\boldsymbol k=(k_1,\cdots,k_g) \in \mathbb{Z}^g$, $\widehat{\widetilde u}_0^{(\iota)}(\boldsymbol k;\boldsymbol r)$ is real analytic on $(0,\rho_0)^g$ and
    \begin{equation*}
        \widehat{\widetilde u}_0^{(\iota)}(-\boldsymbol k;\boldsymbol r) = \widehat{\widetilde u}_0^{(\iota)}(\boldsymbol k;\boldsymbol r)  =  \frac{48\pi^2}{\gamma^2} (m_1k_1+\cdots + m_gk_g)^2
        \left(q_1^{(\iota)}(\boldsymbol r)\right)^{|k_1|}\cdots  \left(q_g^{(\iota)}(\boldsymbol r)\right)^{|k_g|} \mathcal{L}_{\boldsymbol k}^{(\iota)}(\boldsymbol r).
    \end{equation*}
\end{corollary}

\begin{proof}
$(i)$
The existence of symmetric complex neighborhood $\mathcal{U}_{\tilde\rho_0}^{(\iota)}$ and functions $q^{(\iota)}(r), \omega_{0}^{(\iota)}(r)$ with the stated properties follows from the complex-neighborhood construction preceding Theorem~\ref{thm6.1} and Theorem~\ref{thm6.1}.
Therefore, we only need to derive the form \eqref{finite-gap}.

Consider the Jacobi theta function for $z=\frac{2\pi i}{\gamma}\zeta^{(\iota)}$,
\begin{equation*}
    \theta(z;B^{(\iota)})=\sum_{n\in\mathbb{Z}}e^{\frac{B(r)}{2} n^2+ nz}=\sum_{n\in\mathbb{Z}}\left(q^{(\iota)}(r)\right)^{n^2}e^{nz},z\in i\mathbb{R}.
\end{equation*}
By \eqref{B1property}--\eqref{B2property}, $q^{(\iota)}(r) = r + O(r^3)$ on $\mathcal{U}_{\tilde\rho_0}^{(\iota)}$. Thus, there exists a constant $\rho_0 \in (0,\tilde\rho_0)$ such that $\left|q^{(\iota)}(r)\right| <1$ for $r \in \mathcal{U}_{\rho_0}^{(\iota)}$.
According to the Jacobi triple product identity,
\begin{equation*}
\theta(z; B^{(\iota)}) = \prod_{n=1}^{\infty} \left(1-\left(q^{(\iota)}\right)^{2n} \right) \left(1+\left(q^{(\iota)}\right)^{2n-1}e^{z} \right) \left( 1+\left(q^{(\iota)}\right)^{2n-1}e^{-z} \right).
\end{equation*}
Thus
\begin{equation*}
\log \theta(z; B^{(\iota)}) = \sum_{n=1}^{\infty} \log \left(1-\left(q^{(\iota)}\right)^{2n}\right) + \sum_{n=1}^{\infty} \log\left(1 + \left(q^{(\iota)}\right)^{2n-1}e^{z} \right) + \sum_{n=1}^{\infty} \log \left(1 + \left(q^{(\iota)}\right)^{2n-1}e^{-z} \right).
\end{equation*}
With
\begin{equation*}
    \log(1+p)=\sum_{k=1}^{\infty} \frac{(-1)^{k-1}}{k}p^k, |p|<1,
\end{equation*}
a simple calculation yields
\begin{align*}
\log \theta(z; B^{(\iota)}) &= \sum_{n=1}^{\infty} \log \left(1-\left(q^{(\iota)}\right)^{2n} \right) + 2 \sum_{k=1}^{\infty} \frac{(-1)^{k-1}}{k} \left( \sum_{n=1}^{\infty} \left(q^{(\iota)}\right)^{k(2n-1)} \right) \cos\left(\frac{2\pi i}{\gamma}k\zeta^{(\iota)}\right)\\
&=\sum_{n=1}^{\infty} \log \left(1-\left(q^{(\iota)}\right)^{2n} \right) + 2 \sum_{k=1}^{\infty} \frac{(-1)^{k-1}}{k} \frac{\left(q^{(\iota)}\right)^k}{1-\left(q^{(\iota)}\right)^{2k}} \cos\left(\frac{2\pi i}{\gamma}k\zeta^{(\iota)}\right).
\end{align*}
Recalling \eqref{onegap1},\eqref{onegap2} and differentiating term by term, we obtain
\begin{align*}
u_0^{(\iota)}(t,x,y;r) &= 12\left(\frac{2\pi}{\gamma}m_0\right)^2 \cdot 2 \sum_{k=1}^{\infty} \frac{(-1)^{k-1}}{k} \frac{\left(q^{(\iota)}\right)^k}{1-\left(q^{(\iota)}\right)^{2k}} \left( -k^2 \cos\left(\frac{2\pi i}{\gamma}k\zeta^{(\iota)}\right) \right)\\
&= \frac{96m_0^2\pi^2}{\gamma^2} \sum_{k=1}^{\infty} \frac{(-1)^{k} k \left(q^{(\iota)}\right)^k}{1-\left(q^{(\iota)}\right)^{2k}} \cos\left(\frac{2\pi i}{\gamma}k\zeta^{(\iota)}\right)
=\sum_{k\in\mathbb{Z}\setminus\{0\}} \hat{\tilde{u}}_{0}^{(\iota)}(k;r)e^{\frac{2\pi i}{\gamma} k \zeta^{(\iota)}}.
\end{align*}
These coefficients are real and even in $k$, and they decay exponentially as \(|k|\to\infty\).
For every fixed \(k\neq0\), the function $\hat{\tilde{u}}_{0}^{(\iota)}(k;r)$ is therefore holomorphic in \(\mathcal U_{\rho_0}^{(\iota)}\), and its restriction to \((0,\rho_0)\) is real analytic.

\bigskip

$(ii)$
The existence of $\mathcal{U}_{\tilde\rho_0}^{(\iota)}$ and holomorphicity of $q_j^{(\iota)}(\boldsymbol r), \Lambda^{(\iota)}(\boldsymbol k;\boldsymbol r)$ on $\mathcal{U}_{\tilde\rho_0}^{(\iota)}$ also follows from Theorem \ref{thm6.1}. By \eqref{B1property}--\eqref{B2property}, we have
\begin{equation}\label{qjiotaestimates}
    q_j^{(\iota)}(\boldsymbol r)
    =
    r_j\left(1+O(\boldsymbol r^2)\right).
\end{equation}

Set $M > \sup_{i< j}\sup_{\boldsymbol r\in \overline{\mathcal U}_{\tilde\rho_0}} |B_{real,ij}^{(\iota)}(\boldsymbol r)|$ for $\iota = 1,2$, and $\eta \in(0,e^{-\frac{M(g-1)}{2}})$ such that
\begin{equation}
    \sum_{k=1}^{g} \binom{g}{k}e^{\frac{M(k-1)k}{2}} (2\eta)^k < \frac12.\label{eta-small-normalizedg}
\end{equation}
Then, there exists a constant $\rho_0$ such that $|q_j^{(\iota)}(\boldsymbol r)|<\eta<e^{-\frac{M(g-1)}{2}},\ \boldsymbol r \in \mathcal{U}_{\rho_0}^{(\iota)}$. Denote $Q:=\sup_{\boldsymbol r\in \overline{\mathcal U}_{\rho_0}}\max\{|q_1^{(\iota)}(\boldsymbol r)|,\cdots,|q_g^{(\iota)}(\boldsymbol r)|\}$, which satisfies $Q<e^{-\frac{M(g-1)}{2}}$.

Notice that $\Lambda^{(\iota)}(\boldsymbol n;\boldsymbol r) = \left( q_1^{(\iota)}(\boldsymbol r) \right)^{n_1^2}\cdots \left( q_g^{(\iota)}(\boldsymbol r) \right)^{n_g^2} e^{\sum_{i<j}B_{ij}^{(\iota)}(\boldsymbol r)n_in_j}$.
Let
\[
    w_{\boldsymbol n}^{(\iota)}(\boldsymbol r)
    :=
    \eta^{|\boldsymbol n|_1}
    |q_1^{(\iota)}|^{n_1^2-|n_1|}\cdots
    |q_g^{(\iota)}|^{n_g^2-|n_g|}
    e^{M\sum_{i< j} |n_in_j|},\quad
    \mathfrak R_\eta^{(\iota)}(\boldsymbol r)
    :=
    \sum_{\boldsymbol n\neq\boldsymbol0}
    w_{\boldsymbol n}^{(\iota)}(\boldsymbol r).
\]
Since
\[
    n_j^2-|n_j| \geq \frac12n_j^2-\frac12,\qquad
    \sum_{i< j}|n_in_j| \leq \sum_{i< j}\frac12(n_i^2+n_j^2) \leq \frac{g-1}{2}\sum_{j\in\mathcal{G}}|n_j|^2,
\]
we obtain
\[
\begin{aligned}
    w_{\boldsymbol n}^{(\iota)}(\boldsymbol r) &\leq \eta^{|\boldsymbol n|_1} Q^{n_1^2-|n_1|+\cdots+n_g^2-|n_g|} e^{M\sum_{i< j}|n_in_j|}\\
    &\leq Q^{-g/2} \exp\left( -\frac{-\log Q-\frac{M(g-1)}{2}}{2} (n_1^2+\cdots+n_g^2) \right).
\end{aligned}
\]
The right-hand side is a summable Gaussian sequence independent of
$\boldsymbol r$. Hence $\mathfrak R_\eta^{(\iota)}$ converges absolutely and locally uniformly on
$\mathcal U_{\rho_0}^{(\iota)}$.
Notice that all terms with $\max\{|n_1|,\cdots,|n_g|\}\geq2$ tend to zero as $\boldsymbol r\to\boldsymbol0$. Thus,
\[
    \lim_{\boldsymbol r\to\boldsymbol0} \mathfrak R_\eta^{(\iota)}(\boldsymbol r) = \sum_{k=1}^{g} \binom{g}{k}e^{\frac{M(k-1)k}{2}} (2\eta)^k.
\]
It follows from \eqref{qjiotaestimates}--\eqref{eta-small-normalizedg} that, after decreasing
$\rho_0$, there exists $\varrho_*\in(0,1)$ such that
\begin{equation}
\label{R-eta}
    |\mathfrak R_\eta^{(\iota)}(\boldsymbol r)| \leq \varrho_*<1,
\end{equation}

Fix $\boldsymbol k=(k_1,k_2,\cdots,k_g)\in\mathbb Z^g\setminus\{\boldsymbol0\}$. For every decomposition $\boldsymbol n^{(1)}+\cdots+\boldsymbol n^{(l)} = \boldsymbol k$, define
\[
    N_j := \sum_{\nu=1}^l \left(n_j^{(\nu)}\right)^2-|k_j|,\quad
    A_j := \sum_{\nu=1}^l\left(n_j^{(\nu)}\right)^2,\quad
    S_j := \sum_{\nu=1}^l|n_j^{(\nu)}|, \quad j\in\mathcal{G},
\]
Then, we have
\begin{align*}
    T^{(\iota)} &:= \left|\left(q_1^{(\iota)}\right)^{N_1}\cdots \left(q_g^{(\iota)}\right)^{N_g}
        \exp\left( \sum_{i< j}B_{ij}^{(\iota)}(\boldsymbol r) \sum_{\nu=1}^l n_i^{(\nu)}n_j^{(\nu)} \right)\right|\\
        &\leq |q_1^{(\iota)}|^{A_1-|k_1|}\cdots |q_g^{(\iota)}|^{A_g-|k_g|} e^{M\sum_{i<j}\sum_\nu|n_i^{(\nu)}n_j^{(\nu)}|},\\
    \prod_{\nu=1}^l w_{\boldsymbol n^{(\nu)}}^{(\iota)} &= \eta^{S_1+\cdots+S_g}|q_1^{(\iota)}|^{A_1-S_1}\cdots|q_g^{(\iota)}|^{A_g-S_g}e^{M\sum_{i<j}\sum_\nu|n_i^{(\nu)}n_j^{(\nu)}|},
\end{align*}
and
\[
    \frac{T^{(\iota)}}{\prod_{\nu=1}^l w_{\boldsymbol n^{(\nu)}}^{(\iota)}} \leq \eta^{-(S_1+\cdots +S_g)}|q_1^{(\iota)}|^{S_1-|k_1|}\cdots |q_g^{(\iota)}|^{S_g-|k_g|} \leq \eta^{-|\boldsymbol k|_1}.
\]
Consequently,
\begin{align*}
\sum_{\substack{
    \boldsymbol n^{(1)}+\cdots+\boldsymbol n^{(l)}
    =\boldsymbol k\\
    \boldsymbol n^{(\nu)}\neq\boldsymbol0
}}
\left|
    \left( q_1^{(\iota)}\right)^{N_1}\cdots \left(q_g^{(\iota)}\right)^{N_g}
    e^{\sum_{i<j}B_{ij}^{(\iota)}\sum_\nu
       n_i^{(\nu)}n_j^{(\nu)}}
\right|
&\leq
\eta^{-|\boldsymbol k|_1}
\sum_{\substack{
    \boldsymbol n^{(1)}+\cdots+\boldsymbol n^{(l)}
    =\boldsymbol k\\
    \boldsymbol n^{(\nu)}\neq\boldsymbol0
}}
\prod_{\nu=1}^l
w_{\boldsymbol n^{(\nu)}}^{(\iota)}\\
&\leq
\eta^{-|\boldsymbol k|_1}
\left(
    \sum_{\boldsymbol n\neq\boldsymbol0}
    w_{\boldsymbol n}^{(\iota)}
\right)^l
\leq
\eta^{-|\boldsymbol k|_1}\varrho_*^l.
\end{align*}
Therefore the series \eqref{mathcalLdefinition} converges
absolutely and locally uniformly on $\mathcal U_{\rho_0}^{(\iota)}$, and
\begin{equation*}
    |\mathcal{L}_{\boldsymbol k}^{(\iota)}(\boldsymbol r)|
    \leq
    \eta^{-|\boldsymbol k|_1}
    \sum_{l=1}^\infty\frac{\varrho_*^l}{l}
    =
    -\log(1-\varrho_*)
    \eta^{-|\boldsymbol k|_1}.
\end{equation*}
Notice that every summand in \eqref{mathcalLdefinition} is holomorphic, since $N_1,N_2,\cdots,N_g$ are nonnegative integers. Hence, the Weierstrass theorem implies that $\mathcal L_{\boldsymbol k}^{(\iota)}(\boldsymbol r)$ is holomorphic in $\mathcal U_{\rho_0}^{(\iota)}$ and its restriction to $(0,\rho_0)^g$ is real analytic.

Denote $\boldsymbol z=\frac{2\pi i}{\gamma}\boldsymbol\zeta\in i\mathbb{T}_{\gamma}^g$. The Riemann theta function appearing in Theorem \ref{thm6.1} is
\begin{align*}
    \theta(\boldsymbol z; B^{(\iota)}(\boldsymbol r)) &= \sum_{\boldsymbol k \in \mathbb{Z}^g} \exp \left( \frac{1}{2} \boldsymbol k^TB^{(\iota)}(\boldsymbol r)\boldsymbol k + \frac{2\pi i}{\gamma} \boldsymbol k\cdot \boldsymbol\zeta \right)\\
    &= 1+\sum_{\boldsymbol k \in \mathbb{Z}^g\setminus\{0\}} \Lambda^{(\iota)}(\boldsymbol k;\boldsymbol r) e^{\frac{2\pi i}{\gamma}\boldsymbol k\cdot\boldsymbol\zeta}.
\end{align*}
With \eqref{R-eta}, we have
\begin{align*}
    \log \theta(\boldsymbol z; B^{(\iota)}(\boldsymbol r)) &= \sum_{l=1}^{\infty} \frac{(-1)^{l-1}}{l} \left(\sum_{\boldsymbol k \in\mathbb{Z}^g\setminus\{0\}}\Lambda^{(\iota)}(\boldsymbol k;\boldsymbol r)e^{\frac{2\pi i}{\gamma}\boldsymbol k\cdot\boldsymbol\zeta}\right)^l,\\
    &= -\sum_{\boldsymbol k = (k_1,\cdots,k_g) \in \mathbb{Z}^g \setminus \{0\}} \left(q_1^{(\iota)}\right)^{|k_1|}\cdots \left(q_g^{(\iota)}\right)^{|k_g|}\mathcal{L}_{\boldsymbol k}^{(\iota)}(\boldsymbol r) e^{\frac{2\pi i}{\gamma}\boldsymbol k\cdot\boldsymbol \zeta}.
\end{align*}
Applying the operator $12\partial_x^2$ where $\partial_x= m_1\partial_{\zeta_1} + \cdots +m_g\partial_{\zeta_g}$, we have
\begin{equation*}
u_0^{(\iota)}(t,x,y ; \boldsymbol r) = \sum_{(k_1, \cdots, k_g) \in \mathbb{Z}^g \setminus \{0\}} \frac{48\pi^2}{\gamma^2} (k_1 m_1 + \cdots + k_g m_g)^2 \left(q_1^{(\iota)}\right)^{|k_1|} \cdots \left(q_g^{(\iota)}\right)^{|k_g|} \mathcal{L}_{\boldsymbol k}^{(\iota)}(\boldsymbol r) e^{\frac{2\pi i}{\gamma}\boldsymbol k\cdot\boldsymbol \zeta}.
\end{equation*}
The change $\boldsymbol n^{(\nu)} \mapsto -\boldsymbol n^{(\nu)}$ shows that
\[
    \mathcal L_{-\boldsymbol k}^{(\iota)}(\boldsymbol r) = \mathcal L_{\boldsymbol k}^{(\iota)}(\boldsymbol r),\quad
    \widehat{\widetilde u}_0^{(\iota)}(-\boldsymbol k;\boldsymbol r) = \widehat{\widetilde u}_0^{(\iota)}(\boldsymbol k;\boldsymbol r).
\]
Set $\delta_{ij} := \dfrac{\epsilon_{ij}}{\pi} = \begin{cases}
    0,\quad \kappa_{ij,0}^{(1)} > 0,\\
    1,\quad \kappa_{ij,0}^{(1)} < 0.
\end{cases}$
Moreover, for real $\boldsymbol r$,
\[
    e^{B_{ij}^{(1)}n_in_j} = (-1)^{\delta_{ij} n_in_j} e^{B_{real,ij}^{(1)}(\boldsymbol r)n_in_j} \in\mathbb R,\quad e^{B_{ij}^{(2)}n_in_j} = e^{B_{real,ij}^{(2)}(\boldsymbol r)n_in_j} \in\mathbb R.
\]
Thus $\mathcal L_{\boldsymbol k}^{(\iota)}(\boldsymbol r)$ and $\widehat{\widetilde u}_0^{(\iota)}(\boldsymbol k;\boldsymbol r)$ are real analytic on $(0,\rho_0)^g$.
\end{proof}

\section{Persistence of one-gap solutions}\label{sec:one-gap}

In this section, we apply the parameterization obtained in the preceding section to study the persistence of one-gap solutions under Hamiltonian perturbations. We use the one-gap solutions obtained in Theorem~\ref{thm6.1} and Corollary~\ref{thm6.1corollary}.

\subsection{Lyapunov-Schmidt reduction}
\begin{definition}
    We denote by $\ell_e^2(\mathbb{T}_{\gamma})$ the complex Hilbert space of zero-$x$-mean even Fourier coefficient sequences
    \[
        \ell_e^2(\mathbb{T}_\gamma) := \left\{ \widehat u = \left(\hat{u}(k) \right)_{k \in \mathbb{Z}}: \widehat u(0)=0,\ \widehat u(-k)=\widehat u(k) \text{ for every }k\in\mathbb Z,\ \sum_{k\in\mathbb Z} |\widehat u(k)|^2<\infty \right\}.
    \]
     We also define its real subspace by $\ell_{re}^2(\mathbb{T}_\gamma) := \left\{ \widehat u\in\ell_e^2(\mathbb{T}_\gamma): \widehat u(k)\in\mathbb R \text{ for every }k\in\mathbb Z \right\}$.
\end{definition}
Notice that when $r \in (0,\rho_0)$, $\hat{\tilde{u}}_{0}^{(\iota)}(-k) = \hat{\tilde{u}}_{0}^{(\iota)}(k) \in \mathbb{R}$ and $\hat{\tilde{u}}_{0}^{(\iota)}(0) = 0$, which implies that $\hat{\tilde{u}}_0^{(\iota)} \in \ell_{re}^2(\mathbb{T}_{\gamma})$. We consequently seek solutions to the perturbed KP equation in the form $u^{(\iota)}(t,x,y) = \tilde{u}^{(\iota)}(\omega^{(\iota)} t+ m_0x + n_0y + d^{(\iota)})$, where $\hat{\tilde{u}}^{(\iota)} \in \ell_{re}^2(\mathbb{T}_{\gamma})$. For the analytic arguments below, we complexify this problem and regard \(\hat{\tilde u}^{(\iota)}\) as an element of \(\ell_e^2(\mathbb{T}_\gamma)\). To simplify notation, we abuse notation and let $u_0,u,\hat{u}_0,\hat{u}$ also denote $\tilde{u}_0,\tilde{u},\hat{\tilde{u}}_0,\hat{\tilde{u}}$ in the sequel.

By \eqref{eq:intro-P1}, let
\[
    P_1(u)
    :=
    Q_1(J^2u)\big|_{\partial_x=m_0\partial_{\zeta},\,
    \partial_y=n_0\partial_{\zeta}}.
\]
Then
\[
    P_1(u)
    =
    u^p+u^{p+1}+P_{1,>p+1}(u),
\]
where
\[
    P_{1,>p+1}(u)
    :=
    Q_{1,>p+1}
    \left(
        u,m_0u_{\zeta},n_0u_{\zeta},
        m_0^2u_{\zeta\zeta},
        m_0n_0u_{\zeta\zeta},
        n_0^2u_{\zeta\zeta}
    \right).
\]
Consequently, $\mathcal{P}_1[u]=m_0\partial_{\zeta}P_1(u)$, and the perturbed KP equation reduces to

\begin{equation}\label{finitegap-kp}
    \omega u_{\zeta}+ m_0^3 \partial_{\zeta}^3 u-\lambda_{\iota}\frac{n_0^2}{m_0}u_{\zeta} + m_0 uu_{\zeta}+ m_0\partial_{\zeta}P_1(u) =0.
\end{equation}
Passing to the Fourier coefficients, equation \eqref{finitegap-kp} is equivalent to
\begin{equation}\label{Fouriereq1}
    \frac{2\pi i}{\gamma}\left(k\omega-\frac{4\pi^2}{\gamma^2}m_0^3k^3-\frac{\lambda_{\iota} n_0^2}{m_0}k\right)\hat{u}(k)+\frac{2\pi i}{\gamma}\frac{m_0 k}{2}\widehat{u^2}(k)+ \frac{2\pi im_0k}{\gamma}\widehat{P_1(u)}(k)=0,k\in\mathbb{Z}\setminus\{0\}.
\end{equation}
Notice that the parity condition \eqref{eq:intro-P1-parity} together with the reality of the Taylor coefficients, imply that $\widehat{P_1(u)}$ depends real analytically on $\hat{u}$ and preserves the real subspace. Thus, if $(\omega,\hat{u}) \in \mathbb{C}\times \ell_{re}^2(\mathbb{T}_{\gamma})$ solves \eqref{Fouriereq1} and $\hat{u}(1) \neq 0$, then $\omega\in\mathbb{R}$.

Now we derive the leading-order expansion of $\omega_0^{(\iota)}(r)$.
Denote $\widehat{u_{0}^{(\iota)}(r)} = \left( \widehat{u_{0}^{(\iota)}}(k;r) \right)_{k \in \mathbb{Z}\setminus\{0\}}$.
For every $r\in (0,\rho_0)$, Corollary~\ref{thm6.1corollary} implies that $\widehat{u_0^{(\iota)}(r)}, \omega_0(r)$ satisfy the Fourier form of the unperturbed KP equation
\begin{equation}
    \frac{2\pi i}{\gamma} \left(k\omega-\frac{4\pi^2}{\gamma^2}m_0^3k^3-\frac{\lambda_{\iota} n_0^2}{m_0}k\right)\hat{u}(k)+\frac{2\pi i}{\gamma}\frac{m_0 k}{2}\sum_{k_1+k_2 = k}\hat{u}(k_1)\hat{u}(k_2)=0, |k|\geq 1. \label{kp-fourier1}
\end{equation}
Denote the left hand side of \eqref{kp-fourier1} for $k\in \mathbb{Z}\setminus\{0\}$ by $F_k^{(\iota)}(\hat{u},\omega)$.
The calculation at $|k|=1$ yields $\omega_0^{(\iota)}(r) \in \mathbb{R}$, and for $r \in (0,\rho_0)$,
\begin{equation}\label{omega0jianjin}
\omega_0^{(\iota)}(r) = \frac{4\pi^2}{\gamma^2}m_0^3 + (-1)^{\iota-1}\frac{n_0^2}{m_0} - \frac{96m_0^3\pi^2}{\gamma^2} r^2+ O(r^4).
\end{equation}
Notice that $\widehat{u_{0}^{(\iota)}}(k;r)$ is a well-defined holomorphic function on $\mathcal{U}_{\rho_0}^{(\iota)}$, and the series $\sum_{k_1+k_2 = k}\allowbreak \widehat{u_0^{(\iota)}}(k_1;r)\widehat{u_0^{(\iota)}}(k_2;r)$ converges absolutely and locally uniformly on $\mathcal{U}_{\rho_0}^{(\iota)}$.
Thus, the functions $F_k^{(\iota)}(\widehat{u_{0}^{(\iota)}(r)}, \omega_0^{(\iota)}(r))$ admit well-defined holomorphic extensions on the uniform complex neighborhood $\mathcal{U}_{\rho_0}^{(\iota)}$ for any $k\in\mathbb{Z}\setminus \{0\}$.
As a consequence, $\widehat{u_0^{(\iota)}(r)}, \omega_0^{(\iota)}(r)$ also satisfy the unperturbed KP equation \eqref{kp-fourier1} on $\mathcal{U}_{\rho_0}^{(\iota)}$.
Thus, a similar calculation implies that \eqref{omega0jianjin} hold on $\mathcal{U}_{\rho_0}^{(\iota)}$, and there exists a constant $C_0$ such that
\begin{equation}\label{omega0}
    \left|\omega_0^{(\iota)}(r)- \frac{4\pi^2}{\gamma^2}m_0^3 - (-1)^{\iota-1}\frac{n_0^2}{m_0}\right|\leq C_0|r|^2,\quad r \in \mathcal{U}_{\rho_0}^{(\iota)}
\end{equation}
We set the resonant set of \eqref{Fouriereq1} as
\begin{align*}
    S_1 &= \left\{ k\in \mathbb{Z}\setminus \{0\}:k\left(\frac{4\pi^2}{\gamma^2}m_0^3+(-1)^{\iota-1} \frac{n_0^2}{m_0}\right)=\frac{4\pi^2}{\gamma^2}k^3m_0^3+(-1)^{\iota-1} \frac{n_0^2}{m_0}k\right\}= \{1,-1\}.
\end{align*}
By restricting the perturbed KP equation \eqref{Fouriereq1} to $S_1$ and $\mathbb{Z}\setminus (S_1\cup \{0\})$, we obtain the bifurcation equation and the range equation as follows:
\begin{align}
    &\frac{2\pi i}{\gamma} \left(k\omega-\frac{4\pi^2}{\gamma^2}m_0^3k^3-\frac{\lambda_{\iota} n_0^2}{m_0}k\right)\hat{u}(k)+ \frac{2\pi i}{\gamma}\frac{m_0 k}{2}\widehat{u^2}(k)+ \frac{2\pi i m_0k}{\gamma} \widehat{P_1(u)}(k)=0,|k|=1,\label{qfangcheng1}\\
    &\frac{2\pi i}{\gamma} \left(k\omega-\frac{4\pi^2}{\gamma^2}m_0^3k^3-\frac{\lambda_{\iota} n_0^2}{m_0}k\right)\hat{u}(k)+ \frac{2\pi i}{\gamma} \frac{m_0 k}{2}\widehat{u^2}(k)+ \frac{2\pi i m_0 k}{\gamma} \widehat{P_1(u)}(k)=0,|k|\geq 2\label{pfangcheng1}.
\end{align}

\subsection{Solving the range equation}

In this section, we solve the range equations associated with both KP-I and KP-II under the Hamiltonian perturbation $\mathcal{P}_1$.
To start with, we introduce the following spaces.

\begin{definition}
    Let $\sigma>0, s,s_1,s_2>\frac{1}{2}$. We define the following spaces:
    \begin{itemize}
        \item Denote $\nu_1 = \frac{2\pi}{\gamma}(|m_0|+|n_0|)$. Let the isotropic analytic Sobolev space $H^{\sigma,s,\nu_1} := \{\hat{u}\in \ell_e^2(\mathbb{T}_{\gamma}): \hat{u}(0) = 0, \|\hat{u}\|_{\sigma,s,\nu_1} < \infty\}$, where 
        \begin{equation*}
            \|\hat{u}\|_{\sigma,s,\nu_1}^2 = \sum_{k\in \mathbb{Z}} |\hat{u}(k)|^2 e^{2\sigma \nu_1|k|} \langle \nu_1 k\rangle^{2s}.
        \end{equation*}
        Equipped with this norm, $H^{\sigma,s,\nu_1}$ is a complex Banach space.
        \item 
        Let $P_{\mathbb Z\setminus(S_1\cup\{0\})}$ be the Fourier projection defined by 
        \[ 
            (P_{\mathbb Z\setminus(S_1\cup\{0\})}\hat u)(k) = 
            \begin{cases} 
                \hat u(k), & k\in\mathbb Z\setminus(S_1\cup\{0\}),\\ 
                0, & k\in S_1\cup\{0\}.
            \end{cases}
        \]
        We set
        \[ 
            P\ell_e^2(\mathbb T_\gamma) := \{P_{\mathbb Z\setminus(S_1\cup\{0\})}\hat u:\hat u\in\ell_e^2(\mathbb T_\gamma)\}, \qquad PH^{\sigma,s,\nu_1} := \{P_{\mathbb Z\setminus(S_1\cup\{0\})}\hat u:\hat u\in H^{\sigma,s,\nu_1}\}.
        \]
        For every $\hat u\in H^{\sigma,s,\nu_1}$, define the projected seminorm
        \[ 
            \|\hat u\|_{\sigma,s,\nu_1,P} := \|P_{\mathbb Z\setminus(S_1\cup\{0\})}\hat u\|_{\sigma,s,\nu_1}. 
        \]
        Its restriction to $PH^{\sigma,s,\nu_1}$ is a norm and satisfies $\|\hat u\|_{\sigma,s,\nu_1,P} = \|\hat u\|_{\sigma,s,\nu_1},\ \hat u\in PH^{\sigma,s,\nu_1}.$ Since $P_{\mathbb Z\setminus(S_1\cup\{0\})}$ is a bounded projection, $PH^{\sigma,s,\nu_1}$ is a closed subspace of $H^{\sigma,s,\nu_1}$ and hence is a Banach space with respect to this norm.
        We also define its real subspace by
        \[
            PH_{\mathrm{re}}^{\sigma,s,\nu_1} := PH^{\sigma,s,\nu_1} \cap  \ell_{\mathrm{re}}^2(\mathbb T_\gamma).
        \]
        \item Let the bounded linear operator space $\mathcal{L}(PH^{\sigma,s_1,\nu_1}, PH^{\sigma,s_2,\nu_1})$ be the Banach space with norm
        \begin{equation*}
            \|T\|_{\sigma,s_1,s_2,\nu_1,P} = \sup_{0\neq \hat{u} \in PH^{\sigma,s_1,\nu_1}} \frac{\| T\hat{u}\|_{\sigma,s_2,\nu_1,P}}{\|\hat{u}\|_{\sigma,s_1,\nu_1,P}}.
        \end{equation*}
    \end{itemize}
\end{definition}
\begin{remark}
    Fix $s,\nu_1,\rho_0$ and $\sigma >0$. Since $q^{(\iota)}(r) = r + O(r^3)$, after decreasing $\rho_0$ if necessary, $2(2\nu_1)^{s}e^{\sigma\nu_1}|q^{(\iota)}(r)| \leq 4(2\nu_1)^{s}e^{\sigma\nu_1}\rho_0 <1$ for $r \in \mathcal{U}_{\rho_0}^{(\iota)}$.
    Thus, for every $r \in \mathcal{U}_{\rho_0}^{(\iota)}$,
    \begin{align}
        &\|\hat{u}_0^{(\iota)}(r)\|_{\sigma,s,\nu_1}^2 = \sum_{|k|\geq 1} |\hat{u}_0^{(\iota)}(k;r)|^2 e^{2\sigma\nu_1|k|}\langle \nu_1 k \rangle^{2s} \lesssim |r|^2,\label{u0estimate}\\ 
        &\|\hat{u}_0^{(\iota)}(r)\|_{\sigma,s,\nu_1,P}^2 = \sum_{|k|\geq 2} |\hat{u}_0^{(\iota)}(k;r)|^2 e^{2\sigma\nu_1|k|}\langle \nu_1 k \rangle^{2s} \lesssim |r|^4.\label{u0estimateP}
    \end{align}
\end{remark}
Once $m_0, n_0, \gamma$ are given, $\nu_1$ is a fixed constant.
To simplify notation, we suppress the dependence on $\nu_1$ in the sequel.

We illustrate the persistence of one-gap solutions to the KP-II equation as our primary example. Theorem \ref{thm:intro-one-gap} and its proof for the KP-I equation follow analogously from the arguments presented for the KP-II equation. 
For notational simplicity, we denote $u_0=u_0^{(2)}, q=q^{(2)}, \omega_0 = \omega_0^{(2)}, \mathcal{U}_{\rho_0} = \mathcal{U}_{\rho_0}^{(2)}$ in the following argument.
Let $\hat{u} = \hat{u}_0 + \hat{v}, \hat{v}\in PH^{\sigma,s}$. Denote
\begin{align*}
    \tilde{\Omega}_{\rho,C} = \{(\omega,r)\in \mathbb{C}^2: |\omega-\omega_0(r)| < C|r|^{p-1}, r\in \mathcal{U}_{\rho} \},\quad
    \Omega_{\rho,C} = \tilde{\Omega}_{\rho,C} \cap \mathbb{R}^2
\end{align*}
for fixed constants $0< \rho \leq \rho_0$ and $C=C(\mathcal{P}_1,\gamma,m_0)>0$. The constant $C$ will be chosen in subsection \ref{sec:one-gap-2}.

Since $\hat{u}_0$ satisfies \eqref{kp-fourier1}, the equations \eqref{pfangcheng1} for KP-II reduce to
\begin{align}
    \bigg(&\omega-\frac{4\pi^2}{\gamma^2}m_0^3k^2+\frac{n_0^2}{m_0}\bigg)\hat{v}(k)+ \frac{m_0 }{2}\widehat{[2u_0 v+v^2]}(k)\label{Pfangcheng1.1}\\
    &+(\omega-\omega_0)\hat{u}_0(k)+ m_0 \widehat{P_1(u_0+v)}(k)=0,|k|\geq 2\label{Pfangcheng1.2}.
\end{align}
Define $L(\omega,r)=L_0(\omega)+L_1(r), G(\omega,r): PH^{\sigma,s} \to P\ell_e^2(\mathbb{T}_{\gamma})$ for $(\omega,r) \in \tilde{\Omega}_{\rho_0,C}$ as follows:
\begin{align*}
    &L_0(\omega)(\hat{v}) = P_{\mathbb{Z}\setminus (S_1\cup \{0\})} \left( \left\{ \left(\omega-\frac{4\pi^2}{\gamma^2}m_0^3k^2+\frac{ n_0^2}{m_0}\right)\hat{v}(k) \right\}_{k \in \mathbb{Z}} \right), \\
    &L_1(r)(\hat{v}) = P_{\mathbb{Z}\setminus (S_1\cup \{0\})} \left(\left\{ m_0\ \widehat{u_0v}(k) \right\}_{k \in \mathbb{Z}} \right), \\
    &G(\omega,r)(\hat{v}) = P_{\mathbb{Z}\setminus (S_1\cup \{0\})}\left( \bigg\{ (\omega-\omega_0(r))\hat{u}_0(k) + \frac{m_0 }{2}\widehat{v^2}(k) + m_0\widehat{P_1(u_0+v)}(k) \bigg\}_{k \in \mathbb{Z}} \right).
\end{align*}
Then \eqref{Pfangcheng1.1}--\eqref{Pfangcheng1.2} are equivalent to
\begin{equation}
    L(\omega,r)(\hat{v})+G(\omega,r)(\hat{v}) = 0.\label{banachform}
\end{equation}
We also denote the entries of the Fourier matrix representation of $L_0(\omega), L_1(r), L(\omega,r)$ at $(k,k')$ by $L_0(\omega)(k,k'),\allowbreak L_1(r)(k,k'), L(\omega,r)(k,k')$ in the following subsections.

To describe the Taylor expansion of the reduced perturbation, let
\[
    \mathcal{I} = \{(0,0), (1,0), (0,1), (2,0), (1,1), (0,2)\}
\]
For \(\alpha=(\alpha_1,\alpha_2)\in \mathcal{I}\), define $D^\alpha := m_0^{\alpha_1} n_0^{\alpha_2} \partial_\zeta^{|\alpha|}$. Thus,
\[
    \left\{ D^\alpha:\alpha\in \mathcal{I} \right\} = \left\{ id, m_0\partial_\zeta, n_0\partial_\zeta, m_0^2\partial_{\zeta\zeta}, m_0n_0\partial_{\zeta\zeta}, n_0^2\partial_{\zeta\zeta} \right\}.
\]
For $\beta = (\beta_{\alpha})_{\alpha \in \mathcal{I}} \in \mathbb{N}^{6}$, set $\boldsymbol{u}^{\beta}=\prod_{\alpha\in \mathcal{I}}(D^{\alpha}u)^{\beta_\alpha}$.
Since $P_1(u)$ is real analytic, we expand it into a Taylor series
\begin{equation*}
    P_1(u)=\sum_{l=p}^{\infty} P_{1,l}(u),\quad
    P_{1,l}(u)=\sum_{\beta=(\beta_\alpha)_{\alpha\in I},|\beta|_1=l} p_{\beta}\boldsymbol{u}^{\beta}, \quad p_{\beta} \in \mathbb{R}.
\end{equation*}
In particular, $P_{1,p}(u) = u^{p}, P_{1,p+1}(u) = u^{p+1}$.

We define the scalar majorant
\[
    f_1(z) := z^p + z^{p+1} + f_{1,>p+1}(z),\quad 
    f_{1,>p+1}(z) := \sum_{l=p+2}^{\infty} \sum_{|\beta|_1=l} |p_{\beta}|\cdot z^{l}, \quad z\geq 0.
\]
Since $Q_{1,>p+1}(z_0,\cdots, z_5)$ is real analytic on $|(z_0,\cdots,z_5)|_\infty< R_1$, the series defining $f_{1,>p+1}(z)$ converges for $0<z< R_1$.

\begin{lemma}\label{lm4.2}
    If $s>\frac{1}{2}$ and $\hat{u}_1,\hat{u}_2 \in H^{\sigma,s}$, there exists a constant $C_{\sigma, s}>0$ such that
    \begin{equation*}
        \|\hat{u}_1* \hat{u}_2\|_{\sigma,s}\leq C_{\sigma, s}\|\hat{u}_1\|_{\sigma,s}\|\hat{u}_2\|_{\sigma,s}.
    \end{equation*}
\end{lemma}
Now we estimate the nonlinear operator $G$.
\begin{lemma}\label{lm3.6}    
    Denote $\hat{u}=\hat{u}_0+\hat{v}, \hat{u}_1=\hat{u}_0+\hat{v}_1, \hat{u}_2=\hat{u}_0+\hat{v}_2$. Suppose $(\omega,r) \in \tilde{\Omega}_{\rho_0,C}$ for some constant $C$, $s>\frac{5}{2}$, and $\hat{v}\in PH^{\sigma,s}$ such that $C_{\sigma, s}\|\hat{u}\|_{\sigma,s}< R_1$. Then
    \begin{equation}\label{nonlinearestimate1}
        \|G(\omega,r)(\hat{v})\|_{\sigma,s-2,P}\leq C|r|^{p-1} \|\hat{u}_0\|_{\sigma,s,P}+ \frac{|m_0|}{2}C_{\sigma, s}\|\hat{v}\|_{\sigma,s,P}^2+ |m_0| f_1(C_{\sigma, s}\|\hat{u}_0+ \hat{v} \|_{\sigma,s}).
    \end{equation}
    Moreover, if $\max(\|\hat{v}_1\|_{\sigma,s,P},\|\hat{v}_2\|_{\sigma,s,P})\leq M_v$, we have
    \begin{equation}\label{Lipchitz1}
        \|G(\omega,r)(\hat{v}_2)-G(\omega,r)(\hat{v}_1)\|_{\sigma,s-2,P}\leq |m_0| \left( C_{\sigma, s} M_v + f_1'(C_{\sigma, s}(\| \hat{ u }_0\|_{\sigma,s}+M_v))\right) \|\hat{v}_2-\hat{v}_1\|_{\sigma,s,P}.
    \end{equation}
\end{lemma}
\begin{proof}    
    With Lemma \ref{lm4.2}, for any $\beta$ with $|\beta|_1=l$, we have
    \begin{equation*}
        \|\widehat{ p_{\beta} \boldsymbol{u}^{\beta} }\|_{\sigma, s-2} \leq  | p_{\beta} | C_{\sigma, s}^l \|\hat{u}\|_{\sigma,s}^{l},
    \end{equation*}
    which implies that
    \begin{align*}
        \|\widehat{ P_1(u) }\|_{\sigma,s-2}\leq \sum_{l=p}^{\infty} \sum_{|\beta|_1=l} | p_{\beta} |(C_{\sigma, s}\|\hat{u}\|_{\sigma,s})^{l}.
    \end{align*}
    If $C_{\sigma, s}\| \hat{u} \|_{\sigma,s}< R_1$, we have
    \begin{align*}
        \|G(\omega,r)(\hat{v})\|_{\sigma,s-2,P}
        \leq& \|(\omega-\omega_0)\hat{ u }_0\|_{\sigma,s-2}+ \frac{m_0}{2}\| \widehat{v^2}\|_{\sigma,s-2,P}+ m_0\|\widehat{ P_1(u_0+v) }\|_{\sigma,s-2,P}\\
        \leq& C|r|^{p-1} \|\hat{ u }_0\|_{\sigma,s,P} + \frac{|m_0|}{2}C_{\sigma, s}\|\hat{v}\|_{\sigma,s,P}^2 + |m_0| f_1(C_{\sigma, s}\|\hat{ u }_0+\hat{v}\|_{\sigma,s})<\infty.
    \end{align*}
    For the Lipschitz estimate, we use the integral form of the remainder:
    \begin{align*}
        G(\omega,r)(\hat{v}_2)-G(\omega,r)(\hat{v}_1)
        &= \mathcal{F} \bigg( m_0(v_2-v_1)\cdot\int_0^1 (v_1+t(v_2-v_1))dt\\
        &+m_0\sum_{l=p}^{\infty}\sum_{\alpha\in \mathcal{I}}D^{\alpha}(v_2-v_1)\int_0^1\frac{\partial P_{1,l}}{\partial (D^{\alpha}v)}(u_1+ t(u_2-u_1)) dt \bigg).
    \end{align*}
    The first term can be estimated as
    \begin{align*}
        \bigg\|\mathcal{F}\bigg( m_0(v_2-v_1)\int_0^1 (v_1+t(v_2-v_1))dt \bigg) \bigg\|_{\sigma,s-2,P}\leq &C_{\sigma, s}|m_0|\max(\|\hat{v}_1\|_{\sigma,s,P},\|\hat{v}_2\|_{\sigma,s,P})\\
        &\cdot\|\hat{v}_2-\hat{v}_1\|_{\sigma,s,P}.
    \end{align*}
    For the second term, with Lemma \ref{lm4.2} and the coefficient-wise estimates of $P_1$, we have
    \begin{align*}
        &\sum_{l=p}^{\infty}\left\| \mathcal{F}\left( \sum_{\alpha\in \mathcal{I}}D^{\alpha}(v_2-v_1)\int_0^1\frac{\partial P_{1,l}}{\partial (D^{\alpha}u)}(u_1+ t(u_2-u_1)) dt \right)\right\|_{\sigma,s-2}\\
        \leq &  \sum_{l=p}^{\infty} \sum_{\alpha\in I}\|\widehat{ D^{\alpha}(v_2-v_1) } \|_{\sigma,s-2} \cdot \sum_{|\beta|_1=l}\left\| \mathcal{F}\left( \int_0^1 \beta_\alpha p_{\beta}\boldsymbol{\tilde{u}}(t)^{\beta-e_\alpha} dt \right) \right\|_{\sigma,s-2}\\
        \leq & \sum_{l=p}^{\infty} \sum_{|\beta|_1=l} | p_{\beta} |\,l\,(C_{\sigma, s} \max_{t\in [0,1]}\|\widehat{ \tilde{u}(t) }\|_{\sigma,s})^{l-1}\cdot\|\hat{v}_2-\hat{v}_1\|_{\sigma,s}\\
        = & f_1'(C_{\sigma, s}\max(\| \hat{u}_1\|_{\sigma,s},\|\hat{u}_2\|_{\sigma,s}))\cdot\|\hat{v}_2-\hat{v}_1\|_{\sigma,s,P},
    \end{align*}
    where $\tilde{u}(t) = u_1+t(u_2-u_1)$, $\boldsymbol{\tilde{u}}(t)^{\beta} = \prod_{\alpha\in \mathcal{I}}\left(D^{\alpha}\tilde{u}(t)\right)^{\beta_\alpha}$, and $e_{\alpha} = (\delta_{\alpha,\alpha'})_{\alpha'\in I}$.

    In conclusion, we have
    \begin{align*}
        \|G(\omega,r)(\hat{v}_2)-G(\omega,r)(\hat{v}_1)\|_{\sigma,s-2,P} \leq |m_0| \left( C_{\sigma, s}M_v + f_1'(C_{\sigma, s}(\|\hat{ u }_0\|_{\sigma,s}+M_v))\right) \|\hat{v}_2-\hat{v}_1\|_{\sigma,s,P}.
    \end{align*}
\end{proof}
Choose $\rho_1 \in (0,\rho_0]$ such that
\begin{align*}
    C\rho_1^{p-1} \leq \dfrac{\pi^2}{\gamma^2}|m_0|^3,\quad C_0\rho_1^2\leq \dfrac{\pi^2}{\gamma^2}|m_0|^3, \quad C_{\sigma, s}\|\hat{ u }_0(r)\|_{\sigma,s} < R_1, \textup{ for } r\in \mathcal{U}_{\rho_1}.
\end{align*}
\begin{lemma}\label{lm3.7}
    Suppose $s>\frac{5}{2}$. Then, there exist constants $C_2>0$, and $\rho_2 \in (0,\rho_1]$ such that $L^{-1}(\omega,r)$ is a bounded operator from $PH^{\sigma,s-2}$ to $PH^{\sigma,s}$ for $(\omega,r) \in \tilde{\Omega}_{\rho_2,C}$ with bound less than $2C_2$.
\end{lemma}
\begin{proof}
    Let $\rho_2 \leq \rho_1$. Since $(\omega,r) \in \tilde{\Omega}_{\rho_2,C}$, the estimate \eqref{omega0} for $\omega_0(r)$ leads to
    \begin{equation*}
        \left|\omega - \left(\frac{4\pi^2}{\gamma^2}m_0^3 - \frac{n_0^2}{m_0}\right)\right| \leq C_0\rho_1^2 + C\rho_1^{p-1}.
    \end{equation*}
    Suppose $k\notin S_1\cup\{0\}$. Since $|k|^2\geq 4$, the denominator
    \begin{align*}
        \bigg| \omega -\frac{4\pi^2}{\gamma^2}m_0^3k^2+\frac{n_0^2}{m_0} \bigg| \geq \bigg| \frac{4\pi^2}{\gamma^2}m_0^3(k^2-1) \bigg| -C_0\rho_1^2 - C \rho_1^{p-1}
        \geq \bigg| \frac{\pi^2}{\gamma^2}m_0^3 k^2 \bigg|.
    \end{align*}
    Thus, there exists a constant $C_2 = C_2(m_0,n_0,\gamma,\rho_1)$ such that
    \begin{align*}
        \|L_0^{-1}(\omega)\hat{v}\|_{\sigma,s,P}^2
        &=\sum_{k\in \mathbb{Z}\setminus (S_1\cup\{0\})}\frac{|\hat{v}(k)|^2(1+\nu_1^2k^2)^{2}}{\left|\omega -\frac{4\pi^2}{\gamma^2}m_0^3k^2+\frac{n_0^2}{m_0}\right|^2} e^{2\sigma\nu_1 |k|}\langle \nu_1 k \rangle^{2(s-2)}\notag\\
        &\leq\sum_{k\in \mathbb{Z}\setminus (S_1\cup\{0\})}\frac{|\hat{v}(k)|^2(1+ \nu_1^2k^2)^{2}}{\left|\frac{\pi^2}{\gamma^2}m_0^3 k^2\right|^2} e^{2\sigma\nu_1 |k|}\langle \nu_1 k \rangle^{2(s-2)}\notag\\
        &\leq C_2^2\|\hat{v}\|_{\sigma,s-2,P}^2<\infty.
    \end{align*}
    Since $L_0^{-1}(\omega)$ is a bounded operator from $PH^{\sigma,s-2}$ to $PH^{\sigma,s}$, we rewrite $L^{-1}(\omega,r)=(I+L_0^{-1}(\omega)L_1(r))^{-1}L_0^{-1}(\omega)$. By the Neumann-series argument, if $\|L_0^{-1}(\omega)L_1(r)\|_{\sigma,s,s,P}\leq\frac{1}{2}$, we have $\|L^{-1}(\omega,r)\|_{\sigma,s-2,s,P} \leq 2C_2$. In fact, we have
    \begin{align*}
        \|L_0^{-1}(\omega)L_1(r) \hat{v}\|_{\sigma,s,P}&\leq \|L_0^{-1}(\omega)\|_{\sigma,s-2,s,P}\|L_1(r) \hat{v}\|_{\sigma,s-2,P}\\
        &\leq C_2|m_0|\cdot\|\hat{ u }_0*\hat{v}\|_{\sigma,s-2,P}\\
        &\leq C_2C_{\sigma, s}|m_0|\cdot\|\hat{ u }_0\|_{\sigma,s}\|\hat{v}\|_{\sigma,s,P}.
    \end{align*}
    Thus there exist $\rho_2\in (0,\rho_1]$ such that $\|L_0^{-1}(\omega)L_1(r)\|_{\sigma,s,s,P}\leq C_2 C_{\sigma, s}\allowbreak |m_0| \cdot\|\hat{ u }_0\|_{\sigma,s} \allowbreak \leq \frac{1}{2}$ for $|r|\leq \rho_2$.
\end{proof}

\begin{lemma}\label{123}
    Suppose $s>\frac{5}{2}$ and $p \geq 3$. Then, there exist positive constants $M_0>0$ and $\rho_{3} \in (0,\rho_2]$ such that for any $(\omega,r) \in \tilde{\Omega}_{\rho_3,C}$, the functional $-L^{-1}G(\omega, r)$ maps 
    \begin{equation*}
        \mathcal{B}_{r} :=\{\hat{v} \in PH^{\sigma,s} : \|\hat{v}\|_{\sigma,s,P} \leq M_0 |r|^p\}
    \end{equation*}
    into itself. Moreover, $-L^{-1}G(\omega,r)$ is a uniform contraction on $\mathcal{B}_{r}$ for $(\omega,r) \in \tilde{\Omega}_{\rho_3,C}$, satisfying
    \begin{equation*}
        \|L^{-1}G(\omega,r)(\hat{v}_1) - L^{-1}G(\omega,r)(\hat{v}_2)\|_{\sigma,s,P} \leq |r|^{p-2}\cdot \|\hat{v}_1 - \hat{v}_2\|_{\sigma,s,P}
    \end{equation*}
    for any $\hat{v}_1, \hat{v}_2 \in \mathcal{B}_{r}$.
\end{lemma}
\begin{proof}
    According to Lemmas \ref{lm3.6} and \ref{lm3.7}, the mapping $-L^{-1}G(\omega, r)$ defines a self-mapping and a uniform contraction on $\mathcal{B}_r$ provided that the following three conditions are met uniformly for $(\omega, r) \in \tilde{\Omega}_{\rho_3,C}$ with $\rho_3 \leq \rho_2$:
    \begin{align}
         C_{\sigma, s} &\big( \|\hat{ u }_0(r)\|_{\sigma,s} + M_0 |r|^p \big) < R_1, \label{ball1}\\
         C |r|^{p-1} \|\hat{ u }_0(r)\|_{\sigma,s,P} &+ \frac{|m_0|}{2}C_{\sigma, s} (M_0 |r|^p)^2 \label{ball2.1}\\
        &+ |m_0| f_1\left( C_{\sigma, s}(\|\hat{ u }_0(r)\|_{\sigma,s}+M_0 |r|^p)\right) \leq \frac{M_0}{2C_2} |r|^p, \label{ball2.2} \\
         2C_2 |m_0|\big( 2C_{\sigma, s} M_0 |r|^p &+ f_1'\big( C_{\sigma, s} (\|\hat{ u }_0(r)\|_{\sigma,s} + 2M_0 |r|^p) \big) \big) \leq |r|^{p-2}< 1 \label{ball3}.
    \end{align}
    
    With the high-order vanishing property of $f_1$ and $f_1'$ at the origin, there exist positive constants $C_{f_1}$ and $C_{f_1}'$ such that $f_1(z) \leq C_{f_1} |z|^p$ and $f_1'(z) \leq C_{f_1}' |z|^{p-1}$ for $p \geq 3$.
    
    Since \eqref{u0estimate}--\eqref{u0estimateP} hold, there exist constants $M_0 = M_0(P_1, C, C_2, C_{\sigma, s},R_1, m_0, \| \hat{ u }_0 \|_{\sigma,s},\rho_2) \allowbreak >0$ and $\rho_{3} = \rho_{3}(P_1, C_2, C_{\sigma, s},R_1, m_0, \|\hat{ u }_0\|_{\sigma,s},M_0) >0$ such that the inequalities \eqref{ball1}--\eqref{ball3} hold.
\end{proof}

With Lemma \ref{123}, we obtain the following theorem.

\begin{theorem}\label{Banachfixed1}
    Suppose $s>\frac{5}{2}$ and $p\geq 3$. For $(\omega, r) \in \tilde{\Omega}_{\rho_3,C}$, the range equation \eqref{banachform} admits a unique solution $\hat{v}(\omega,r)$ in $\mathcal{B}_r$. In particular, for every $(\omega, r) \in \tilde{\Omega}_{\rho_3,C}$,
    \[
        \| \hat{v}(\omega,r) \|_{\sigma,s,P} \leq M_0 |r|^{p}.
    \]
    The solution $\hat{v}(\omega,r)$ is holomorphic in $\tilde{\Omega}_{\rho_3,C}$. Moreover, for every $(\omega,r) \in \Omega_{\rho_3,C}$, we have $\hat{v}(\omega,r)\in PH_{\mathrm{re}}^{\sigma,s}$. Consequently, the restriction of $\hat{v}(\omega,r)$ to $\Omega_{\rho_3,C}$ is real analytic.

    Finally, for every $(\omega, r) \in \tilde{\Omega}_{\rho_3,C}$, we have the following estimates:
    \begin{equation}\label{qiudaoguji}
        \left\| \frac{\partial \hat{v}}{\partial \omega} \right\|_{\sigma,s,P}  \lesssim |r|^2,\qquad \left\| \frac{\partial^2 \hat{v}}{\partial \omega^2} \right\|_{\sigma,s,P}  \lesssim |r|^2.
    \end{equation}
\end{theorem}
\begin{proof}
    \textit{Step 1}.
    We initiate the standard Picard iteration scheme for the range equation:
    \begin{equation*}
        \hat{v}_0(\omega, r) = 0, \quad \hat{v}_{n+1}(\omega, r) = -L^{-1}(\omega, r)G(\omega, r)(\hat{v}_n(\omega, r)), \quad n \geq 0.
    \end{equation*}
    With Lemma \ref{123}, it is easy to see that $\| \hat{v}_1 \|_{\sigma,s,P} \leq M_0\rho_3^p$ uniformly on $\tilde{\Omega}_{\rho_3,C}$.
    Since $\hat{u}_0(k;r)$ and $\omega_0(r)$ are holomorphic functions, the operators $L^{-1}(\omega, r)$ and $G(\omega, r)$ are holomorphic in $\tilde{\Omega}_{\rho_3,C}$. Since $\hat{v}_0 = 0$ is holomorphic, by finite composition rules, each iteration layer $\hat{v}_n(\omega, r)$ defines a well-defined holomorphic function of $(\omega, r)$ on $\tilde{\Omega}_{\rho_3,C}$.

    With the contraction property in Lemma \ref{123}, we have
    \begin{align*}
        \|\hat{v}_{n+1} - \hat{v}_n\|_{\sigma,s,P} = \|L^{-1}(\omega, r)G(\omega, r)(\hat{v}_n) - L^{-1}(\omega, r)G(\omega, r)(\hat{v}_{n-1})\|_{\sigma,s,P} \leq \mathfrak q \|\hat{v}_n - \hat{v}_{n-1}\|_{\sigma,s,P},
    \end{align*}
    uniformly on $\tilde{\Omega}_{\rho_3,C}$, where $\mathfrak q := \rho_3^{p-2} \in (0,1)$. A geometric telescoping expansion reveals that for any shift $k \geq 1$:
    \begin{align*}
        \|\hat{v}_{n+k} - \hat{v}_n\|_{\sigma,s,P} \leq \sum_{j=n}^{n+k-1} \mathfrak q^j \|\hat{v}_1 - \hat{v}_0\|_{\sigma,s,P} \leq \frac{\mathfrak q^n}{1-\mathfrak q} M_0 \rho_3^p.
    \end{align*}
    Since $\mathfrak q < 1$ is uniform, $\{\hat{v}_n(\omega, r)\}_{n=0}^{\infty}$ forms a Cauchy sequence in $PH^{\sigma,s}$ that converges locally uniformly on $\tilde{\Omega}_{\rho_3,C}$. By Weierstrass's theorem for holomorphic functions in Banach spaces, the uniform limit $\hat{v}(\omega, r) \in PH^{\sigma,s}$ is holomorphic in $\tilde{\Omega}_{\rho_3,C}$.
    Passing to the limit in the Picard iteration shows that $\hat{v}(\omega, r)$ solves the range equation \eqref{banachform}.
    Moreover, Lemma~\ref{123} implies that, for every fixed $(\omega,r)\in\tilde\Omega_{\rho_3,C}$ and $n\geq0,\ \hat{v}_n(\omega,r)\in\mathcal B_r$. Since $\mathcal B_r$ is closed in $PH^{\sigma,s}$, the limit also belongs to $\mathcal B_r$. Therefore, $\|\hat{v}(\omega,r)\|_{\sigma,s,P} \leq M_0 |r|^p$.

    \textit{Step 2}.
    We next show that $\hat{v}(\omega, r) \in PH_{\mathrm{re}}^{\sigma,s}$ for $(\omega , r) \in \Omega_{\rho_3,C}$.
    For a $PH^{\sigma,s}$-valued function $\hat v: \Omega_{\rho_3,C}\to PH^{\sigma,s}$, define its reflection $\mathcal{R}$ by
    \begin{equation*}
        (\mathcal{R}\hat{v})(\omega, r) := \overline{\hat{v}(\overline{\omega}, \overline{r})}.
    \end{equation*}
    Since $\widehat{u_0(r)},\omega_0(r)$ are holomorphic in $\mathcal{U}_{\rho_3}$ with real-analytic restriction to $(0,\rho_0)$, by the Schwarz reflection principle,
    \[
        \omega_0(\overline{r}) = \overline{\omega_0(r)},\quad \widehat{u_0(\overline{r})} = \overline{\widehat{ u_0(r) }}.
    \]
    Together with the real analyticity and the parity properties of $Q_1$, we have
    \[
        \mathcal{R}(L^{-1}(\omega, r)G(\omega, r)(\hat{v})) = L^{-1}(\omega, r)G(\omega, r)(\mathcal{R}\hat{v}).
    \]
    
    Since the base approximation $\hat{v}_0 = 0$ satisfies $\mathcal{R}\hat{v}_0 = \hat{v}_0$, mathematical induction guarantees that $\mathcal{R}\hat{v}_n = \hat{v}_n$ holds identically for all $n \geq 0$ across $\tilde{\Omega}_{\rho_3,C}$. By the uniqueness of the fixed point guaranteed by Lemma \ref{123}, the uniform limit must inherit the exact symmetry, yielding $\mathcal{R}\hat{v} = \hat{v}$, which means $\overline{\hat{v}(\overline{\omega}, \overline{r})} = \hat{v}(\omega, r)$. 
    Restricting $(\omega,r)$ to $\Omega_{\rho_3,C}$, the relation implies that $\overline{\widehat{v(\omega, r)}(k)} = \widehat{v(\omega,r)}(k) \in \mathbb{R}$. Hence the solution is real-valued and real analytic on $\Omega_{\rho_3,C}$.

    \textit{Step 3}.
    For $(\omega,r)\in\tilde\Omega_{\rho_3,C}$ and any $\hat{h} \in PH^{\sigma,s}$, there exists sufficiently small $\varepsilon>0$ such that $\|\hat{v}+\varepsilon\hat{h}\|_{\sigma,s,P} \leq  2M_0 |r|^{p}$. Hence, the definition of the Fr\'echet derivative combined with \eqref{ball3} implies
    \begin{align*}
        \| L^{-1}(\omega,r) DG(\omega,r)(\hat{v})[\hat{h}]\|_{\sigma,s,P}
        \leq &2C_2 \|DG(\omega,r)(\hat{v})[\hat{h}]\|_{\sigma,s-2,P} \\
        \leq &2C_2 \lim_{\epsilon \to 0} \frac{\|G(\omega,r)(\hat{v}+\epsilon \hat{h}) - G(\omega,r)(\hat{v})\|_{\sigma,s-2,P}}{\epsilon} \\
        \leq &\mathfrak q\ \|\hat{h}\|_{\sigma,s,P}.
    \end{align*}
    This leads to the operator norm bound $\|L^{-1}(\omega,r)DG(\omega,r)(\hat{v})\|_{\sigma,s,s,P} \leq \mathfrak q < 1$. Differentiating the fixed-point identity $\hat{v} = -L^{-1}(\omega,r)G(\omega,r)(\hat{v})$ with respect to $\omega$ via the chain rule provides:
    \begin{align*}
        \frac{\partial \hat{v}}{\partial \omega} = -\bigg[\frac{\partial L^{-1}(\omega,r)}{\partial \omega} \bigg] G(\omega,r)(\hat{v}) - L^{-1}(\omega,r) DG(\omega,r)(\hat{v})\left[ \frac{\partial \hat{v}}{\partial \omega} \right] - L^{-1}(\omega,r) \frac{\partial G(\omega,r)}{\partial \omega}(\hat{v}),
    \end{align*}
    which can be rearranged into:
    \begin{equation}
        \frac{\partial \hat{v}}{\partial \omega} = -\big[ I + L^{-1}(\omega,r) DG(\omega,r)(\hat{v}) \big]^{-1}\bigg( \frac{\partial L^{-1}(\omega,r)}{\partial \omega} G(\omega,r)(\hat{v}) + L^{-1}(\omega,r) \frac{\partial G(\omega,r)}{\partial \omega}(\hat{v}) \bigg). \label{partialformula}
    \end{equation}
    Since $\|L^{-1}(\omega,r) DG(\omega,r)(\hat{v})\|_{\sigma,s,s,P} \leq \mathfrak q < 1$, the Neumann series expansion is valid, guaranteeing
    \begin{equation*}
        \| [ I + L^{-1}(\omega,r) DG(\omega,r)(\hat{v}) ]^{-1} \|_{\sigma,s,s,P} \leq \frac{1}{1-\mathfrak q}.
    \end{equation*}
    Since $L^{-1}(\omega,r) \in \mathcal{L} (PH^{\sigma,s-2},PH^{\sigma,s}) , \dfrac{\partial L(\omega,r)}{\partial \omega} = id \in \mathcal{L} (PH^{\sigma,s},PH^{\sigma,s-2})$, the operator
    \begin{equation*}
        \dfrac{\partial L^{-1}(\omega,r)}{\partial\omega} = -L^{-1}(\omega,r) \dfrac{\partial L(\omega,r)}{\partial\omega} L^{-1}(\omega,r)
    \end{equation*}
    is bounded from $PH^{\sigma,s-2}$ to $PH^{\sigma,s}$.
    Similarly, we have
    \begin{align*}
        \left\| L^{-1}(\omega,r) \frac{\partial G(\omega,r)}{\partial \omega}(\hat{v}) \right\|_{\sigma,s,P} = \|L^{-1}(\omega,r) P_{\mathbb{Z}\setminus (S_1\cup \{0\})}\hat{ u }_0\|_{\sigma,s,P}\leq 2C_2\|\hat{ u }_0\|_{\sigma,s,P}\lesssim |r|^2.
    \end{align*}
    For $(\omega,r)\in\tilde\Omega_{\rho_3,C}$, taking the norm on both sides of \eqref{partialformula} yields
    \begin{align*}
        \left\| \frac{\partial \hat{v}(\omega,r)}{\partial \omega} \right\|_{\sigma,s,P} 
        \leq& \frac{1}{1-\mathfrak q} \left\| \frac{\partial L^{-1} (\omega,r)}{\partial \omega} \right\|_{\sigma,s-2,s,P} \|G(\omega,r)(\hat{v})\|_{\sigma,s-2,P} \\
        & +  \frac{1}{1-\mathfrak q} \left\|L^{-1}(\omega,r) \frac{\partial G(\omega,r)}{\partial \omega}  (\hat{v}) \right\|_{\sigma,s,P}\\
        \lesssim& |r|^2.
    \end{align*}
    In addition, differentiating the range equation $L(\omega,r)\hat{v} + G(\omega,r)(\hat{v}) = 0$ twice yields
    \begin{equation*}
        \frac{\partial^2 \hat{v}}{\partial \omega^2} = -\big[ L(\omega,r) + DG(\omega,r)(\hat{v}) \big]^{-1}\bigg( 2\frac{\partial L(\omega,r)}{\partial \omega} \frac{\partial \hat{v}}{\partial \omega} + D^2G(\omega,r)(\hat{v})\left[\frac{\partial \hat{v}}{\partial \omega},\frac{\partial \hat{v}}{\partial \omega}\right] \bigg).
    \end{equation*}
    Similarly, for $(\omega,r)\in\tilde\Omega_{\rho_3,C}$ and $\hat{h}_1,\hat{h}_2 \in PH^{\sigma,s}$, the definition of the second Fr\'echet derivative and the definition of $G$ implies that
    \begin{align*}
        &\| L^{-1}(\omega,r) D^2G(\omega,r)(\hat{v})[\hat{h}_1,\hat{h}_2] \|_{\sigma,s,P} \\
        \leq &2C_2 \| D^2G(\omega,r)(\hat{v})[\hat{h}_1,\hat{h}_2] \|_{\sigma,s-2,P} \\
        \leq &2C_2|m_0|\cdot \|\widehat{h_1h_2} + \widehat{ D^2P_1(u_0+v)[h_1,h_2]} \|_{\sigma,s-2,P}\\
        \leq &2C_2|m_0|\cdot [1 +f_1''(C_{\sigma, s}(\|\hat{u}_0\|_{\sigma,s}+\|\hat{v}\|_{\sigma,s,P})]\cdot C_{\sigma, s}\|\hat{h}_1\|_{\sigma,s,P}\|\hat{h}_2\|_{\sigma,s,P}
    \end{align*}
    Notice that $f_1''(z)\lesssim z^{p-2}$ for $p\geq 3$. Taking $\hat{h}_1 = \hat{h}_2 = \frac{\partial \hat{v}}{\partial \omega}$, we obtain
    \[
        \| L^{-1}(\omega,r) D^2G(\omega,r)(\hat{v})[\hat{h}_1,\hat{h}_2] \|_{\sigma,s,P} \lesssim |r|^4.
    \]
    Thus, we have
    \begin{align*}
        \left\|\frac{\partial^2 \hat{v}}{\partial \omega^2}\right\|_{\sigma,s,P} \leq \frac{1}{1-\mathfrak q} \bigg( 4C_2 \left\| \frac{\partial \hat{v}}{\partial \omega}\right\|_{\sigma,s,P} + \left\| L^{-1}(\omega,r) D^2G(\omega,r)(\hat{v})\left[\frac{\partial \hat{v}}{\partial \omega},\frac{\partial \hat{v}}{\partial \omega}\right]\right\|_{\sigma,s,P} \bigg)
        \lesssim |r|^2.
    \end{align*}
    This proves \eqref{qiudaoguji} and completes the proof.
\end{proof}

\subsection{Solving the bifurcation equation}\label{sec:one-gap-2}

We solve the bifurcation equation and calculate the perturbed frequency by a perturbative argument.
Recall the bifurcation equation \eqref{qfangcheng1} is equivalent to
\begin{equation}\label{solvebifurcation}
    \left(\omega-\frac{4\pi^2}{\gamma^2}m_0^3k^2-\frac{\lambda_{\iota} n_0^2}{m_0}\right)\hat{u}(k)+ \frac{m_0}{2}\widehat{u^2}(k)+ m_0 \widehat{P_1(u)}(k)=0,|k|=1.
\end{equation}
For notational simplicity, we denote the Fourier coefficients $\hat{u}_0(k;r)$ by $\hat{u}_0(k)$.
\begin{theorem}\label{thm:solveomega}
    After reducing $\rho_3$, if necessary, there exists a real-analytic function $\omega: (0,\rho_3) \to \mathbb{R}$ such that $\hat{u} = \hat{u}_0 + \widehat{v(\omega(r),r)}$ satisfies the bifurcation equation \eqref{solvebifurcation}. Moreover, when $p$ is odd, we have
    \begin{equation*}
        \omega(r) =\omega_0(r) - m_0 \binom{p}{(p-1)/2}\hat{u}_0(1)^{(p-1)/2}\hat{u}_0(-1)^{(p-1)/2} +O(r^{p}),
    \end{equation*}
    and when $p$ is even, we have
    \begin{align*}
        &\omega(r) =\omega_0(r) - \frac{\gamma^2 \binom{p}{p/2-1}}{12\pi^2m_0}\hat{u}_0(-1)^{p/2}\hat{u}_0(1)^{p/2} - m_0 \hat{u}_0(1)^{p/2-2}\hat{u}_0(-1)^{p/2-2}\\
        & \cdot \left(p\ \binom{p-1}{p/2-1}\hat{u}_0(2)\hat{u}_0(-1)^2 + p\binom{p-1}{p/2+1}\hat{u}_0(-2)\hat{u}_0(1)^2 + \binom{p+1}{p/2} \hat{u}_0(1)^2  \hat{u}_0(-1)^2 \right)+O(r^{p+1}).
    \end{align*}
    Thus, $(\omega(r), r) \in \Omega_{\rho_3,C}$ for sufficiently large $C$.
\end{theorem}
\begin{proof}    
    Fix $C>0$ sufficiently large, which will be determined later. Denote
    \begin{equation*}
        I_{C}(r) := \{\omega \in \mathbb{C}, |\omega - \omega_0(r)| < C|r|^{p-1} \}, \textup{ for }r\in \mathcal{U}_{\rho_3}.
    \end{equation*}
    Throughout the proof, we seek a solution satisfying the a priori restriction
    \begin{equation}
       \omega\in I_{C}(r),\textup{ for }r\in \mathcal{U}_{\rho_3}.\label{xianyanguji}
    \end{equation}
    If $(\omega,r) \in \Omega_{\rho_3,C}$, Theorem \ref{Banachfixed1} implies that $\widehat{u_0(r)}+ \widehat{v(\omega,r)} \in \ell_{re}^2(\mathbb{T}_{\gamma})$. Hence the bifurcation equation \eqref{solvebifurcation} for $k=1$ and $k=-1$ is the same equation.
    
    Since $\widehat{u_0(r)}$ satisfies \eqref{kp-fourier1} for $r \in \mathcal{U}_{\rho_3}$, with $\hat{u} = \widehat{u_0(r)} + \widehat{v(\omega,r)}$, the bifurcation equation \eqref{solvebifurcation} for $k=1$ becomes
    \begin{equation}\label{solvebifurcationcomplex}
        (\omega-\omega_0)\hat{u}_0(1)+\frac{ m_0 }{2}\widehat{[2u_0 v(\omega,r)+v(\omega,r)^2]}(1)+m_0\widehat{P_1(u_0+v(\omega,r))}(1)=0.
    \end{equation}
    Denote the left hand side of \eqref{solvebifurcationcomplex} by $\tilde{F}_{1}(\omega, r)$.
    Then, $\tilde{F}_{1}(\omega, r)$ is a holomorphic function on $\tilde{\Omega}_{\rho_3,C}$, whose restriction to $\Omega_{\rho_3,C}$ is real analytic. As a consequence, to obtain a real-analytic solution on $(0,\rho_3)$ to the bifurcation equation \eqref{solvebifurcation}, we only need to find a holomorphic function $\omega(r) :\mathcal{U}_{\rho_3} \to \mathbb{C}$ such that $\tilde{F}_{1}(\omega(r), r) = 0$ and $\omega(\overline r)=\overline{\omega(r)}$.
    
    Now we analyze the bifurcation equation.
    
    The leading-order term of $\widehat{P_1(u_0+v(\omega,r))}(1)$ depends on whether $p$ is odd or even. 
    To start with, with Theorem \ref{Banachfixed1} and \eqref{Lipchitz1}, we have
    \begin{equation*}
        \|\widehat{ P_1(u_0) }- \widehat{ P_1(u_0 + v(\omega,r) })\|_{\sigma,s-2}\leq  f_1'(C_{\sigma, s}(\| \hat{ u }_0\|_{\sigma,s}+ M_0|r|^p ))\cdot M_0|r|^p.
    \end{equation*}
    Moreover, with \eqref{nonlinearestimate1}, we have
    \begin{equation*}
        \|\widehat{ P_{1,>p+1}(u_0 + v(\omega,r) })\|_{\sigma,s-2}\leq f_{1,>p+1}(C_{\sigma, s}\|\hat{u}_0\|_{\sigma,s} + M_0|r|^p).
    \end{equation*}
    Thus, when $p$ is odd, it can be rewritten into
    \begin{align*}
        \widehat{P_1(u_0+v)}(1)=&\widehat{P_1(u_0)}(1)+O(r^{2p-1})\\
        =& \sum_{k_1 +\cdots + k_p = 1} \hat{u}_0(k_1)\cdots \hat{u}_0(k_p) + \widehat{u_0^{p+1}}(1) + O(r^{p+2})\\
        =& \binom{p}{(p-1)/2}\hat{u}_0(1)^{(p+1)/2}\hat{u}_0(-1)^{(p-1)/2}\\
        &+\sum_{l \geq p+2}\sum_{\substack{k_1 +\cdots + k_p = 1\\ |k_1| +\cdots + |k_p| = l}} \hat{u}_0(k_1)\cdots \hat{u}_0(k_p) + O(r^{p+1})\\
        =& \binom{p}{(p-1)/2}\hat{u}_0(1)^{(p+1)/2}\hat{u}_0(-1)^{(p-1)/2}+O(r^{p+1}),
    \end{align*}
    whose leading order is $p$.
    When $p$ is even, it can be rewritten into
    \begin{align*}
        &\widehat{P_1(u_0+v)}(1)\\
        =&\widehat{P_1(u_0)}(1)+O(r^{2p-1})\\
        =& \sum_{k_1 +\cdots + k_p = 1} \hat{u}_0(k_1)\cdots \hat{u}_0(k_p) + \sum_{k_1 +\cdots + k_{p+1} = 1} \hat{u}_0(k_1)\cdots \hat{u}_0(k_{p+1}) + O(r^{p+2})\\
        =& p\binom{p-1}{p/2-1}\hat{u}_0(2)\hat{u}_0(1)^{p/2-1}\hat{u}_0(-1)^{p/2} + p\binom{p-1}{p/2+1}\hat{u}_0(-2)\hat{u}_0(1)^{p/2+1}\hat{u}_0(-1)^{p/2-2} \\
        &+ \sum_{l \geq p+3 }\sum_{\substack{k_1 +\cdots + k_p = 1\\ |k_1| +\cdots + |k_p| = l}} \hat{u}_0(k_1)\cdots \hat{u}_0(k_p) \\
        &+ \binom{p+1}{p/2}\hat{u}_0(1)^{p/2+1}\hat{u}_0(-1)^{p/2}
        + \sum_{l\geq p+3}\sum_{\substack{k_1 +\cdots + k_{p+1} = 1\\ |k_1| +\cdots + |k_{p+1}| = l}} \hat{u}_0(k_1)\cdots \hat{u}_0(k_{p+1}) + O(r^{p+2})\\
        =& p\binom{p-1}{p/2-1}\hat{u}_0(2)\hat{u}_0(1)^{p/2-1}\hat{u}_0(-1)^{p/2} + p\binom{p-1}{p/2+1}\hat{u}_0(-2)\hat{u}_0(1)^{p/2+1}\hat{u}_0(-1)^{p/2-2} \\
        &+ \binom{p+1}{p/2}\hat{u}_0(1)^{p/2+1}\hat{u}_0(-1)^{p/2} + O(r^{p+2}),
    \end{align*}
    whose leading order is $p+1$.

    When $p$ is odd, a simple calculation implies that
    \begin{align*}
        |\widehat{u_0v}(1)| &\leq \sum_{k_1+k_2=1} |\hat{u}_0(k_1)\hat{v}(k_2) |
        \lesssim \sum_{ k_1\in \mathbb{Z}\setminus\{0\} } |r|^{|k_1|}\cdot |r|^p
        \lesssim |r|^{p+1},\\
        |\widehat{v^{2}}(1)| &\leq \sum_{k_1+k_2=1} |\hat{v}(k_1)\hat{v}(k_2) |
        \lesssim \sum_{ k_1\in \mathbb{Z}\setminus\{0\} } e^{-\sigma \nu_1|k_1|}\cdot |r|^{2p} \lesssim |r|^{2p}.
    \end{align*}
    Thus, we can obtain the bifurcation equation when $p$ is odd as follows:
    \begin{equation*}
        (\omega-\omega_0)\hat{u}_0(1) + m_0\binom{p}{(p-1)/2}\hat{u}_0(1)^{(p+1)/2}\hat{u}_0(-1)^{(p-1)/2} + O(r^{p+1}) = 0.
    \end{equation*}
    
    When $p$ is even, we have
    \begin{align*}
        &\left|\widehat{u_{0}v}(1) - \hat{u}_0(-1)\hat{v}(2)\right| \leq \sum_{\substack{k_1+k_2=1,\\ k_1\neq -1}} \left|\hat{u}_0(k_1)\hat{v}(k_2) \right| \lesssim |r|^{p+2},\\
        &|\widehat{v^{2}}(1)| \leq \sum_{k_1+k_2=1} \left|\hat{v}(k_1)\hat{v}(k_2) \right| \lesssim \sum_{ k_1\in \mathbb{Z}\setminus\{0\} } e^{-\sigma \nu_1|k_1|}\cdot |r|^{2p} \lesssim |r|^{2p}.
    \end{align*}
    We therefore need to calculate the leading-order term of $\hat{v}(2)$ when $p$ is even.
    
    Recall the Picard iteration sequence for the range equation \eqref{banachform}:
    \begin{equation*}
        \hat{v}_0(\omega, r)=0,\hat{v}_{n+1}(\omega, r)= -L^{-1}(\omega, r)G(\omega, r)(\hat{v}_n(\omega, r)),n\geq 0.
    \end{equation*}
    The first iterate is
    \begin{equation*}
        \hat{v}_1(\omega, r) = \sum_{l\geq 0} (-L_0^{-1}(\omega)L_1(r))^l(-L_0^{-1}(\omega)G(\omega, r)(0)).
    \end{equation*}
    With \eqref{omega0},\eqref{xianyanguji}, the calculation yields
    \begin{align*}
        L_0(\omega)(2, 2) &= \omega-\frac{4\pi^2}{\gamma^2}4m_0^3+\frac{n_0^2}{m_0}\\
        &= (\omega_0-\frac{4\pi^2}{\gamma^2}4m_0^3+\frac{n_0^2}{m_0})+O(r^{p-1})\\
        &= -\frac{12\pi^2}{\gamma^2}m_0^3+O(r^{2}).
    \end{align*}
    Notice that $\left|(\omega-\omega_0)\hat{u}_0(2)\right| \lesssim |r|^{p+1}$, we obtain
    \begin{align*}
        \left(G(\omega, r)(0)\right)(2) &= m_0\binom{p}{p/2-1}\hat{u}_0(-1)^{p/2-1}\hat{u}_0(1)^{p/2+1}+O(r^{p+1}).
    \end{align*}
    Thus, we have
    \begin{align*}
        -L_0^{-1}(\omega) \left( G(\omega, r)(0)\right)(2)
        &= \frac{\gamma^2 \binom{p}{p/2-1}}{12\pi^2m_0^2}\hat{u}_0(-1)^{p/2-1}\hat{u}_0(1)^{p/2+1}+O(r^{p+1}),
    \end{align*}
    and for any $l>0$,
    \begin{align*}
        \left|\left[(-L_0^{-1}(\omega) L_1(r))^l\circ (-L_0^{-1}(\omega, r) G(\omega, r)(0))\right](2)\right| &\lesssim |r|^{p+l}.
    \end{align*}
    Thus
    \begin{equation*}
        \hat{v}_1(2) = \frac{\gamma^2 \binom{p}{p/2-1}}{12\pi^2m_0^2}\hat{u}_0(-1)^{p/2-1}\hat{u}_0(1)^{p/2+1}+O(r^{p+1}).
    \end{equation*}
    With Theorem \ref{Banachfixed1}, we have
    \begin{align*}
        \|\hat{v}_{n} - \hat{v}_1\|_{\sigma,s,P}& \leq \sum_{j=1}^{n-1} |r|^{(p-2)j}  \|\hat{v}_1 - \hat{v}_{0}\|_{\sigma,s,P} \leq \frac{M_0|r|^{2p-2}}{1-|r|^{p-2}},
    \end{align*}
    which implies that
    \begin{equation*}
        \| \widehat{v(\omega,r)}-\hat{v}_1 \|_{\sigma,s,P} \lesssim |r|^{2p-2}.
    \end{equation*}
    Thus,
    \begin{align*}
        \widehat{v(\omega,r)}(2) = \frac{\gamma^2 \binom{p}{p/2-1}}{12\pi^2m_0^2}\hat{u}_0(-1)^{p/2-1}\hat{u}_0(1)^{p/2+1}+O(r^{p+1}).
    \end{align*}
    
    With the above calculations, we can obtain the bifurcation equation when $p$ is even as follows:
    \begin{align*}
        &(\omega-\omega_0)\hat{u}_0(1) + O(r^{p+2}) + \frac{\gamma^2 \binom{p}{p/2-1}}{12\pi^2m_0}\hat{u}_0(-1)^{p/2}\hat{u}_0(1)^{p/2+1} + m_0 \hat{u}_0(1)^{p/2-1}\hat{u}_0(-1)^{p/2-2} \notag\\
        &\cdot \left(p\ \binom{p-1}{p/2-1}\hat{u}_0(2)\hat{u}_0(-1)^2 + p\binom{p-1}{p/2+1}\hat{u}_0(-2)\hat{u}_0(1)^2 + \binom{p+1}{p/2} \hat{u}_0(1)^2  \hat{u}_0(-1)^2 \right) = 0.
    \end{align*}
    
    Notice that $|\hat{u}_0(1)| \sim |r|$. Thus, we can rewrite the bifurcation equation into the following form:
    When $p$ is odd,
    \begin{equation}\label{reducedbifurcation1}
        (\omega-\omega_0) + m_0 \binom{p}{(p-1)/2}\hat{u}_0(1)^{(p-1)/2}\hat{u}_0(-1)^{(p-1)/2} + \mathscr{R}_{odd}(\omega,r) = 0,
    \end{equation}
    where $\mathscr{R}_{odd}(\omega,r) = O(r^p)$.
    When $p$ is even,
    \begin{align}\label{reducedbifurcation2}
        &(\omega-\omega_0) + \mathscr{R}_{even}(\omega,r) + \frac{\gamma^2 \binom{p}{p/2-1}}{12\pi^2m_0}\hat{u}_0(-1)^{p/2}\hat{u}_0(1)^{p/2} + m_0 \hat{u}_0(1)^{p/2-2}\hat{u}_0(-1)^{p/2-2} \notag\\
        &\cdot \left(p\binom{p-1}{p/2-1}\hat{u}_0(2)\hat{u}_0(-1)^2 + p\binom{p-1}{p/2+1}\hat{u}_0(-2)\hat{u}_0(1)^2 + \binom{p+1}{p/2} \hat{u}_0(1)^2  \hat{u}_0(-1)^2 \right) = 0,
    \end{align}
    where $\mathscr{R}_{even}(\omega,r) = O(r^{p+1})$.
    
    Denote the left hand side of \eqref{reducedbifurcation1},\eqref{reducedbifurcation2} by $F_{odd}(\omega,r)$ and $F_{even}(\omega,r)$, respectively.
    Let
    \begin{align}
        \omega_{0,odd}(r) = &\omega_0(r) - m_0 \binom{p}{(p-1)/2}\hat{u}_0(1)^{(p-1)/2}\hat{u}_0(-1)^{(p-1)/2},\label{zhubuoddbeq}\\
        \omega_{0,even}(r) = &\omega_0(r) - \frac{\gamma^2 \binom{p}{p/2-1}}{12\pi^2m_0}\hat{u}_0(-1)^{p/2}\hat{u}_0(1)^{p/2} - m_0 \hat{u}_0(1)^{p/2-2}\hat{u}_0(-1)^{p/2-2}\label{zhubuevenbeq}\\
        \cdot \bigg(p\binom{p-1}{p/2-1}&\hat{u}_0(2)\hat{u}_0(-1)^2 + p\binom{p-1}{p/2+1}\hat{u}_0(-2)\hat{u}_0(1)^2 + \binom{p+1}{p/2} \hat{u}_0(1)^2  \hat{u}_0(-1)^2 \bigg).\label{zhubuevenbeq2}
    \end{align}
    Then $F_{odd}(\omega_{0,odd}(r), r) = O(r^{p}), F_{even}(\omega_{0,even}(r), r) = O(r^{p+1})$. We may suppose
    \[
        |F_{odd}(\omega_{0,odd}(r), r)| \leq \mathscr E_0|r|^p
    \]
    for a constant $\mathscr E_0>0$.
    With \eqref{Lipchitz1} and \eqref{qiudaoguji}, we have
    \begin{align}
        \dfrac{\partial F_{\cdot}}{\partial \omega}(\omega,r) &= 1 + \frac{m_0}{\hat{u}_0(1)} \left[ \hat{u}_0*\left(\frac{\partial \hat{v}}{\partial \omega}\right)(1) + \hat{v}*\left(\frac{\partial \hat{v}}{\partial \omega}\right)(1) \right] + \frac{m_0}{\hat{u}_0(1)} \widehat{(DP_1(u_0+v)[\partial_{\omega}v])}(1)\notag \\
        &= 1 + O(r^2)+O(r^{p})\neq 0,\label{setm}
    \end{align}
    and
    \begin{align*}
        \left| \dfrac{\partial^2 F_{\cdot}}{\partial \omega^2}(\omega,r) \right| &= \bigg| \frac{m_0}{\hat{u}_0(1)} \left[ \hat{u}_0*\left(\frac{\partial^2 \hat{v}}{\partial \omega^2}\right)(1) + \left(\frac{\partial \hat{v}}{\partial \omega}\right)*\left(\frac{\partial \hat{v}}{\partial \omega}\right)(1) + \hat{v}*\left(\frac{\partial^2 \hat{v}}{\partial \omega^2}\right)(1) \right] \\
        &+ \frac{m_0}{\hat{u}_0(1)} \bigg[ \widehat{\left(D^2P_1(u_0+v)[\partial_{\omega}v,\partial_{\omega}v]\right)}(1) + \widehat{\left(DP_1(u_0+v)[\partial_{\omega}^2v]\right)}(1)\bigg] \bigg|  \\
        &\lesssim |r|^2.
    \end{align*}
    for $\cdot = \textup{odd, even}$.
    With the uniform estimate \eqref{setm}, we set $m:= \inf_{(\omega,r)\in\overline{\tilde\Omega}_{\rho_3,C}} \left| \partial_\omega F_{odd} (\omega,r) \right| >0$.
    Moreover, by the preceding estimate for the second derivative, there
    exists a constant $M>0$ such that
    \[
        \left| \partial_\omega^2F_{odd}(\omega,r) \right| \leq M|r|^2, \qquad (\omega,r)\in\overline{\tilde \Omega}_{\rho_3,C}.
    \]
    Notice that with \eqref{zhubuoddbeq},\eqref{zhubuevenbeq} and \eqref{zhubuevenbeq2}, there exists a constant $C_{\mathrm{odd}}>0$ such that
    \[
        \left| \omega_{0,\mathrm{odd}}(r) - \omega_0(r) \right| \leq C_{\mathrm{odd}}|r|^{p-1}.
    \]
    Choose $C>2C_{\mathrm{odd}}$. After reducing $\rho_3$ if necessary, we may
    assume that
    \[
        \frac{\mathscr E_0}{m} \sum_{l\geq1} |r|^{2^{l-1}p} \leq \left( \frac C2-C_{\mathrm{odd}} \right) |r|^{p-1}.
    \]
    For $l\geq 1$, let
    \begin{align*}
        E_{l-1}(r) &:= |F_{odd}(\omega_{l-1,odd}(r),r)|,\\
        \Delta_{l,odd}(r) &:= -\partial _{\omega} F_{odd}(\omega_{l-1,odd}(r),r)^{-1} F_{odd}(\omega_{l-1,odd}(r),r),\\
        \omega_{l,odd}(r) &:= \omega_{0,odd}(r)+\sum_{l'=1}^{l}\Delta_{l',odd}(r).
    \end{align*}
    Fix $l=1$. Notice that $\omega_{0,odd}(r) \in I_C(r)$. Since $\omega_{0,odd}(r),F_{odd}(\omega,r)$ are real analytic on $\Omega_{\rho_3,C}$ and holomorphic on $\tilde\Omega_{\rho_3,C}$, we have
    \begin{align*}
        &\omega_{0,odd}(\overline r) = \overline{\omega_{0,odd}(r)},\\
        &\Delta_{1,odd}(\overline r) = -\partial _{\omega} F_{odd}(\overline{\omega_{0,odd}(r)},\overline r)^{-1} F_{odd}(\overline{\omega_{0,odd}(r)},\overline r) = \overline{\Delta_{1,odd}(r)}.
    \end{align*}
    Hence $\Delta_{1,odd}(r)$ is holomorphic in $\mathcal{U}_{\rho_3}$ and $|\Delta_{1,odd}(r)| \leq m^{-1}E_0(r)$. Moreover, since
    \[
        \left| \omega_{0,\mathrm{odd}}(r) - \omega_0(r) \right| + |\Delta_{1,odd}(r)| < C|r|^{p-1},
    \]
    the function $\omega_{0,\mathrm{odd}}(r)+ t\Delta_{1,odd}(r) \in I_C(r)$ for $t\in [0,1]$.
    This implies that
    \begin{align*}
        &|F_{odd}(\omega_{0,odd}(r)+\Delta_{1,odd}(r),r)| \\
        \leq &|F_{odd}(\omega_{0,odd}(r),r)+\partial _{\omega} F_{odd}(\omega_{0,odd}(r),r)\Delta_{1,odd}(r)|\\
        &+ |\Delta_{1,odd}(r)^2\int_0^1 \partial_{\omega}^2F_{odd}(\omega_{0,odd}(r)+t\Delta_{1,odd}(r),r)(1-t) dt| \\
        \leq &M|r|^2\cdot |\Delta_{1,odd}(r)|^2 \leq \frac{M|r|^2}{m^2}E_0(r)^2.
    \end{align*}
    Suppose the preceding results hold for $l-1$.
    Then, $\Delta_{l,odd}(r)$ is holomorphic in $\mathcal{U}_{\rho_3}$ and $|\Delta_{l,odd}(r)|\leq m^{-1}E_{l-1}(r)$.
    If $\frac{M\mathscr E_0}{m^2}\rho_3^2 < 1$, we claim that $E_{l'}(r) \leq \mathscr E_0|r|^{2^{l'}p}$ for $l'<l$. It is easy to see that $E_0(r)$ satisfies the estimate. If it holds for $E_{l'-1}(r)$, then
    \begin{align*}
        E_{l'}(r) &\leq M|r|^2|\Delta_{l'}(r)|^2 \leq \frac{M|r|^2}{m^2}E_{l'-1}(r)^2 \\
        &\leq \frac{M\mathscr E_0 |r|^2}{m^2}\cdot \mathscr E_0|r|^{2^{l'}p}\leq \mathscr E_0 |r|^{2^{l'}p}.
    \end{align*}
    Thus, for $t \in [0,1]$,
    \begin{align*}
        \left| \omega_{l-1,\mathrm{odd}}(r) + t\Delta_{l,odd}(r) - \omega_0(r) \right| &\leq \left| \omega_{0,\mathrm{odd}}(r) - \omega_0(r) \right| + \sum_{l'=1}^{l}|\Delta_{l',odd}(r)|\\
        &\leq C_{\mathrm{odd}}|r|^{p-1} + m^{-1}\mathscr E_0\sum_{l\geq 1} |r|^{2^{l-1}p} < C|r|^{p-1},
    \end{align*}
    which implies that $\omega_{l-1,\mathrm{odd}}(r) + t\Delta_{l,odd}(r) \in I_C(r)$.
    As a consequence,
    \begin{align*}
        &|F_{odd}(\omega_{l-1,odd}(r)+\Delta_{l,odd}(r),r)| \\
        = &\left|\Delta_{l,odd}(r)^2\int_0^1 \partial_{\omega}^2F_{odd}(\omega_{l-1,odd}(r)+t\Delta_{l,odd}(r),r)(1-t) dt\right|\\
        \leq &\frac{M|r|^2}{m^2}E_{l-1}(r)^2 \leq \frac{M\mathscr E_0|r|^2}{m^2} \cdot \mathscr E_0|r|^{2^lp} \leq \mathscr E_0|r|^{2^lp}.
    \end{align*}
    It also follows inductively that
    \[
        \omega_{l,odd}(\overline r) = \overline{\omega_{l,odd}(r)}, \quad \Delta_{l,odd}(\overline r) = \overline{\Delta_{l,odd}(r)}.
    \]
    
    Thus, the series
    \begin{align*}
        \sum_{l\geq 1} |\Delta_{l,odd}(r)| \leq m^{-1}\mathscr E_0\sum_{l\geq 1} |r|^{2^{l-1}p} \lesssim |r|^p 
    \end{align*}
    converges absolutely and locally uniformly on $\mathcal{U}_{\rho_3}$. Hence we obtain a holomorphic function
    \begin{equation}
        \omega_{odd}(r) = \omega_{0,odd}(r) + \sum_{l=1}^{\infty}\Delta_{l,odd}(r).\label{jieoddbeq}
    \end{equation}
    Since $F_{odd}$ is continuous, $F_{odd}(\omega_{odd}(r),r) = \lim_{l\to\infty}F_{\mathrm{odd}}(\omega_{l,odd}(r),r) =0$.
    Thus, $\omega_{odd}(r)$ is a holomorphic solution of the equation \eqref{reducedbifurcation1} on $\mathcal{U}_{\rho_3}$.
    Moreover, since $\omega_{l,odd}(\overline r) = \overline{\omega_{l,odd}(r)}$ hold for any $l\geq 1$, the solution satisfies $\omega_{odd}(\overline r) = \overline{\omega_{odd}(r)}$.
    Similarly, we can obtain a real-analytic solution
    \begin{equation}
        \omega_{even}(r) = \omega_{0,even}(r) + \sum_{l=1}^{\infty}\Delta_{l,even}(r) \label{jieevenbeq}
    \end{equation}
    of the equation \eqref{reducedbifurcation2} on $(0,\rho_3)$.
    
    Therefore, after first choosing $C$ sufficiently large and then reducing $\rho_3$, if necessary, the a priori restriction \eqref{xianyanguji} is satisfied. Thus, the functions \eqref{jieoddbeq}, \eqref{jieevenbeq} are real-analytic solutions to the bifurcation equation \eqref{solvebifurcation} on $(0,\rho_3)$.
\end{proof}

The proofs of Theorem \ref{Banachfixed1} and Theorem \ref{thm:solveomega} are also valid for the KP-I equation with $u_0=u_0^{(1)}, q=q^{(1)}, \mathcal{U}_{\rho_0} = \mathcal{U}_{\rho_0}^{(1)}$, and we do not provide a detailed description here. 

\section{Persistence of bi-periodic two-gap solutions}\label{sec:two-gap}

As in the one-gap case, we begin by recording the two-gap specialization of the parameterization and Fourier expansion established in Theorem~\ref{thm6.1} and Corollary~\ref{thm6.1corollary}.
\begin{corollary}\label{construction2}
Suppose $\dfrac{4\pi^2}{\gamma^2}\notin\mathbb{Q}$, and let $(m_j,n_j)\in\mathbb{Z}^2$, $j=1,2$, satisfy $m_1m_2\neq 0$ and $m_1 n_2-m_2 n_1\neq 0$. 
Set $\boldsymbol m=(m_1,m_2),\  \boldsymbol n=(n_1,n_2)$, and
\[
    \kappa := \frac{\gamma^2(m_2n_1-m_1n_2)^2 - 12\pi^2m_1^2m_2^2(m_1-m_2)^2}{\gamma^2(m_2n_1-m_1n_2)^2 - 12\pi^2m_1^2m_2^2(m_1+m_2)^2} \neq 0,\quad
    E_{12}= \pi\epsilon_{\kappa} = 
        \begin{cases}
            \pi, \textup{when }\kappa<0,\\
            0, \textup{when }\kappa>0.
        \end{cases}
\]

$(i)$
Then, there exists a constant $\tilde\rho_0 \in (0,1)$, a symmetric complex neighborhood $\mathcal{U}_{\tilde\rho_0}$ of the rectangle $(0,\tilde\rho_0)^2$, and holomorphic functions $B_{11}(\boldsymbol r), B_{22}(\boldsymbol r), B_{real,12}(\boldsymbol r), \omega_{1,0}(\boldsymbol r),\omega_{2,0}(\boldsymbol r)$ on $\mathcal{U}_{\tilde\rho_0}$ with real-analytic restrictions to $(0,\tilde\rho_0)^2$ such that
\begin{align}
    &\boldsymbol r \in \mathcal{U}_{2\tilde\rho_0} \Leftrightarrow \overline{\boldsymbol r} \in \mathcal{U}_{2\tilde\rho_0},\quad B_{real,12}(\boldsymbol r)= \log |\kappa| +O(\boldsymbol r^2),\label{B12}\\
    &B_{11}(\boldsymbol r)=2\log r_1+O(\boldsymbol r^2),\quad
    B_{22}(\boldsymbol r)=2\log r_2+O(\boldsymbol r^2),\label{Bjj}
\end{align}
where $\log$ denotes the holomorphic branch of the logarithm on the right half-plane.

Denote $B_{12}(\boldsymbol r) := B_{real,12}(\boldsymbol r) + iE_{12},\ B( \boldsymbol r) = \begin{pmatrix}
    B_{11}(\boldsymbol r) &B_{12}(\boldsymbol r)\\
    B_{12}(\boldsymbol r) &B_{22}(\boldsymbol r)
\end{pmatrix}$, $\boldsymbol \omega_{0}(\boldsymbol r) = (\omega_{1,0}(\boldsymbol r), \omega_{2,0}(\boldsymbol r))$. For every $\boldsymbol d=(d_1, d_2)\in \mathbb{R}^2$ and $\boldsymbol r \in (0,\tilde\rho_0)^2$,
\begin{align}
    u_{0}(t,x,y;\boldsymbol r)=12\frac{\partial^2}{\partial x^2}
    \log \theta\left(
        \dfrac{2\pi i}{\gamma} \left(\boldsymbol\omega_{0}(\boldsymbol r)t + \boldsymbol m x + \boldsymbol n y
        + \boldsymbol d\right);
        B( \boldsymbol r)
    \right)\label{two-gap1}
\end{align}
is a bi-periodic two-gap solution to the KP-I case of the equation \eqref{kp}.

$(ii)$
Let $q_j(\boldsymbol r)$ and $\mathcal L_{\boldsymbol k}(\boldsymbol r)$, $j=1,2,\ \boldsymbol k\in\mathbb Z^2\setminus{\boldsymbol0}$, be the functions obtained from Corollary \ref{thm6.1corollary} for $g=2$. Then, there exist constants $\eta\in(0,1)$ and $\rho_0\in(0,\tilde\rho_0)$ such that for $\boldsymbol r\in \mathcal{U}_{\rho_0}$,
\[
    |q_j(\boldsymbol r)| \leq \eta,\quad |q_j(\boldsymbol r)| = |r_j| + O(\boldsymbol r^3),\quad 
    |\mathcal{L}_{\boldsymbol k}(\boldsymbol r)| \leq C_\eta\eta^{-|\boldsymbol k|_1},
\]
where $C_\eta$ is independent of $\boldsymbol k$ and $\boldsymbol r$.

For $j=1,2$, denote $\zeta_j = \omega_{j,0}(\boldsymbol r)t + m_jx + n_jy + d_j$.
As a consequence, the solution \eqref{two-gap1} admits the following Fourier expansion on $(0,\rho_0)^2:$
\begin{equation*}
    u_0(t,x,y;\boldsymbol r)
    =
    \tilde {u}_0(\zeta_1,\zeta_2;\boldsymbol r)
    =
    \sum_{\boldsymbol k\in\mathbb Z^2\setminus\{\boldsymbol0\}}
    \widehat{\widetilde u}_0(\boldsymbol k;\boldsymbol r)
    e^{\frac{2\pi i}{\gamma}
       (k_1\zeta_1+k_2\zeta_2)}.
\end{equation*}
Here, for any $\boldsymbol k=(k_1,k_2) \in \mathbb{Z}^2$, $\widehat{\widetilde u}_0(\boldsymbol k;\boldsymbol r)$ is real analytic on $(0,\rho_0)^2$ and
\begin{equation*}
    \widehat{\widetilde u}_0(-\boldsymbol k;\boldsymbol r) = \widehat{\widetilde u}_0(\boldsymbol k;\boldsymbol r)  =  \frac{48\pi^2}{\gamma^2} (m_1k_1+m_2k_2)^2
    q_1(\boldsymbol r)^{|k_1|}  q_2(\boldsymbol r)^{|k_2|} \mathcal{L}_{\boldsymbol k}(\boldsymbol r).
\end{equation*}
\end{corollary}
\begin{remark}
    Since $\frac{4\pi^2}{\gamma^2} = \frac{4}{3}\left(\frac{\sqrt{3}\pi}{\gamma} \right)^2$, the restriction $\frac{\sqrt{3}\pi}{\gamma}\notin\mathbb{Q}$ in Theorem \ref{thm6.1} follows from the assumption $\frac{4\pi^2}{\gamma^2}\notin\mathbb{Q}$.
\end{remark}

\subsection{Lyapunov-Schmidt reduction}
\begin{definition}
    We denote by $\ell_e^2(\mathbb{T}_{\gamma}^2)$ the complex Hilbert space of zero-$x$-mean even Fourier coefficient sequences
    \[
        \ell_e^2(\mathbb{T}_\gamma^2) := \left\{ \widehat u = \left(\hat{u}(\boldsymbol k) \right)_{k \in \mathbb{Z}^2}: \text{for }\boldsymbol k\in S_0,\widehat u(\boldsymbol k)=0,\ \text{for }\boldsymbol k\in\mathbb Z^2,\widehat u(-\boldsymbol k)=\widehat u(\boldsymbol k),\ \sum_{\boldsymbol k\in\mathbb Z^2} |\widehat u(\boldsymbol k)|^2<\infty \right\}.
    \]
     We also define its real subspace by $\ell_{re}^2(\mathbb{T}_\gamma^2) := \left\{ \widehat u\in\ell_e^2(\mathbb{T}_\gamma^2): \widehat u(\boldsymbol k)\in\mathbb R \text{ for every }\boldsymbol k\in\mathbb Z^2 \right\}$.
\end{definition}
Notice that $\hat{\tilde{u}}_{0}(\boldsymbol k) = 0$ for $\boldsymbol k \in S_0$, and $\hat{\tilde{u}}_{0}(-\boldsymbol k) = \hat{\tilde{u}}_{0}(\boldsymbol k) \in \mathbb{R}$, which implies that $\hat{\tilde{u}}_0 \in \ell_{re}^2(\mathbb{T}_{\gamma}^2)$. We consequently seek solutions to the perturbed KP-I equation in the form $u(t,x,y) = \tilde{u}(\omega_1 t+ m_1x + n_1y + d_1, \omega_2 t+ m_2x + n_2y + d_2)$, where $\hat{\tilde{u}} \in \ell_{re}^2(\mathbb{T}_{\gamma}^2)$. To simplify notation, we abuse notation and let $u_0,u,\hat{u}_0,\hat{u}$ also denote $\tilde{u}_0, \tilde{u}, \hat{\tilde{u}}_0, \hat{\tilde{u}}$ in the following subsections.

With $\zeta_j := \omega_jt+m_jx+n_jy+d_j, j=1,2$,
we have $\partial_x=D_m, \partial_y=D_n,$
where
\[
    D_m :=m_1\partial_{\zeta_1}+m_2\partial_{\zeta_2},
    \qquad
    D_n :=n_1\partial_{\zeta_1}+n_2\partial_{\zeta_2}.
\]
In accordance with \eqref{eq:intro-P21}--\eqref{eq:intro-P22}, define
\[
    P_{2}(u) := P_{2,1}(u) + P_{2,2}(u), \quad
    P_{2,1}(u) := Q_{2,1}(J_x^2u)\big|_{\partial_x=D_m}, \quad 
    P_{2,2}(u) := Q_{2,2}(J_y^1u)\big|_{\partial_y=D_n}.
\]
Then
\begin{align*}
    &P_{2,1}(u) = a_p u^p + a_{p+1} u^{p+1} + P_{2,1,>p+1}(u),\\
    &P_{2,2}(u) = b_p u^p + b_{p+1} u^{p+1} + P_{2,2,>p+1}(u),
\end{align*}
where
\[
    P_{2,1,>p+1}(u) := Q_{2,1,>p+1}(J_x^2u)\big|_{\partial_x=D_m},\quad
    P_{2,2,>p+1}(u) := Q_{2,2,>p+1}(J_y^1u)\big|_{\partial_y=D_n}.
\]
Consequently, $\mathcal{P}_{2}[u] = D_m P_{2}(u)$, and the perturbed KP-I equation is transformed into
\begin{equation}\label{finitegap-kp1}
    \omega_1u_{\zeta_1} + \omega_2u_{\zeta_2} + \left(D_m^3- D_m^{-1}D_n^2\right)u + u D_m u+ D_m P_2(u) = 0.
\end{equation}
Denote $m(\boldsymbol k) = m_1k_1+m_2k_2,\ n(\boldsymbol k) = n_1k_1+n_2k_2$ for $\boldsymbol k = (k_1,k_2) \in \mathbb{Z}^2$.
Passing to the Fourier coefficients, equation \eqref{finitegap-kp1} is equivalent to
\begin{equation}\label{Fouriereq2}
    \frac{2\pi i}{\gamma} \left(k_1\omega_1+k_2\omega_2-\frac{4\pi^2}{\gamma^2}m(\boldsymbol k)^3-\frac{n(\boldsymbol k)^2}{m(\boldsymbol k)}\right)\hat{u}(\boldsymbol k)
    +\frac{2\pi i}{\gamma} \frac{m(\boldsymbol k)}{2} (\hat{u}*\hat{u})(\boldsymbol k)
    + \frac{2\pi i}{\gamma}m(\boldsymbol k)\widehat{P_2(u)}(\boldsymbol k)=0,
\end{equation}
for $\boldsymbol k = (k_1,k_2)\in \mathbb{Z}^2\setminus S_0$.

Notice that the parity conditions \eqref{eq:intro-P21-parity}--\eqref{eq:intro-P22-parity} together with the reality of the Taylor coefficients, imply that $\widehat{P_2(u)}$ depends real analytically on $\hat{u}$ and preserves the real subspace. Thus, if $(\omega_1,\omega_2,\hat{u}) \in \mathbb{C}^2 \times \ell_{re}^2(\mathbb{T}^2_{\gamma})$ solves \eqref{Fouriereq2} and $\hat{u}(1,0),\hat{u}(0,1) \neq 0$, then $\omega_1,\omega_2\in\mathbb{R}$.

Now we derive the leading-order expansions of $\omega_{1,0}(\boldsymbol r),\omega_{2,0}(\boldsymbol r)$.
Denote $\widehat{u_{0}(\boldsymbol r)} = \left( \hat{u}_{0}(\boldsymbol k;\boldsymbol r) \right)_{\boldsymbol k \in \mathbb{Z}^2}$.
For every $\boldsymbol r\in (0,\rho_0)^2$, Lemma \ref{construction2} implies that $\widehat{u_0(\boldsymbol r)},\omega_{1,0}(\boldsymbol r),\omega_{2,0}(\boldsymbol r)$ satisfy the Fourier form of the unperturbed KP-I equations
\begin{equation}\label{kp-fourier2}
    \frac{2\pi i}{\gamma}\left(k_1\omega_1+k_2\omega_2-\frac{4\pi^2}{\gamma^2}m(\boldsymbol k)^3-\frac{n(\boldsymbol k)^2}{m(\boldsymbol k)}\right)\hat{u}(\boldsymbol k) + \frac{2\pi i}{\gamma}\frac{m(\boldsymbol k)}{2} (\hat{u}*\hat{u})(\boldsymbol k)=0,\boldsymbol k\in\mathbb{Z}^2\setminus S_0.
\end{equation}
Denote the left hand side of \eqref{kp-fourier2} for $\boldsymbol k \in \mathbb{Z}^2 \setminus S_0$ by $F_{\boldsymbol k}(\hat{u},\omega_1, \omega_2)$.
The calculation at $\boldsymbol k=(1,0)$ and $\boldsymbol k=(0,1)$ yields $\omega_{1,0}(\boldsymbol r), \omega_{2,0}(\boldsymbol r) \in \mathbb{R}$ and
\begin{align}
    \omega_{1,0}(\boldsymbol r) &= \frac{4\pi^2}{\gamma^2}m_1^3 + \frac{n_1^2}{m_1}  - \frac{192\pi^2m_1^3}{\gamma^2} \frac{\mathcal{L}_{(-1,0)}\mathcal{L}_{(2,0)}}{\mathcal{L}_{(1,0)}}r_1^2\label{omega10line1}\\
    &- \frac{48\pi^2m_2^2}{\gamma^2m_1\mathcal{L}_{(1,0)}}  \left((m_1+m_2)^2\mathcal{L}_{(1,1)}\mathcal{L}_{(0,-1)} + (m_1-m_2)^2\mathcal{L}_{(1,-1)}\mathcal{L}_{(0,1)} \right)r_2^2  + O(\boldsymbol r^{4}),\label{omega10line2}\\
    \omega_{2,0}(\boldsymbol r) &= \frac{4\pi^2}{\gamma^2}m_2^3 + \frac{n_2^2}{m_2} - \frac{192\pi^2 m_2^3}{\gamma^2} \frac{\mathcal{L}_{(0,-1)}\mathcal{L}_{(0,2)}}{\mathcal{L}_{(0,1)}} r_2^2\label{omega20line1}\\
    &- \frac{48\pi^2m_1^2}{\gamma^2m_2\mathcal{L}_{(0,1)}}  \left((m_1+m_2)^2\mathcal{L}_{(1,1)}\mathcal{L}_{(-1,0)} + (m_2-m_1)^2\mathcal{L}_{(-1,1)}\mathcal{L}_{(1,0)} \right) r_1^2+ O(\boldsymbol r^{4})\label{omega20line2}.
\end{align}
on $(0,\rho_0)^2$.
Notice that $\hat{u}_0(\boldsymbol k;\boldsymbol r)$ is a well-defined holomorphic function on $\mathcal{U}_{\rho_0}$, and the series $\sum_{\boldsymbol k_1+\boldsymbol k_2 = \boldsymbol k}\hat{u}_0(\boldsymbol k_1;\boldsymbol r)\hat{u}_0(\boldsymbol k_2;\boldsymbol r)$ converges absolutely and locally uniformly on $\mathcal{U}_{\rho_0}$. Thus, the functions  $F_{\boldsymbol k}(\widehat{u_{0}(\boldsymbol r)}, \omega_{1,0}(\boldsymbol r), \omega_{2,0}(\boldsymbol r))$ admit well-defined holomorphic extensions on the uniform complex neighborhood $\mathcal{U}_{\rho_0}$ for any $\boldsymbol k \in \mathbb{Z}^2\setminus S_0$. As a consequence, $\widehat{u_0(\boldsymbol r)}, \omega_{1,0}(\boldsymbol r), \omega_{2,0}(\boldsymbol r)$ also satisfy the unperturbed KP equation \eqref{kp-fourier2} on $\mathcal{U}_{\rho_0}$. A similar calculation implies that \eqref{omega10line1}--\eqref{omega20line2} hold on $\mathcal{U}_{\rho_0}$, and there exists a constant $C_0$ such that for $j=1,2$,
\begin{equation}
    \left| \omega_{j,0}(\boldsymbol r)- \left(\frac{4\pi^2}{\gamma^2}m_j^3 + \frac{n_j^2}{m_j} \right) \right|\leq C_0 |\boldsymbol r|_{\infty}^{2},\quad \boldsymbol r \in \mathcal{U}_{\rho_0}.\label{omegaj0}
\end{equation}
Thus
\[
    \omega_{j,0}(\boldsymbol 0) = \lim_{\boldsymbol r\to \boldsymbol 0} \omega_{j,0}(\boldsymbol r) = \frac{4\pi^2}{\gamma^2}m_j^3 + \frac{n_j^2}{m_j}, \quad j=1,2.
\]
We set the resonant set of \eqref{Fouriereq2} as
\begin{equation*}
    S_2=\left\{ \boldsymbol k=(k_1,k_2)\in \mathbb{Z}^2\setminus S_0:\ k_1\omega_{1,0}(\boldsymbol 0)+k_2\omega_{2,0}(\boldsymbol 0)=\frac{4\pi^2}{\gamma^2}m(\boldsymbol k)^3+\frac{n(\boldsymbol k)^2}{m(\boldsymbol k)} \right\}.
\end{equation*}
Suppose $\dfrac{4\pi^2}{\gamma^2}$ is irrational and $(m_1,n_1), (m_2,n_2)$ are two linearly independent modes with $m_1m_2\neq 0$. The resonant equation for $k_1,k_2$ becomes
\begin{align*}
    k_1m_1^3+k_2m_2^3 &=(k_1m_1+k_2m_2)^3,\\
    k_1\frac{n_1^2}{m_1}+k_2\frac{n_2^2}{m_2} &=\frac{(k_1n_1+k_2n_2)^2}{k_1m_1+k_2m_2}.
\end{align*}
which is equivalent to 
\begin{align} 
    &k_1m_1^3+k_2m_2^3 =(k_1m_1+k_2m_2)^3,\label{algebra1}\\
    &k_1k_2\left(\frac{n_1^2m_2}{m_1}+\frac{n_2^2m_1}{m_2}-2n_1n_2\right)=0.\label{algebra2}
\end{align}
If $k_1k_2\neq 0$, \eqref{algebra2} leads to
\begin{equation*}
    \frac{(n_1m_2-n_2m_1)^2}{m_1m_2}=0,
\end{equation*}
which contradicts the assumed linear independence.
It follows that $k_1k_2 = 0$, and setting $k_1=0$ or $k_2=0$ in \eqref{algebra1} leads to $S_2=\{(\pm1,0),(0,\pm1)\}$. 

By restricting the perturbed KP equation \eqref{Fouriereq2} to $S_2$ and $\mathbb{Z}^2\setminus (S_0\cup S_2)$, we obtain the bifurcation equation and the range equation as follows:
\begin{align}
    \frac{2\pi i}{\gamma} \left(k_1\omega_1+k_2\omega_2-\frac{4\pi^2}{\gamma^2}m(\boldsymbol k)^3-\frac{n(\boldsymbol k)^2}{m(\boldsymbol k)}\right)\hat{u}(\boldsymbol k)
    +\frac{2\pi i}{\gamma} \frac{m(\boldsymbol k)}{2} (\hat{u}*\hat{u})(\boldsymbol k)& \notag\\
    + \frac{2\pi i}{\gamma}m(\boldsymbol k)\widehat{P_2(u)}(\boldsymbol k)=0,\quad \boldsymbol k=(k_1,k_2)\in S_2&,\label{qfangcheng2}\\
    \frac{2\pi i}{\gamma} \left(k_1\omega_1+k_2\omega_2-\frac{4\pi^2}{\gamma^2}m(\boldsymbol k)^3-\frac{n(\boldsymbol k)^2}{m(\boldsymbol k)}\right)\hat{u}(\boldsymbol k)
    +\frac{2\pi i}{\gamma} \frac{m(\boldsymbol k)}{2} (\hat{u}*\hat{u})(\boldsymbol k)& \notag\\
    + \frac{2\pi i}{\gamma}m(\boldsymbol k)\widehat{P_2(u)}(\boldsymbol k)=0,\quad\boldsymbol k=(k_1,k_2)\in \mathbb{Z}^2\setminus (S_0\cup S_2)&.\label{pfangcheng2}
\end{align}

\subsection{Solving the range equation}

In this subsection, we solve the range equation of the KP-I equation under the Hamiltonian perturbation $\mathcal{P}_2$.
To start with, we introduce the following spaces.

\begin{definition}
    Let $\sigma'>0, s_1>\frac{1}{2}, s_2>\frac{1}{2}$. We define the following spaces:
    \begin{itemize}
        \item Denote $\nu_2 = \frac{2\pi}{\gamma}$. Let the anisotropic analytic Sobolev space $H^{\sigma',s_1,s_2,\nu_2} := \{ \hat{u} \in \ell^2_e (\mathbb{T}_{\gamma}^2): \hat{u}(\boldsymbol k) = 0,\textup{ for }\boldsymbol k\in S_0,\ \|\hat{u}\|_{\sigma',s_1,s_2,\nu_2}<\infty\}$, where
        \begin{equation*}
            \|\hat{u}\|_{\sigma',s_1,s_2,\nu_2}^2 = \sum_{\boldsymbol k=(k_1,k_2)\in \mathbb{Z}^2 } |\hat{u}(\boldsymbol k)|^2 e^{2\sigma'\nu_2|(m(\boldsymbol k),n(\boldsymbol k))|_1} \langle \nu_2 m(\boldsymbol k) \rangle^{2s_1} \langle \nu_2 n(\boldsymbol k)\rangle^{2s_2} < \infty.
        \end{equation*}
        Equipped with this norm, $H^{\sigma',s_1,s_2,\nu_2}$ is a complex Banach space.
        \item Let $P_{\mathbb{Z}^2\setminus (S_0\cup S_2)}$ be the Fourier projection defined by
        \[
            \left(P_{\mathbb{Z}^2\setminus (S_0\cup S_2)}\hat{u}\right)(\boldsymbol k) = 
            \begin{cases}
                \hat{u}(\boldsymbol k), &\boldsymbol k\in \mathbb{Z}^2\setminus (S_0\cup S_2),\\
                0, &\boldsymbol k\in S_0\cup S_2.
            \end{cases}
        \]
        We set
        \[ 
            P\ell_e^2(\mathbb T_\gamma^2) := \{P_{\mathbb{Z}^2\setminus (S_0\cup S_2)}\hat u:\hat u\in\ell_e^2(\mathbb T_\gamma^2)\}, \quad PH^{\sigma',s_1,s_2,\nu_2} := \{P_{\mathbb{Z}^2\setminus (S_0\cup S_2)}\hat u:\hat u\in H^{\sigma',s_1,s_2,\nu_2}\}.
        \]
        For every $\hat u\in H^{\sigma',s_1,s_2,\nu_2}$, define the projected seminorm
        \[ 
            \|\hat u\|_{\sigma',s_1,s_2,\nu_2,P} := \|P_{\mathbb{Z}^2\setminus (S_0\cup S_2)}\hat u\|_{\sigma',s_1,s_2,\nu_2}. 
        \]
        Its restriction to $PH^{\sigma',s_1,s_2,\nu_2}$ is a norm and satisfies $\|\hat u\|_{\sigma',s_1,s_2,\nu_2,P} = \|\hat u\|_{\sigma',s_1,s_2,\nu_2},\ \hat u\in PH^{\sigma',s_1,s_2,\nu_2}.$ Since $P_{\mathbb{Z}^2\setminus (S_0\cup S_2)}$ is a bounded projection, $PH^{\sigma',s_1,s_2,\nu_2}$ is a closed subspace of $H^{\sigma',s_1,s_2,\nu_2}$ and hence is a Banach space with respect to this norm.
        We also define its real subspace by
        \[
            PH_{\mathrm{re}}^{\sigma',s_1,s_2,\nu_2} := PH^{\sigma',s_1,s_2,\nu_2} \cap  \ell_{\mathrm{re}}^2(\mathbb T_\gamma^2).
        \]
        \item Set $t_1>\frac{1}{2}, t_2>\frac{1}{2}$. Let the bounded linear operator space $\mathcal{L}(PH^{\sigma',s_1,s_2,\nu_2}, PH^{\sigma',t_1,t_2,\nu_2})$ be the Banach space with norm
        \begin{equation*}
            \|T\|_{\sigma',(s_1,s_2),(t_1,t_2),\nu_2,P} = \sup_{0\neq \hat{u} \in PH^{\sigma',s_1,s_2,\nu_2}} \frac{\| T\hat{u}\|_{\sigma',t_1,t_2,\nu_2,P}}{\|\hat{u}\|_{\sigma',s_1,s_2,\nu_2,P}}.
        \end{equation*}
    \end{itemize}
\end{definition}
\begin{remark}
    Denote $\mathcal{M}=\begin{pmatrix}
            m_1 &m_2\\
            n_1 &n_2
        \end{pmatrix}$.
    Notice that
    \begin{equation*}
        \dfrac{1}{\|\mathcal{M}^{-1}\|_1}|\boldsymbol k|_1\leq |(m(\boldsymbol k),n(\boldsymbol k))|_1\leq \|\mathcal{M}\|_1 |\boldsymbol k|_1, \qquad \|\mathcal{M}^{-1}\|_1 \leq \dfrac{4\|\mathcal{M}\|_{max}}{|\det \mathcal{M}|}.
    \end{equation*}
    Thus, we set $\sigma=\dfrac{|\det \mathcal{M}| }{4\|\mathcal{M}\|_{max}}\sigma'$ in Theorem \ref{thm:intro-two-gap}.
\end{remark}
\begin{remark}
    Fix $s_1,s_2,\nu_2,\rho_0$ and $\sigma'>0$. Since $q_j(\boldsymbol r) = r_j + O(\boldsymbol r^3)$, after decreasing $\rho_0$,
    \[
        \|\mathcal{M}\|_{max}\cdot (\|\mathcal{M}\|_{max}\nu_2)^{s_1+s_2}e^{\sigma'\nu_2\|\mathcal{M}\|_1}C_{\eta}\eta^{-1}\max\{|q_1(\boldsymbol r)|,|q_2(\boldsymbol r)|\}<1,\quad \boldsymbol r\in \mathcal{U}_{\rho_0}.
    \]
    Thus, for $\boldsymbol r \in \mathcal{U}_{\rho_0}$, we have
    \begin{align}
        &\|\hat{u}_0(\boldsymbol r)\|_{\sigma',s_1,s_2,\nu_2}^2 = \sum_{\boldsymbol k\in \mathbb{Z}^2} |\hat{u}_0(\boldsymbol k;\boldsymbol r)|^2 e^{2\sigma'\nu_2|(m(\boldsymbol k),n(\boldsymbol k))|_1}\langle \nu_2 m(\boldsymbol k) \rangle^{2s_1} \langle \nu_2 n(\boldsymbol k) \rangle^{2s_2} \lesssim |\boldsymbol r|_{\infty}^2,\label{twogapu0estimate}\\
        &\|\hat{u}_0(\boldsymbol r)\|_{\sigma',s_1,s_2,\nu_2,P}^2 = \sum_{\boldsymbol k\in \mathbb{Z}^2\setminus S_2} |\hat{u}_0(\boldsymbol k;\boldsymbol r)|^2 e^{2\sigma'\nu_2|(m(\boldsymbol k),n(\boldsymbol k))|_1}\langle \nu_2 m(\boldsymbol k) \rangle^{2s_1} \langle \nu_2 n(\boldsymbol k) \rangle^{2s_2} \lesssim |\boldsymbol r|_{\infty}^4\label{twogapu0estimateP}.
    \end{align}
\end{remark}
Once $\gamma$ is given, $\nu_2$ is a fixed constant.
To simplify notation, we suppress the dependence on $\nu_2$ in the sequel.
Let $\hat{u} = \widehat{u_0(\boldsymbol r)} + \hat{v}, \hat{v}\in PH^{\sigma',s_1,s_2}$.
Denote
\begin{align*}
    &\tilde{\Omega}_{\rho,C} = \{(\omega_1,\omega_2,r_1,r_2)\in \mathbb{C}^4:\boldsymbol r = (r_1,r_2) \in \mathcal{U}_{\rho},
        |\omega_j-\omega_{j,0}(\boldsymbol r)| < C|\boldsymbol r|_{\infty}^{p-1},\ j=1,2\},\\
    &\Omega_{\rho,C} = \tilde{\Omega}_{\rho,C} \cap \mathbb{R}^4,
\end{align*}
for fixed constants $0<\rho \leq \rho_0$ and $C = C(\mathcal{P}_2,\gamma,m_1,m_2,n_1,n_2)$. The constant $C$ will be chosen in subsection \ref{sec:two-gap-2}.

Since $u_0$ satisfies \eqref{kp-fourier2}, the equations \eqref{pfangcheng2} reduce to
\begin{align}
    \left(\frac{k_1\omega_1 +k_2\omega_2}{m(\boldsymbol k)}-\frac{4\pi^2}{\gamma^2}m(\boldsymbol k)^2-\frac{n(\boldsymbol k)^2}{m(\boldsymbol k)^2}\right) \hat{v}(\boldsymbol k) + \frac{k_1(\omega_1-\omega_{1,0})+k_2(\omega_2-\omega_{2,0})}{m(\boldsymbol k)}\hat{u}_0(\boldsymbol k)& \notag\\
    + \frac{1}{2} \widehat{[2u_0v+v^2]}(\boldsymbol k)
    + \widehat{P_2(u_0+v)}(\boldsymbol k) = 0, \textup{ for } \boldsymbol k =(k_1,k_2)\in\mathbb{Z}^2\setminus (S_0\cup S_2)&.\label{Pfangcheng2}
\end{align}
Define $L(\boldsymbol \omega, \boldsymbol r)=L_0(\boldsymbol \omega)+L_1(\boldsymbol r), G(\boldsymbol \omega, \boldsymbol r)=G_1(\boldsymbol r) + G_2(\boldsymbol \omega, \boldsymbol r): PH^{\sigma',s_1,s_2} \to P\ell_e^2(\mathbb{T}_{\gamma}^2)$ for $(\boldsymbol \omega, \boldsymbol r) \in \tilde{\Omega}_{\rho_0,C}$ as follows:

For $\boldsymbol k \in \mathbb{Z}^2 \setminus (S_0\cup S_2)$,
\begin{align*}
    (L_0(\boldsymbol \omega)\hat{v})(\boldsymbol k) =&  \left(\frac{k_1\omega_1+k_2\omega_2}{m(\boldsymbol k)}-\frac{4\pi^2}{\gamma^2}m(\boldsymbol k)^2-\frac{n(\boldsymbol k)^2}{m(\boldsymbol k)^2}\right)\hat{v}(\boldsymbol k), \\
    (L_1(\boldsymbol r)\hat{v})(\boldsymbol k) =& \widehat{u_0v}(\boldsymbol k),\quad
    [G_1(\boldsymbol r)(\hat{v})](\boldsymbol k) = \frac{ 1}{2}\widehat{v^2}(\boldsymbol k) +  \widehat{P_{2,1}(u_0+v)}(\boldsymbol k),\\
    [G_2(\boldsymbol \omega, \boldsymbol r)(\hat{v})](\boldsymbol k) =& \frac{k_1(\omega_1-\omega_{1,0}(\boldsymbol r))+k_2(\omega_2-\omega_{2,0}(\boldsymbol r))}{m(\boldsymbol k)}\hat{u}_0(\boldsymbol k) +\widehat{P_{2,2}(u_0+v)}(\boldsymbol k).
\end{align*}
All four sequences are extended by zero on $S_0\cup S_2$.
Then \eqref{Pfangcheng2} is equivalent to
\begin{equation}
    L(\boldsymbol \omega, \boldsymbol r)(\hat{v})+G(\boldsymbol \omega, \boldsymbol r)(\hat{v}) = 0.\label{banachform2}
\end{equation}
We also denote the entries of the Fourier matrix representation of $L_0(\boldsymbol \omega), L_1(\boldsymbol r), L(\boldsymbol \omega, \boldsymbol r)$ at $(\boldsymbol k,\boldsymbol k')$ by $L_0(\boldsymbol \omega)(\boldsymbol k,\boldsymbol k'), L_1(\boldsymbol r)(\boldsymbol k,\boldsymbol k'), L(\boldsymbol \omega, \boldsymbol r)(\boldsymbol k,\boldsymbol k')$ in the following subsections.

To describe the Taylor expansions of the reduced perturbations, set $I_1 = \{ 0,1,2 \}, I_2 = \{ 0,1 \}$. 
For $\beta_1 = (\beta_{\alpha_1})_{\alpha_1 \in I_1} \in \mathbb{N}^{3},\ \beta_2 = (\beta_{\alpha_2})_{\alpha_2 \in I_2} \in \mathbb{N}^{2}$, set $\boldsymbol{u}^{\beta_1}=\prod_{\alpha_1 \in I_1}(D_m^{\alpha_1}u)^{\beta_{\alpha_1}}, \boldsymbol{u}^{\beta_2}=\prod_{\alpha_2 \in I_2}(D_n^{\alpha_2}u)^{\beta_{\alpha_2}}$.
Since $P_{2,1}(u),P_{2,2}(u)$ are real analytic, we expand them into Taylor series
\begin{equation*}
    P_{2,j}(u)=\sum_{l=p}^{\infty} P_{2,j,l}(u),\quad
    P_{2,j,l}(u)=\sum_{\beta_j=(\beta_{\alpha_j})_{\alpha_j\in I_j},|\beta_j|_1=l} p_{\beta_j}\boldsymbol{u}^{\beta_j},\quad p_{\beta_j} \in \mathbb{R},\quad j=1,2.
\end{equation*}
In particular,
\begin{equation*}
    P_{2,1,p}(u) = a_p u^{p},\quad P_{2,1,p+1}(u) = a_{p+1}u^{p+1},\quad
    P_{2,2,p}(u) = b_p u^{p},\quad P_{2,2,p+1}(u) = b_{p+1}u^{p+1}.
\end{equation*}
We define the scalar majorants
\begin{align*}
    &f_{2,1}(z) := |a_p|\cdot z^p + |a_{p+1}|\cdot z^{p+1} + f_{2,1,>p+1}(z),\\ 
    &f_{2,1,>p+1}(z) := \sum_{l=p+2}^{\infty} \sum_{|\beta_1|_1=l} |p_{\beta_1}|\cdot z^{l}, \quad z> 0,
\end{align*}
and
\begin{align*}
    &f_{2,2}(z) := |b_p|\cdot z^p + |b_{p+1}|\cdot z^{p+1} + f_{2,2,>p+1}(z),\\ 
    &f_{2,2,>p+1}(z) := \sum_{l=p+2}^{\infty} \sum_{|\beta_2|_1=l} |p_{\beta_2}|\cdot z^{l}, \quad z> 0.
\end{align*}
Denote $f_2(z) = f_{2,1}(z) + f_{2,2}(z)$.
Since $Q_{2,1,>p+1}(z_0, z_1, z_2), Q_{2,2,>p+1}(z_0, z_1)$ are analytic in $|(z_0,z_1,z_2)|_\infty< R_{2}$ and $|(z_0,z_1)|_\infty< R_{2}$ respectively, both series $f_{2,1,>p+1}(z)$ and $f_{2,2,>p+1}(z)$ converge for $0<z< R_2$.

\begin{lemma}\label{lm5.2}
    If $s_1,s_2>\frac{1}{2}$ and $\hat{u}_1,\hat{u}_2 \in H^{\sigma',s_1,s_2}$, there exists a constant $C_{\sigma',s_1,s_2}>0$ such that
    \[
    \| \hat{u}_1*\hat{u}_2\|_{\sigma',s_1,s_2} \leq C_{\sigma', s_1,s_2} \| \hat{u}_1\|_{\sigma',s_1,s_2} \| \hat{u}_2\|_{\sigma',s_1,s_2}.
    \]
\end{lemma}

\begin{lemma}\label{lm3.8}  
    Denote $\hat{u}=\hat{u}_0+\hat{v},\hat{u}_1= \hat{u}_0+\hat{v}_1,\hat{u}_2= \hat{u}_0+\hat{v}_2$. Suppose $(\boldsymbol \omega,\boldsymbol r) \in \tilde{\Omega}_{\rho_0,C}$ for some constant $C$, $s_1>\frac{5}{2},s_2>\frac{3}{2}$, and $\hat{v}\in PH^{\sigma',s_1,s_2}$ such that $C_{\sigma', s_1,s_2}\|\hat{u}\|_{\sigma',s_1,s_2}< R_2$. Then 
    \begin{align}
        &\|G_1(\boldsymbol r)(\hat{v})\|_{\sigma',s_1-2,s_2,P} \leq C_{\sigma', s_1,s_2}\|\hat{v}\|_{\sigma',s_1,s_2,P}^2+f_{2,1}(C_{\sigma', s_1,s_2}\|\hat{u}_0+\hat{v}\|_{\sigma',s_1,s_2}),\label{nonlinearestimate2.1}\\
        &\|G_2(\boldsymbol \omega, \boldsymbol r)(\hat{v})\|_{\sigma',s_1,s_2-1,P} \leq C_{\mathcal{M}}C|\boldsymbol r|_{\infty}^{p-1} \|\hat{u}_0\|_{\sigma',s_1,s_2,P} + f_{2,2}(C_{\sigma', s_1,s_2}\|\hat{u}_0+\hat{v}\|_{\sigma',s_1,s_2})\label{nonlinearestimate2.2}.
    \end{align}
    Here $C_{\mathcal{M}}$ is a constant depending on $\mathcal{M}$. Moreover, if $\max( \|\hat{v}_1\|_{\sigma',s_1,s_2,P}, \|\hat{v}_2\|_{\sigma',s_1,s_2,P})\leq M_v$, we have
    \begin{align}
        \|G_1(\boldsymbol r)(\hat{v}_2)-G_1(\boldsymbol r)(\hat{v}_1)\|_{\sigma',s_1-2,s_2,P}
        \leq &\big( C_{\sigma', s_1,s_2}M_v \label{Lipchitz2.1line1}\\
        + f_{2,1}'(C_{\sigma', s_1,s_2}&(\|\hat{u}_0\|_{\sigma',s_1,s_2}+M_v))\big) \|\hat{v}_2-\hat{v}_1\|_{\sigma',s_1,s_2,P},\label{Lipchitz2.1line2}\\
        \|G_2(\boldsymbol \omega, \boldsymbol r)(\hat{v}_2)-G_2(\boldsymbol \omega, \boldsymbol r)(\hat{v}_1)\|_{\sigma',s_1,s_2-1,P}\leq &f_{2,2}'(C_{\sigma', s_1,s_2}(\|\hat{u}_0\|_{\sigma',s_1,s_2}+M_v)) \|\hat{v}_2-\hat{v}_1\|_{\sigma',s_1,s_2,P}.\label{Lipchitz2.2}
    \end{align}
\end{lemma}
\begin{proof}
    With Lemma \ref{lm5.2}, for any $\beta_j$ with $|\beta_j|_1=l$, we have
    \begin{align*}
        &\|\widehat{ p_{\beta_1}  \boldsymbol{u}^{\beta_1} }\|_{\sigma',s_1-2,s_2} \leq |p_{\beta_1}| C_{\sigma', s_1,s_2}^l \|\hat{u}\|_{\sigma',s_1,s_2}^{l},\\
        &\|\widehat{ p_{\beta_2}  \boldsymbol{u}^{\beta_2} }\|_{\sigma',s_1,s_2-1} \leq |p_{\beta_2}| C_{\sigma', s_1,s_2}^l \|\hat{u}\|_{\sigma',s_1,s_2}^{l},
    \end{align*}
    which implies that
    \begin{align*}
        \|\widehat{ P_{2,1}(u) }\|_{\sigma',s_1-2,s_2}\leq \sum_{l=p}^{\infty} \sum_{|\beta_1|_1=l} |p_{\beta_1}| (C_{\sigma', s_1,s_2}\|\hat{u}\|_{\sigma',s_1,s_2})^{l},\\
        \|\widehat{ P_{2,2}(u) }\|_{\sigma',s_1,s_2-1}\leq \sum_{l=p}^{\infty} \sum_{|\beta_2|_1=l} |p_{\beta_2}| (C_{\sigma', s_1,s_2}\|\hat{u}\|_{\sigma',s_1,s_2})^{l}.
    \end{align*}
    If $C_{\sigma', s_1,s_2}\|\hat{u}\|_{\sigma',s_1,s_2}< R_2$, we have
    \begin{align*}
        \|G_1(\boldsymbol r)(\hat{v})\|_{\sigma',s_1-2,s_2,P}
        \leq& \left\| \frac{1}{2}\hat{v}*\hat{v}\right\|_{\sigma',s_1-2,s_2,P}+\| \widehat{P_{2,1} (u_0 +v)} \|_{\sigma',s_1-2,s_2}\\
        \leq& C_{\sigma', s_1,s_2}\|\hat{v}\|_{\sigma',s_1,s_2,P}^2 + f_{2,1}( C_{\sigma', s_1,s_2}\|\hat{u}_0+\hat{v}\|_{\sigma',s_1,s_2}),
    \end{align*}
    and
    \begin{align*}
        &\|G_2(\boldsymbol \omega, \boldsymbol r)(\hat{v})\|_{\sigma',s_1,s_2-1,P}\\
        \leq& \left\| \frac{k_1(\omega_1-\omega_{1,0}(\boldsymbol r))+k_2(\omega_2-\omega_{2,0}(\boldsymbol r))}{m(\boldsymbol k)}\hat{u}_0\right\|_{\sigma',s_1,s_2-1,P} + \| \widehat{ P_{2,2}(u_0 +v)} \|_{\sigma',s_1,s_2-1}\\
        \leq& C_{\mathcal{M}}C|\boldsymbol r|_{\infty}^{p-1}\|\hat{u}_0\|_{\sigma',s_1,s_2,P} + f_{2,2}( C_{\sigma', s_1,s_2}\|\hat{u}_0+\hat{v}\|_{\sigma',s_1,s_2}).
    \end{align*}    
    For the Lipschitz estimate, we use the integral form of the remainder:
    \begin{align*}
        G_1(\boldsymbol r)(\hat{v}_2)-G_1(\boldsymbol r)(\hat{v}_1)
        &=\mathcal{F}\bigg( (v_2-v_1)\cdot\int_0^1 (v_1+t(v_2-v_1))dt\\
        &+\sum_{l=p}^{\infty}\sum_{\alpha_1 \in I_1}D_m^{\alpha_1}(v_2-v_1)\int_0^1\frac{\partial P_{2,1,l}}{\partial (D_m^{\alpha_1}u)}(u_1+ t(u_2-u_1)) dt \bigg).
    \end{align*}
    The first term can be estimated as
    \begin{align*}
        \bigg\|\mathcal{F}\left( (v_2-v_1)\int_0^1 (v_1+t(v_2-v_1))dt \right)\bigg\|_{\sigma',s_1-2,s_2,P} \leq &C_{\sigma',s_1,s_2} \max(\|\hat{v}_1\|_{\sigma',s_1,s_2,P},\|\hat{v}_2\|_{\sigma',s_1,s_2,P}) \\
        &\cdot \|\hat{v}_2-\hat{v}_1\|_{\sigma',s_1,s_2,P},
    \end{align*}
    For the second term, with Lemma \ref{lm5.2} and the coefficient-wise estimates of $P_{2,1}$, we have
    \begin{align*}
        &\sum_{l=p}^{\infty}\left\|\mathcal{F} \left( \sum_{\alpha_1 \in I_1}D_m^{\alpha_1}(v_2-v_1)\int_0^1\frac{\partial P_{2,1,l}}{\partial (D_m^{\alpha_1}u)}(u_1+ t(u_2-u_1)) dt \right) \right\|_{\sigma',s_1-2,s_2}\\
        \leq & \sum_{l=p}^{\infty} \sum_{\alpha_1 \in I_1}\|\widehat{ D_m^{\alpha_1}(v_2-v_1) }\|_{\sigma',s_1-2,s_2} \cdot \sum_{|\beta_1|_1=l}\left\| \mathcal{F}\left( \int_0^1 \beta_{\alpha_1} p_{\beta_1}\boldsymbol{\tilde{u}}(t)^{\beta_1-e_{\alpha_1}} dt \right) \right\|_{\sigma',s_1-2,s_2}\\
        \leq &  \sum_{l=p}^{\infty} \sum_{|\beta_1|_1=l} |p_{\beta_1}| l( C_{\sigma', s_1,s_2} \max_{t\in [0,1]}\|\widehat{ \tilde{u}(t) }\|_{\sigma',s_1,s_2})^{l-1}\cdot\|\hat{v}_2-\hat{v}_1\|_{\sigma',s_1,s_2}\\
        =& f_{2,1}'(C_{\sigma', s_1,s_2} \max( \|\hat{u}_1\|_{\sigma',s_1,s_2}, \|\hat{u}_2\|_{\sigma',s_1,s_2} ) )\cdot \|\hat{v}_2-\hat{v}_1\|_{\sigma',s_1,s_2,P},
    \end{align*}
    where $\tilde{u}(t) =u_1+t(u_2-u_1),\ \boldsymbol{\tilde{u}}(t)^{\beta_1} = \prod_{\alpha_1 \in I_1}\left( D_m^{\alpha_1}\tilde{u}(t)\right)^{\beta_{\alpha_1}}$ for $\alpha_1\in I_1$, and $e_{\alpha_1} = (\delta_{\alpha_1,\alpha_1'})_{\alpha_1'\in I_1}$.
    
    In conclusion, we have
    \begin{align*}
        \|G_1(\boldsymbol r)(\hat{v}_2)-G_1(\boldsymbol r)(\hat{v}_1)\|_{\sigma',s_1-2,s_2,P}\leq &\left( C_{\sigma', s_1,s_2} M_v + f_{2,1}'(C_{\sigma', s_1,s_2}(\|\hat{u}_0\|_{\sigma',s_1,s_2}+M_v))\right) \\
        &\cdot \|\hat{v}_2-\hat{v}_1\|_{\sigma',s_1,s_2,P}.
    \end{align*}
    Similarly, we have
    \begin{align*}
        &\|G_2(\boldsymbol \omega, \boldsymbol r)(\hat{v}_2)-G_2(\boldsymbol \omega, \boldsymbol r)(\hat{v}_1)\|_{\sigma',s_1,s_2-1,P}\\
        \leq&\sum_{l=p}^{\infty}\left\| \mathcal{F}\left( \sum_{\alpha_2 \in I_2}D_n^{\alpha_2}(v_2-v_1)\int_0^1\frac{\partial P_{2,2,l}}{\partial (D_n^{\alpha_2}u)}(u_1+ t(u_2-u_1)) dt \right)\right\|_{\sigma',s_1,s_2-1}\\
        \leq & f_{2,2}'(C_{\sigma', s_1,s_2} \max( \|\hat{u}_1\|_{\sigma',s_1,s_2}, \|\hat{u}_2\|_{\sigma',s_1,s_2} ) )\cdot \|\hat{v}_2-\hat{v}_1\|_{\sigma',s_1,s_2,P}.
    \end{align*}
\end{proof}

Choose $\rho_1 \in (0,\rho_0]$ such that $C_{\sigma', s_1,s_2} \|\hat{u}_0(\boldsymbol r)\|_{\sigma',s_1,s_2} < R_2,\textup{ for } \boldsymbol r\in \mathcal{U}_{\rho_1}$.
\begin{lemma}\label{lm3.9}
    Suppose $s_1>\frac{5}{2},s_2>\frac{3}{2}$. Then, there exist constants $C_2>0$, and $\rho_2 \in (0,\rho_1]$ such that $L^{-1}(\boldsymbol \omega, \boldsymbol r)$ is a bounded operator from $PH^{\sigma',s_1-2,s_2}$ or $PH^{\sigma,s_1,s_2-1}$ to $PH^{\sigma',s_1,s_2}$ for $(\boldsymbol \omega, \boldsymbol r) \in \tilde{\Omega}_{\rho_2,C}$ with bound less than $2C_2$.
\end{lemma}
\begin{proof}
    For $(\boldsymbol \omega, \boldsymbol r) \in \tilde{\Omega}_{\rho_1,C}$, the estimates \eqref{omegaj0} for $\omega_{1,0}(\boldsymbol r), \omega_{2,0}(\boldsymbol r)$ lead to
    \begin{align*}
        &\left|\omega_1-\left( \frac{4\pi^2}{\gamma^2}m_1^3 + \frac{n_1^2}{m_1} \right) \right|\leq C_0\rho_1^{2} + C\rho_1^{p-1},\\
        &\left|\omega_2-\left( \frac{4\pi^2}{\gamma^2}m_2^3 + \frac{n_2^2}{m_2} \right) \right|\leq C_0\rho_1^{2} + C\rho_1^{p-1}.
    \end{align*}
    Suppose $\boldsymbol k=(k_1,k_2)\notin S_0 \cup S_2$.
    If $|m(\boldsymbol k)|^2\geq |n(\boldsymbol k)|$,    
    there exists $N_1 = N_1(m_1,m_2, n_1,n_2,\allowbreak \gamma, C_0,C,\rho_1)$ such that for $|\boldsymbol k|_1\geq N_1$, the denominator
    \begin{align*}
        \left| \frac{k_1\omega_1+k_2\omega_2}{m(\boldsymbol k)}-\frac{4\pi^2}{\gamma^2}m(\boldsymbol k)^2-\frac{n(\boldsymbol k)^2}{m(\boldsymbol k)^2} \right|
        \geq \frac{\pi^2}{\gamma^2}|m(\boldsymbol k)|^2 + \frac{\pi}{\gamma}|n(\boldsymbol k)|,
    \end{align*}
    since
    \begin{align*}
        &\left| k_1\omega_1+k_2\omega_2-\frac{4\pi^2}{\gamma^2}m(\boldsymbol k)^3-\frac{n(\boldsymbol k)^2}{m(\boldsymbol k)} \right|\\
        \geq & \left| \frac{4\pi^2}{\gamma^2}m(\boldsymbol k)^3+\frac{n(\boldsymbol k)^2}{m(\boldsymbol k)} \right| - |k_1\omega_1+k_2\omega_2|\\
        \geq &\left[\frac{2\pi^2}{\gamma^2} \left(\frac{1}{2}|m(\boldsymbol k)|+ \frac{1}{2}|n(\boldsymbol k)|^{1/2} \right)^3 -(|\omega_1|+|\omega_2|)|\boldsymbol k|_1\right] + \frac{\pi^2}{\gamma^2}|m(\boldsymbol k)|^3 + \frac{\pi}{\gamma}|m(\boldsymbol k)|\cdot|n(\boldsymbol k)| \\
        \geq & \left[\frac{\pi^2}{4\gamma^2} (|m(\boldsymbol k)|+|n(\boldsymbol k)|)^{3/2} -(|\omega_1|+|\omega_2|)|\boldsymbol k|_1\right] + \frac{\pi^2}{\gamma^2}|m(\boldsymbol k)|^3 + \frac{\pi}{\gamma}|m(\boldsymbol k)|\cdot|n(\boldsymbol k)| \\
        \geq &\left[\frac{\pi^2|\det \mathcal{M}|^{3/2}}{4\gamma^2\cdot 8\|\mathcal{M}\|_{max}^{3/2}}|\boldsymbol k|_1^{3/2}-(|\omega_1|+|\omega_2|)|\boldsymbol k|_1 \right]+ \frac{\pi^2}{\gamma^2}|m(\boldsymbol k)|^3 + \frac{\pi}{\gamma}|m(\boldsymbol k)|\cdot|n(\boldsymbol k)|\\
        \geq &\frac{\pi^2}{\gamma^2}|m(\boldsymbol k)|^3 + \frac{\pi}{\gamma}|m(\boldsymbol k)|\cdot|n(\boldsymbol k)|.
    \end{align*}
    
    If $|m(\boldsymbol k)|^2\leq |n(\boldsymbol k)|$, there exists $N_2 = N_2(m_1,m_2,n_1,n_2,\gamma,C_0,C,\rho_1)$ such that for $|\boldsymbol k|_1\geq N_2$, the denominator
    \begin{align*}
        &\left| \frac{k_1\omega_1+k_2\omega_2}{m(\boldsymbol k)}-\frac{4\pi^2}{\gamma^2}m(\boldsymbol k)^2-\frac{n(\boldsymbol k)^2}{m(\boldsymbol k)^2} \right|
        \geq \frac{3\pi^2}{\gamma^2}|m(\boldsymbol k)|^2 + \frac{\pi}{\gamma} |n(\boldsymbol k)|,
    \end{align*}
    since
    \begin{align*}
        &\left| k_1\omega_1+k_2\omega_2-\frac{4\pi^2}{\gamma^2}m(\boldsymbol k)^3-\frac{n(\boldsymbol k)^2}{m(\boldsymbol k)} \right|\\
        \geq&\left| \frac{n(\boldsymbol k)^2}{m(\boldsymbol k)} + \frac{4\pi^2}{\gamma^2}m(\boldsymbol k)^3 \right| - |k_1\omega_1+k_2\omega_2|\\
        \geq &\left[ \frac{1}{2} \left( \frac{1}{2}|m(\boldsymbol k)|^2 + \frac{1}{2}|n(\boldsymbol k)| \right)^{\frac{3}{2}}-(|\omega_1|+|\omega_2|) |\boldsymbol k|_1 \right] +\frac{2\pi^2}{\gamma^2}|m(\boldsymbol k)|^3 + \frac{\pi}{\gamma}|m(\boldsymbol k)|\cdot |n(\boldsymbol k)| \\
        \geq &\left[ \left(\frac{|\det \mathcal{M}|}{16\|\mathcal{M}\|_{max}}\right)^{3/2}|\boldsymbol k|_1^{3/2}-(|\omega_1|+|\omega_2|) |\boldsymbol k|_1 \right] +\frac{2\pi^2}{\gamma^2}|m(\boldsymbol k)|^3 + \frac{\pi}{\gamma}|m(\boldsymbol k)|\cdot |n(\boldsymbol k)|\\
        \geq &\frac{2\pi^2}{\gamma^2}|m(\boldsymbol k)|^3 + \frac{\pi}{\gamma}|m(\boldsymbol k)|\cdot |n(\boldsymbol k)|.
    \end{align*}
    
    Set $N = \max\{N_1,N_2\}$.    
    Notice that
    \begin{equation*}
        k_1\omega_{1,0}(\boldsymbol 0)+k_2\omega_{2,0}(\boldsymbol 0)-\frac{4\pi^2}{\gamma^2}m(\boldsymbol k)^3-\frac{n(\boldsymbol k)^2}{m(\boldsymbol k)} \neq 0,\textup{ for }\boldsymbol k\notin S_0\cup S_2.
    \end{equation*}
    For finitely many $\boldsymbol k$ satisfying $|\boldsymbol k|_1\leq N$, there exists a constant $c_1>0$ such that 
    \begin{align*}
        \left| \frac{k_1\omega_{1,0}(\boldsymbol 0)+k_2 \omega_{2,0}(\boldsymbol 0)}{m(\boldsymbol k)}-\frac{4\pi^2}{\gamma^2}m(\boldsymbol k)^2-\frac{n(\boldsymbol k)^2}{m(\boldsymbol k)^2} \right| \geq c_1.
    \end{align*}
    Thus, if $N( C_0 \rho_2^{2} +C \rho_2^{p-1}) \leq \dfrac{c_1}{2}$, the denominator for $|\boldsymbol k|_1\leq N$ satisfies
    \begin{align*}
        \left| \frac{k_1\omega_1+k_2\omega_2}{m(\boldsymbol k)}-\frac{4\pi^2}{\gamma^2}m(\boldsymbol k)^2-\frac{n(\boldsymbol k)^2}{m(\boldsymbol k)^2} \right| &\geq c_1 - N( C_0 \rho_2^{2} +C \rho_2^{p-1})\\
        &\geq \frac{c_1}{2\|\mathcal{M}\|_1^2N^2}(|m(\boldsymbol k)|+|n(\boldsymbol k)|)^2.
    \end{align*}
    
    Thus, there exists a constant $C_2 = C_2(m_1,n_1,m_2,n_2,\gamma, C_0, C, \rho_1)$ such that
    \begin{align*}
        \|L_0^{-1}(\boldsymbol \omega)\hat{v}\|_{\sigma',s_1,s_2,P}^2
        &=\sum_{\boldsymbol k\in \mathbb{Z}^2\setminus (S_0\cup S_2)}\frac{|\hat{v}(\boldsymbol k)|^2(1+|\nu_2m(\boldsymbol k)|^2)^2}{\left|\frac{k_1\omega_1+k_2\omega_2}{m(\boldsymbol k)}-\frac{4\pi^2}{\gamma^2}m(\boldsymbol k)^2-\frac{n(\boldsymbol k)^2}{m(\boldsymbol k)^2}\right|^2} \\
        &\qquad \qquad \qquad\times e^{2\sigma'\nu_2|(m(\boldsymbol k),n(\boldsymbol k))|_1}\langle \nu_2m(\boldsymbol k) \rangle^{2(s_1-2)}\langle \nu_2n(\boldsymbol k) \rangle^{2s_2}\notag\\
        &\leq C_2^2 \|\hat{v}\|_{\sigma',s_1-2,s_2,P}^2<\infty,
    \end{align*}
    and
    \begin{align*}
        \|L_0^{-1}(\boldsymbol \omega)\hat{v}\|_{\sigma',s_1,s_2,P}^2
        &=\sum_{\boldsymbol k\in \mathbb{Z}^2\setminus (S_0\cup S_2)}\frac{|\hat{v}(\boldsymbol k)|^2(1+|\nu_2n(\boldsymbol k)|^2)}{\left|\frac{k_1\omega_1+k_2\omega_2}{m(\boldsymbol k)}-\frac{4\pi^2}{\gamma^2}m(\boldsymbol k)^2-\frac{n(\boldsymbol k)^2}{m(\boldsymbol k)^2}\right|^2} \\
        &\qquad\qquad\qquad \times e^{2\sigma'\nu_2|(m(\boldsymbol k),n(\boldsymbol k))|_1}\langle \nu_2m(\boldsymbol k) \rangle^{2s_1}\langle \nu_2n(\boldsymbol k) \rangle^{2(s_2-1)}\notag\\
        &\leq C_2^2 \|\hat{v}\|_{\sigma',s_1,s_2-1,P}^2<\infty.
    \end{align*}
    
    Since $L_0^{-1}(\boldsymbol \omega)$ is a bounded operator from $PH^{\sigma',s_1-2,s_2}$ or $PH^{\sigma',s_1,s_2-1}$ to $PH^{\sigma',s_1,s_2}$, we rewrite $L^{-1}(\boldsymbol \omega,\boldsymbol r)=(I+L_0^{-1}(\boldsymbol \omega)L_1(\boldsymbol r))^{-1} L_0^{-1}(\boldsymbol \omega)$. By the Neumann-series argument, if $\|L_0^{-1}(\boldsymbol \omega)L_1(\boldsymbol r)\|_{\sigma',(s_1,s_2),(s_1,s_2),P}\leq \frac{1}{2}$, we have
    \begin{equation*}
        \|L^{-1}(\boldsymbol \omega,\boldsymbol r)\|_{\sigma',(s_1-2,s_2),(s_1,s_2),P} \leq 2C_2, \quad\|L^{-1}(\boldsymbol \omega,\boldsymbol r)\|_{\sigma',(s_1,s_2-1),(s_1,s_2),P} \leq 2C_2.
    \end{equation*}
    In fact, we have
    \begin{align*}
        \|L_0^{-1}(\boldsymbol \omega) L_1(\boldsymbol r) \hat{v}\|_{\sigma',s_1,s_2,P}&\leq C_2\|L_1(\boldsymbol r) \hat{v}\|_{\sigma',s_1,s_2-1,P}\\
        &\leq C_2\|\widehat{u_0v}\|_{\sigma',s_1,s_2-1,P}\\
        &\leq C_2C_{\sigma', s_1,s_2}\cdot\|\hat{u}_0\|_{\sigma',s_1,s_2}\|\hat{v}\|_{\sigma',s_1,s_2,P}.
    \end{align*}
    Thus there exist $\rho_2 \in (0,\rho_1]$ such that
    \begin{align*}
        N( C_0 \rho_2^{2} +C \rho_2^{p-1}) &\leq \dfrac{c_1}{2},\\
        \|L_0^{-1}(\boldsymbol \omega) L_1(\boldsymbol r)\|_{\sigma',(s_1,s_2),(s_1,s_2),P} \leq C_2C_{\sigma', s_1, s_2}\|\hat{u}_0\|_{\sigma',s_1,s_2} &\leq \frac{1}{2},\quad \boldsymbol r \in \mathcal{U}_{\rho_2}.
    \end{align*}
\end{proof}

\begin{lemma}\label{321}
    Suppose $s_1>\frac{5}{2},s_2>\frac{3}{2}$ and $p\geq 3$. Then, there exist positive constants $M_0>0$ and $\rho_{3} \in (0,\rho_2]$ such that for any $(\boldsymbol \omega,\boldsymbol r) \in \tilde{\Omega}_{\rho_3,C}$, the functional $-L^{-1}G(\boldsymbol \omega,\boldsymbol r)$ maps
    \begin{equation*}
        \mathcal{B}_{\boldsymbol r} :=\{\hat{v} \in PH^{\sigma',s_1,s_2} : \|\hat{v}\|_{\sigma',s_1,s_2,P} \leq M_0 |\boldsymbol r|_{\infty}^{p}\}
    \end{equation*}
    into itself. Moreover, $-L^{-1}G(\boldsymbol \omega,\boldsymbol r)$ is a uniform contraction on $\mathcal{B}_{\boldsymbol r}$ for $(\boldsymbol \omega, \boldsymbol r) \in \tilde{\Omega}_{\rho_3,C}$, satisfying
    \begin{equation*}
        \|L^{-1}G(\boldsymbol \omega,\boldsymbol r)(\hat{v}_1) - L^{-1}G(\boldsymbol \omega,\boldsymbol r)(\hat{v}_2)\|_{\sigma',s_1,s_2,P} \leq |\boldsymbol r|_{\infty}^{p-2} \|\hat{v}_1 - \hat{v}_2\|_{\sigma',s_1,s_2,P}
    \end{equation*}
    for any $\hat{v}_1, \hat{v}_2 \in \mathcal{B}_{\boldsymbol r}$.
\end{lemma}
\begin{proof}    
    According to Lemmas \ref{lm3.8} and \ref{lm3.9}, the mapping $-L^{-1}G(\boldsymbol \omega,\boldsymbol r)$ defines a self-mapping and a uniform contraction on $\mathcal{B}_{\boldsymbol r}$ provided that the following three conditions are met uniformly for $(\boldsymbol \omega, \boldsymbol r) \in \tilde{\Omega}_{\rho_3,C}$ with $\rho_3 \leq \rho_2$:
    \begin{align}
        C_{\sigma', s_1,s_2}( \|\hat{u}_0(\boldsymbol r)\|_{\sigma',s_1,s_2}+M_0 |\boldsymbol r|_{\infty}^{p} ) &< R_2, \label{ball4}\\
        C_{\mathcal{M}}C |\boldsymbol r|_{\infty}^{p-1}\|\hat{u}_0(\boldsymbol r)\|_{\sigma',s_1,s_2,P}+ C_{\sigma', s_1,s_2} M_0^2 |\boldsymbol r|_{\infty}^{2p}&\label{ball5.1}\\
        + f_2(C_{\sigma', s_1,s_2}(\|\hat{u}_0(\boldsymbol r)\|_{\sigma',s_1,s_2}+ M_0 |\boldsymbol r|_{\infty}^{p})) &\leq \frac{M_0}{2C_2} |\boldsymbol r|_{\infty}^{p}, \label{ball5.2}\\
        2C_2 \left( 2C_{\sigma', s_1,s_2}M_0 |\boldsymbol r|_{\infty}^{p} + f_2'( C_{\sigma', s_1,s_2} (\|\hat{u}_0(\boldsymbol r)\|_{\sigma',s_1,s_2} + 2M_0 |\boldsymbol r|_{\infty}^{p}) ) \right) &\leq |\boldsymbol r|_{\infty}^{p-2}<1 \label{ball6}.
    \end{align}
    With the high-order vanishing property of $f_2$ and $f_2'$ at the origin, there exist positive constants $C_{f_2}$ and $C_{f_2}'$ such that $f_2(z) \leq C_{f_2} |z|^p$ and $f_2'(z) \leq C_{f_2}' |z|^{p-1}$ for $p \geq 3$.
    
    Since \eqref{twogapu0estimate}--\eqref{twogapu0estimateP} hold, there exist positive constants $M_0 = M_0(P_2, C, C_{\mathcal{M}}, C_2, C_{\sigma', s_1, s_2},R_2,\allowbreak \| \hat{u}_0 \|_{\sigma',s_1,s_2}\allowbreak ,\rho_2) >0$ and $\rho_{3} = \rho_{3}(P_2, C_2, C_{\sigma', s_1, s_2}, R_2, \| \hat{u}_0 \|_{\sigma',s_1,s_2}, M_0) >0$ such that the inequalities \eqref{ball4}--\eqref{ball6} hold.    
\end{proof}

With Lemma \ref{321}, we obtain the following theorem.
\begin{theorem}\label{Banachfixed2}
    Suppose $s_1>\dfrac{5}{2}, s_2>\dfrac{3}{2}$ and $p\geq 3$. For $(\boldsymbol \omega, \boldsymbol r) \in \tilde{\Omega}_{\rho_3,C}$, the range equation \eqref{banachform2} admits a unique solution $\hat{v}(\boldsymbol \omega, \boldsymbol r)$ in $\mathcal{B}_{\boldsymbol r} $. In particular, for every $(\boldsymbol \omega, \boldsymbol r) \in \tilde{\Omega}_{\rho_3,C}$,
    \[
        \| \hat{v}(\boldsymbol \omega, \boldsymbol r) \|_{\sigma',s_1,s_2,P} \leq M_0 |\boldsymbol r|_{\infty}^{p}.
    \]
    The solution $\hat{v}(\boldsymbol \omega, \boldsymbol r)$ is holomorphic in $\tilde{\Omega}_{\rho_3,C}$. Moreover, for every $(\boldsymbol \omega, \boldsymbol r) \in \Omega_{\rho_3,C}$, we have $\hat{v}(\boldsymbol \omega, \boldsymbol r) \in PH^{\sigma',s_1,s_2}_{\mathrm{re}}$. Consequently, the restriction of $\hat{v}(\boldsymbol \omega, \boldsymbol r)$ to $\Omega_{\rho_3,C}$ is real analytic.

    Finally, for every $(\boldsymbol \omega, \boldsymbol r) \in \tilde{\Omega}_{\rho_3,C}$, we have the following estimates:
    \begin{equation}\label{qiudaoguji2}
        \left\| \dfrac{\partial \hat{v}}{\partial \omega_j} \right\|_{\sigma',s_1,s_2,P} \lesssim |\boldsymbol r|_{\infty}^2,\quad \left\| \dfrac{\partial^2 \hat{v}}{\partial \omega_i\partial \omega_j} \right\|_{\sigma',s_1,s_2,P} \lesssim |\boldsymbol r|_{\infty}^2,\quad i,j\in \{1,2\}.
    \end{equation}
\end{theorem}
\begin{proof}
    \textit{Step 1}.
    We initiate the standard Picard iteration scheme for the range equation:
    \begin{equation*}
        \hat{v}_0(\boldsymbol \omega, \boldsymbol r) = 0, \quad \hat{v}_{n+1}(\boldsymbol \omega, \boldsymbol r) = -L^{-1}(\boldsymbol \omega, \boldsymbol r)G(\boldsymbol \omega, \boldsymbol r)(\hat{v}_n(\boldsymbol \omega, \boldsymbol r)), \quad n \geq 0.
    \end{equation*}
    With Lemma \ref{321}, it is easy to see that $\|\hat{v}_1\|_{\sigma',s_1,s_2,P} \leq M_0\rho_3^{p}$ uniformly on $\tilde{\Omega}_{\rho_3,C}$.
    Since $\hat{u}_0(\boldsymbol k;\boldsymbol r),\omega_{1,0}(\boldsymbol r)$ and $\omega_{2,0}(\boldsymbol r)$ are holomorphic functions, the operators $L^{-1}(\boldsymbol \omega, \boldsymbol r)$ and $G(\boldsymbol \omega, \boldsymbol r)$ are holomorphic in $\tilde{\Omega}_{\rho_3,C}$. Since $\hat{v}_0 = 0$ is holomorphic, by finite composition rules, each iteration layer $\hat{v}_n(\boldsymbol \omega, \boldsymbol r)$ defines a well-defined holomorphic function of $(\boldsymbol \omega, \boldsymbol r)$ on $\tilde{\Omega}_{\rho_3,C}$.

    With the contraction property in Lemma \ref{321}, we have
    \begin{align*}
        \|\hat{v}_{n+1} - \hat{v}_n\|_{\sigma',s_1,s_2,P} &= \|L^{-1}(\boldsymbol \omega, \boldsymbol r)G(\boldsymbol \omega, \boldsymbol r)(\hat{v}_n) - L^{-1}(\boldsymbol \omega, \boldsymbol r)G(\boldsymbol \omega, \boldsymbol r)(\hat{v}_{n-1})\|_{\sigma',s_1,s_2,P}\\
        &\leq \mathfrak q \|\hat{v}_n - \hat{v}_{n-1}\|_{\sigma',s_1,s_2,P},
    \end{align*}
    uniformly on $\tilde{\Omega}_{\rho_3,C}$, where $\mathfrak q := \rho_3^{p-2} \in (0,1)$. A geometric telescoping expansion reveals that for any shift $k \geq 1$:
    \begin{align*}
        \|\hat{v}_{n+k} - \hat{v}_n\|_{\sigma',s_1,s_2,P} \leq \sum_{j=n}^{n+k-1} \mathfrak q^j \|\hat{v}_1 - \hat{v}_0\|_{\sigma',s_1,s_2,P} \leq \frac{\mathfrak q^n}{1-\mathfrak q} M_0 \rho_3^{p}.
    \end{align*}
    Since $\mathfrak q < 1$ is uniform, $\{\hat{v}_n(\boldsymbol \omega, \boldsymbol r)\}_{n=0}^{\infty}$ forms a Cauchy sequence  in $PH^{\sigma',s_1,s_2}$ that converges locally uniformly on $\tilde{\Omega}_{\rho_3,C}$. By Weierstrass's theorem for holomorphic functions in Banach spaces, the uniform limit $\hat{v}(\boldsymbol \omega, \boldsymbol r) \in PH^{\sigma',s_1,s_2}$ is holomorphic in $\tilde{\Omega}_{\rho_3,C}$.
    Passing to the limit in the Picard iteration shows that $\hat{v}(\boldsymbol \omega, \boldsymbol r) $ solves the range equation \eqref{banachform2}.
    Moreover, Lemma~\ref{321} implies that, for every fixed $(\boldsymbol\omega, \boldsymbol r) \in \tilde{\Omega}_{\rho_3,C}$ and $n\geq0,\ \hat{v}_n(\boldsymbol \omega, \boldsymbol r) \in \mathcal B_{\boldsymbol r}$. Since $\mathcal B_{\boldsymbol r}$ is closed in $PH^{\sigma',s_1,s_2}$, the limit also belongs to $\mathcal B_{\boldsymbol r}$. Therefore, $\|\hat{v}(\boldsymbol \omega, \boldsymbol r)\|_{\sigma',s_1,s_2,P} \leq M_0 |\boldsymbol r|_{\infty}^p$.
    
    \textit{Step 2}.
    We next show that $\hat{v}(\boldsymbol \omega, \boldsymbol r) \in PH_{\mathrm{re}}^{\sigma',s_1,s_2}$ for $(\boldsymbol\omega, \boldsymbol r) \in \Omega_{\rho_3,C}$.
    For a $PH^{\sigma',s_1,s_2}$-valued function $\hat{v}: \tilde{\Omega}_{\rho_3,C} \to PH^{\sigma',s_1,s_2}$, define its reflection $\mathcal{R}$ by
    \begin{equation*}
        (\mathcal{R}\hat{v})( \boldsymbol \omega, \boldsymbol r) := \overline{\hat{v}( \overline{\boldsymbol \omega}, \overline{\boldsymbol r})}.
    \end{equation*}
    Since $\widehat{ u_0(\boldsymbol r)}, \omega_{1,0}(\boldsymbol r), \omega_{2,0}(\boldsymbol r)$ are holomorphic in $\mathcal{U}_{\rho_3}$ with real-analytic restriction to $(0,\rho_3)^2$, by the Schwarz reflection principle,
    \[
        \omega_{1,0}(\overline{\boldsymbol r}) = \overline{\omega_{1,0}(\boldsymbol r)}, \quad \omega_{2,0}(\overline{\boldsymbol r})\allowbreak = \overline{\omega_{2,0}(\boldsymbol r)},\quad \widehat{ u_0( \overline{\boldsymbol r})} = \overline{\widehat{ u_0 (\boldsymbol r)}}.
    \]
    Together with the real analyticity and the parity properties of $Q_{2,1}, Q_{2,2}$, we have
    \[
        \mathcal{R}(L^{-1}(\boldsymbol \omega, \boldsymbol r)G(\boldsymbol \omega, \boldsymbol r)(\hat{v})) = L^{-1}(\boldsymbol \omega, \boldsymbol r)G(\boldsymbol \omega, \boldsymbol r)(\mathcal{R}\hat{v}).
    \]
    
    Since the base approximation $\hat{v}_0 = 0$ satisfies $\mathcal{R}\hat{v}_0 = \hat{v}_0$, mathematical induction guarantees that $\mathcal{R}\hat{v}_n = \hat{v}_n$ holds identically for all $n \geq 0$ across $\tilde{\Omega}_{\rho_3,C}$. By the uniqueness of the fixed point guaranteed by Lemma \ref{321}, the uniform limit must inherit the exact symmetry, yielding $\mathcal{R}\hat{v} = \hat{v}$, which means $\overline{\hat{v}(\overline{\boldsymbol \omega}, \overline{\boldsymbol r})} = \hat{v}(\boldsymbol\omega, \boldsymbol r)$. 
    Restricting $(\boldsymbol\omega, \boldsymbol r) \in \Omega_{\rho_3,C}$, the relation implies that $\overline{\widehat{v(\boldsymbol \omega,\boldsymbol  r)}(\boldsymbol k)} = \widehat{v(\boldsymbol \omega,\boldsymbol r)}(\boldsymbol k) \in \mathbb{R}$. Hence the solution is real-valued and real analytic on $\Omega_{\rho_3,C}$.

    \textit{Step 3}.
    For any $\hat{h} \in PH^{\sigma',s_1,s_2}$, there exists sufficiently small $\varepsilon >0$ such that $\|\hat{v}+\varepsilon\hat{h}\|_{\sigma',s_1,s_2,P} \leq 2M_0 |\boldsymbol r|_{\infty}^{p}$. Hence, the definition of the Fr\'echet derivative with \eqref{ball6} implies
    \begin{align*}
        &\|L^{-1}(\boldsymbol\omega, \boldsymbol r)DG(\boldsymbol\omega, \boldsymbol r)(\hat{v})\|_{\sigma',(s_1,s_2),(s_1,s_2),P}\\
        \leq& 2C_2 \frac{\|DG_1(\boldsymbol r)(\hat{v})[\hat{h}]\|_{\sigma',s_1-2,s_2,P}}{\|\hat{h}\|_{\sigma',s_1,s_2,P}} + 2C_2 \frac{\|DG_2(\boldsymbol\omega, \boldsymbol r)(\hat{v})[\hat{h}]\|_{\sigma',s_1,s_2-1,P}}{\|\hat{h}\|_{\sigma',s_1,s_2,P}}  \\
        =& \lim_{\epsilon \to 0} \frac{2C_2 \|G_1(\boldsymbol r)(\hat{v}+\epsilon \hat{h}) - G_1(\boldsymbol r)(\hat{v})\|_{\sigma',s_1-2,s_2,P}}{\epsilon \|\hat{h}\|_{\sigma',s_1,s_2,P}} \\
        +& \lim_{\epsilon \to 0} \frac{ 2C_2 \|G_2(\boldsymbol\omega, \boldsymbol r)(\hat{v}+\epsilon \hat{h}) - G_2(\boldsymbol\omega, \boldsymbol r)(\hat{v})\|_{\sigma',s_1,s_2-1,P}}{\epsilon \|\hat{h}\|_{\sigma',s_1,s_2,P}}
        \leq \mathfrak q.
    \end{align*}
    For $j=1,2$, differentiating the fixed-point identity $\hat{v} = -L^{-1}(\boldsymbol\omega, \boldsymbol r)G(\boldsymbol\omega, \boldsymbol r)(\hat{v})$ with respect to $\omega_j$ via the chain rule provides:
    \begin{align*}
        \frac{\partial \hat{v}}{\partial \omega_j} = -\bigg[\frac{\partial (L^{-1}(\boldsymbol\omega, \boldsymbol r))}{\partial \omega_j} \bigg] G(\boldsymbol\omega, \boldsymbol r)(\hat{v}) - L^{-1}(\boldsymbol\omega, \boldsymbol r) DG(\boldsymbol\omega, \boldsymbol r)(\hat{v})\left[ \frac{\partial \hat{v}}{\partial \omega_j} \right] - L^{-1}(\boldsymbol\omega, \boldsymbol r) \frac{\partial G(\boldsymbol\omega, \boldsymbol r)}{\partial \omega_j}(\hat{v}),
    \end{align*}
    which can be rearranged into:
    \begin{equation}
        \frac{\partial \hat{v}}{\partial \omega_j} = -\big[ I + L^{-1}(\boldsymbol\omega, \boldsymbol r) DG(\boldsymbol\omega, \boldsymbol r)(\hat{v}) \big]^{-1}\bigg( \frac{\partial (L^{-1}(\boldsymbol\omega, \boldsymbol r))}{\partial \omega_j} G(\boldsymbol\omega, \boldsymbol r)(\hat{v}) + L^{-1}(\boldsymbol\omega, \boldsymbol r) \frac{\partial G(\boldsymbol\omega, \boldsymbol r)}{\partial \omega_j}(\hat{v}) \bigg). \label{partialformula2}
    \end{equation}
    Since $\|L^{-1}(\boldsymbol\omega, \boldsymbol r)DG(\boldsymbol\omega, \boldsymbol r)(\hat{v})\|_{\sigma',(s_1,s_2),(s_1,s_2),P} \leq \mathfrak q < 1$, the Neumann series expansion is valid, guaranteeing
    \begin{equation*}
        \| [ I + L^{-1}(\boldsymbol\omega, \boldsymbol r) DG(\boldsymbol\omega, \boldsymbol r)(\hat{v}) ]^{-1} \|_{\sigma',(s_1,s_2),(s_1,s_2),P} \leq \frac{1}{1-\mathfrak q}.
    \end{equation*}
    Notice that $L^{-1}(\boldsymbol\omega, \boldsymbol r)$ is bounded from $ PH^{\sigma',s_1-2,s_2}$ or $PH^{\sigma',s_1,s_2-1} $ to $PH^{\sigma',s_1,s_2}$, and $\frac{\partial L(\boldsymbol\omega, \boldsymbol r)}{\partial \omega_j}$ is bounded from $PH^{\sigma',s_1,s_2}$ to $PH^{\sigma',s_1,s_2-1}$ with $\frac{\partial L(\boldsymbol\omega, \boldsymbol r)}{\partial \omega_j}(\boldsymbol k) = \frac{k_j}{m(\boldsymbol k)}$. Thus, the operator
    \begin{equation*}
        \dfrac{\partial(L^{-1}(\boldsymbol\omega, \boldsymbol r))}{\partial\omega_j} = -L^{-1}(\boldsymbol\omega, \boldsymbol r) \dfrac{\partial L(\boldsymbol\omega, \boldsymbol r)}{\partial\omega_j} L^{-1}(\boldsymbol\omega, \boldsymbol r)
    \end{equation*}
    is bounded from $PH^{\sigma',s_1-2,s_2}$ or $PH^{\sigma',s_1,s_2-1}$ to $PH^{\sigma',s_1,s_2}$.
    Similarly, we have
    \begin{align*}
        \left\| L^{-1}(\boldsymbol\omega, \boldsymbol r) \frac{\partial G(\boldsymbol\omega, \boldsymbol r)}{\partial \omega_j}(\hat{v}) \right\|_{\sigma',s_1,s_2,P} &= \left\|L^{-1}(\boldsymbol\omega, \boldsymbol r)\frac{k_j}{m(\boldsymbol k)}\hat{u}_0 \right\|_{\sigma',s_1,s_2,P}\\
        &\leq \left\|L^{-1}(\boldsymbol\omega, \boldsymbol r)\frac{k_j}{m(\boldsymbol k)} \right\|_{\sigma',(s_1,s_2),(s_1,s_2),P}\|\hat{u}_0\|_{\sigma',s_1,s_2,P}\lesssim |\boldsymbol r|_{\infty}^{2}.
    \end{align*}
    Taking the norm on both sides of \eqref{partialformula2} yields
    \begin{align*}
        \left\| \frac{\partial \hat{v}}{\partial \omega_j} \right\|_{\sigma',s_1,s_2,P} 
        \leq&  \frac{1}{1-\mathfrak q} \left\| \frac{\partial (L^{-1}(\boldsymbol\omega, \boldsymbol r))}{\partial \omega_j} \right\|_{\sigma',(s_1-2,s_2),(s_1,s_2),P} \|G_1( \boldsymbol r)(\hat{v})\|_{\sigma',s_1-2,s_2,P} \\
        &+ \frac{1}{1-\mathfrak q}\left\| \frac{\partial (L^{-1}(\boldsymbol\omega, \boldsymbol r))}{\partial \omega_j} \right\|_{\sigma',(s_1,s_2-1),(s_1,s_2),P} \|G_2(\boldsymbol\omega, \boldsymbol r)(\hat{v})\|_{\sigma',s_1,s_2-1,P}  \\
        &+ \frac{1}{1-\mathfrak q}\left\|L^{-1}(\boldsymbol\omega, \boldsymbol r) \frac{\partial G(\boldsymbol\omega, \boldsymbol r)(\hat{v})}{\partial \omega_j} \right\|_{\sigma',s_1,s_2,P} \\
        \lesssim& |\boldsymbol r|_{\infty}^{2}.
    \end{align*}
    In addition, differentiating the range equation $L(\boldsymbol\omega,\boldsymbol r)\hat{v} + G(\boldsymbol \omega, \boldsymbol r)(\hat{v}) = 0$ twice yields
    \begin{equation*}
        \frac{\partial^2 \hat{v}}{\partial \omega_i\partial \omega_j} = -\big[ L(\boldsymbol\omega, \boldsymbol r) + DG(\boldsymbol\omega, \boldsymbol r)(\hat{v}) \big]^{-1}\bigg( \frac{\partial L(\boldsymbol \omega, \boldsymbol r)}{\partial \omega_j} \frac{\partial \hat{v}}{\partial \omega_i} + \frac{\partial L(\boldsymbol \omega, \boldsymbol r)}{\partial \omega_i} \frac{\partial \hat{v}}{\partial \omega_j} + D^2G(\boldsymbol\omega, \boldsymbol r)(\hat{v})\left[\frac{\partial \hat{v}}{\partial \omega_i},\frac{\partial \hat{v}}{\partial \omega_j}\right] \bigg).
    \end{equation*}
    Similarly, for $(\boldsymbol\omega,\boldsymbol r)\in\tilde\Omega_{\rho_3,C}$ and $\hat{h}_1,\hat{h}_2 \in PH^{\sigma',s_1,s_2}$, the definition of the second Fr\'echet derivative and the definition of $G$ implies that
    \begin{align*}
        &\| L^{-1}(\boldsymbol \omega, \boldsymbol r) D^2G(\boldsymbol \omega, \boldsymbol r)(\hat{v})[\hat{h}_1,\hat{h}_2] \|_{\sigma',s_1,s_2,P} \\
        \leq &2C_2 \| D^2G_1(\boldsymbol r)(\hat{v})[\hat{h}_1,\hat{h}_2] \|_{\sigma',s_1-2,s_2,P} + 2C_2 \| D^2G_2(\boldsymbol \omega, \boldsymbol r)(\hat{v})[\hat{h}_1,\hat{h}_2] \|_{\sigma',s_1,s_2-1,P}\\
        \leq &2C_2\cdot \|\widehat{h_1h_2} + \widehat{ D^2P_{2,1}(u_0+v)[h_1,h_2]} \|_{\sigma',s_1-2,s_2,P} + 2C_2\cdot \|\widehat{ D^2P_{2,2}(u_0+v)[h_1,h_2]} \|_{\sigma',s_1,s_2-1,P}\\
        \leq &2C_2\cdot [1 +f_{2}''(C_{\sigma', s_1,s_2}(\|\hat{u}_0\|_{\sigma',s_1,s_2}+\|\hat{v}\|_{\sigma',s_1,s_2,P}))]\cdot C_{\sigma', s_1,s_2}\|\hat{h}_1\|_{\sigma',s_1,s_2,P}\|\hat{h}_2\|_{\sigma',s_1,s_2,P}
    \end{align*}
    Notice that $f_2''(z)\lesssim z^{p-2}$ for $p\geq 3$. Taking $\hat{h}_1 = \frac{\partial \hat{v}}{\partial \omega_i},\ \hat{h}_2 = \frac{\partial \hat{v}}{\partial \omega_j}$, we obtain
    \[
        \| L^{-1}(\boldsymbol \omega, \boldsymbol r) D^2G(\boldsymbol \omega, \boldsymbol r)(\hat{v})[h_1,h_2] \|_{\sigma',s_1,s_2,P} \lesssim |\boldsymbol r|_{\infty}^4.
    \]
    Thus, we have
    \begin{align*}
        \left\|\frac{\partial^2 \hat{v}}{\partial \omega_i\partial \omega_j}\right\|_{\sigma',s_1,s_2,P} &\leq \frac{1}{1-\mathfrak q} \bigg( 4C_2C_{\mathcal{M}} \max_{j=1,2} \left\| \frac{\partial \hat{v}}{\partial \omega_j}\right\|_{\sigma',s_1,s_2,P} \\
        &+ \left\| L^{-1} D^2G(\boldsymbol \omega, \boldsymbol r)(\hat{v})\left[\frac{\partial \hat{v}}{\partial \omega_i},\frac{\partial \hat{v}}{\partial \omega_j}\right]\right\|_{\sigma',s_1,s_2,P} \bigg)\\
        &\lesssim |\boldsymbol r|_{\infty}^2.
    \end{align*}
    This proves \eqref{qiudaoguji2} and completes the proof.
\end{proof}

\subsection{Solving the bifurcation equation}\label{sec:two-gap-2}
As in the one-gap case, we solve the bifurcation equation and calculate the perturbed frequency by a perturbative argument.
Recall the bifurcation equation \eqref{qfangcheng2} is equivalent to
\begin{equation}\label{solvebifurcation2}
    \left(\frac{k_1\omega_1+k_2\omega_2}{m(\boldsymbol k)}-\frac{4\pi^2}{\gamma^2}m(\boldsymbol k)^2-\frac{n(\boldsymbol k)^2}{m(\boldsymbol k)^2}\right)\hat{u}(\boldsymbol k)
    + \frac{1}{2} (\hat{u}*\hat{u})(\boldsymbol k)
    + \widehat{P_2(u)}(\boldsymbol k)=0,
\end{equation}
for $\boldsymbol k=(k_1,k_2)\in S_2$.

For notational simplicity, we denote the Fourier coefficients $\hat{u}_0(\boldsymbol k;\boldsymbol r)$ by $\hat{u}_0(\boldsymbol k)$, and set $\mathcal{A}_1=\hat{u}_0(1,0)\hat{u}_0(-1,0), \mathcal{A}_2=\hat{u}_0(0,1)\hat{u}_0(0,-1)$.
For $\rho \in (0,1)$, denote $\tilde{\mathcal{U}}_{\rho} = \mathcal{U}_{\rho} \cap \{\boldsymbol r=(r_1,r_2): \rho/2< |r_j|< \rho,j=1,2\}$.
\begin{theorem}\label{solvetwogapbifurcation}
    After reducing $\rho_3$ if necessary, there exists a real-analytic function $\boldsymbol\omega =(\omega_1, \omega_2) : (\rho_3/2,\rho_3)^2 \to \mathbb{R}^2$ such that $\hat{u} = \hat{u}_0 + \widehat{v(\boldsymbol \omega(\boldsymbol r),\boldsymbol r)}$ satisfies the bifurcation equation \eqref{solvebifurcation2} on $(\rho_3/2,\rho_3)^2$. Moreover, when $p$ is odd, we have
    \begin{align*}
        \omega_1(\boldsymbol r) =\omega_{1,0}(\boldsymbol r) - m_1\cdot \mathscr{M}_{1,odd}(\boldsymbol r) +O(\boldsymbol r^{p}),\quad
        \omega_2(\boldsymbol r) =\omega_{2,0}(\boldsymbol r) - m_2\cdot \mathscr{M}_{2,odd}(\boldsymbol r) +O(\boldsymbol r^{p}),
    \end{align*}
    when $p$ is even, we have
    \begin{align*}
        \omega_1(\boldsymbol r) &=\omega_{1,0}(\boldsymbol r) - 
         \Gamma_{1,1}(\boldsymbol r) - \Gamma_{1,2}(\boldsymbol r) - m_1 \mathscr{M}_{1,even}(\boldsymbol r) + O( \boldsymbol r^{p+1}),\\
        \omega_2(\boldsymbol r) &=\omega_{2,0}(\boldsymbol r) - \Gamma_{2,1}(\boldsymbol r) -\Gamma_{2,2}(\boldsymbol r) - m_2 \mathscr{M}_{2,even}(\boldsymbol r)+ O(\boldsymbol r^{p+1}),
    \end{align*}
    where these notations are presented in Appendix \ref{Appendix}.
    
    Because of the expansion of $\boldsymbol\omega$, for sufficiently large $C$, $(\boldsymbol \omega(\boldsymbol r),\boldsymbol r) \in \Omega_{\rho_3,C}$.
\end{theorem}
\begin{proof}
    Fix a constant $C>0$, which will be chosen sufficiently large later. Denote
    \begin{equation*}
        I_{C}(\boldsymbol r) := \{(\omega_1,\omega_2) \in \mathbb{C}^2: |\omega_1 - \omega_{1,0}(\boldsymbol r)| < C|\boldsymbol r|_{\infty}^{p-1}, |\omega_2 - \omega_{2,0}(\boldsymbol r)| < C|\boldsymbol r|_{\infty}^{p-1} \},\quad \boldsymbol r\in \tilde{\mathcal{U}}_{\rho_3}.
    \end{equation*}
    Throughout the proof, we seek a solution satisfying the a priori restriction 
    \begin{equation}
       \boldsymbol \omega = (\omega_1,\omega_2) \in I_{C}(\boldsymbol r),\textup{ for } \boldsymbol r\in \tilde{\mathcal{U}}_{\rho_3}.\label{xianyanguji2}
    \end{equation}
    If $(\boldsymbol \omega,\boldsymbol r) \in \Omega_{\rho_3,C}$, Theorem \ref{Banachfixed2} implies that $\widehat{u_0(\boldsymbol r)} + \widehat{v(\boldsymbol\omega,\boldsymbol r)} \in \ell_{re}^2(\mathbb{T}_{\gamma}^2)$. Hence the bifurcation equation \eqref{solvebifurcation2} for $\boldsymbol k=(1,0)$ and $\boldsymbol k=(-1,0)$ is the same equation; the bifurcation equation \eqref{solvebifurcation2} for $\boldsymbol k=(0,1)$ and $\boldsymbol k=(0,-1)$ is the same equation.
    
    Since $\widehat{u_0(\boldsymbol r)}$ satisfies \eqref{kp-fourier2} for $\boldsymbol r\in \tilde{\mathcal{U}}_{\rho_3}$, with $\hat{u} = \widehat{u_0(\boldsymbol r)} + \widehat{v(\boldsymbol \omega,\boldsymbol r)}$, the bifurcation equations \eqref{qfangcheng2} for $\boldsymbol k = (1,0),(0,1)$ become
    \begin{align}
        (\omega_1-\omega_{1,0})\hat{u}_0(1,0)+\frac{ m_1 }{2}\widehat{[2u_0 v(\boldsymbol\omega,\boldsymbol r)+v(\boldsymbol\omega,\boldsymbol r)^2]}(1,0)+m_1\widehat{P_2(u_0+v(\boldsymbol\omega,\boldsymbol r))}(1,0)=0,\label{solvebifurcationcomplex2.1}\\
        (\omega_2-\omega_{2,0})\hat{u}_0(0,1)+\frac{ m_2 }{2}\widehat{[2u_0 v(\boldsymbol\omega,\boldsymbol r)+v(\boldsymbol\omega,\boldsymbol r)^2]}(0,1)+m_2\widehat{P_2(u_0+v(\boldsymbol\omega,\boldsymbol r))}(0,1)=0.\label{solvebifurcationcomplex2.2}
    \end{align}
    Denote the left-hand sides of \eqref{solvebifurcationcomplex2.1} and \eqref{solvebifurcationcomplex2.2} by $\tilde{F}_{1}(\boldsymbol\omega, \boldsymbol r)$ and $\tilde{F}_{2}(\boldsymbol\omega, \boldsymbol r)$, respectively.
    Then, $\tilde{F}_{1}(\boldsymbol\omega, \boldsymbol r)$ and $\tilde{F}_{2}(\boldsymbol\omega, \boldsymbol r)$ are holomorphic functions on $\tilde{\Omega}_{\rho_3,C}$, whose restrictions to $\Omega_{\rho_3,C}$ are real analytic. 
    Denote $ \tilde{\boldsymbol F} = (\tilde{F}_{1}, \tilde{F}_{2})$. As a consequence, to obtain a real-analytic solution on $(\rho_3/2,\rho_3)^2$ to the bifurcation equation \eqref{solvebifurcation2}, we only need to find a holomorphic function $\boldsymbol\omega(\boldsymbol r): \tilde{\mathcal{U}}_{\rho_3} \to \mathbb{C}^2$ such that $\tilde{\boldsymbol F}(\boldsymbol\omega(\boldsymbol r), \boldsymbol r) = 0$ and $\boldsymbol\omega(\overline{\boldsymbol r}) = \overline{\boldsymbol\omega(\boldsymbol r)}$.
    
    Now we analyze the bifurcation equation.
    
    The leading-order term of $\widehat{P_2(u_0+v(\boldsymbol\omega,\boldsymbol r))}(1,0),\widehat{P_2(u_0+v(\boldsymbol\omega,\boldsymbol r))}(0,1)$ depends on whether $p$ is odd or even. Denote $\mathcal{A}_1=\hat{u}_0(1,0)\hat{u}_0(-1,0), \mathcal{A}_2=\hat{u}_0(0,1)\hat{u}_0(0,-1)$.
    To start with, with Theorem \ref{Banachfixed2}, \eqref{Lipchitz2.1line1} and \eqref{Lipchitz2.1line2}, we have
    \begin{align*}
        \|\widehat{ P_{2,1}(u_0) }- \widehat{ P_{2,1}(u_0 + v(\boldsymbol\omega,\boldsymbol r) })\|_{\sigma',s_1-2,s_2}
        \leq f_{2,1}'(C_{\sigma', s_1,s_2}(\| \hat{ u }_0\|_{\sigma',s_1,s_2}+ M_0|\boldsymbol r|_{\infty}^{p} ))\cdot M_0|\boldsymbol r|_{\infty}^{p},
    \end{align*}
    and
    \begin{equation*}
        \|\widehat{ P_{2,2}(u_0) }- \widehat{ P_{2,2}(u_0 + v(\boldsymbol\omega,\boldsymbol r) })\|_{\sigma',s_1,s_2-1}\leq  f_{2,2}'(C_{\sigma', s_1, s_2}(\| \hat{ u }_0\|_{\sigma',s_1,s_2}+ M_0|\boldsymbol r|_{\infty}^{p} ))\cdot M_0|\boldsymbol r|_{\infty}^{p}.
    \end{equation*}
    Moreover, with \eqref{nonlinearestimate2.1} and \eqref{nonlinearestimate2.2}, we have
    \begin{equation*}
        \|\widehat{ P_{2,1,>p+1}(u_0 + v(\boldsymbol\omega,\boldsymbol r) })\|_{\sigma',s_1-2,s_2}\leq f_{2,1,>p+1}(C_{\sigma', s_1,s_2} (\|\hat{u}_0\|_{\sigma',s_1,s_2} + M_0|\boldsymbol r|_{\infty}^{p})),
    \end{equation*}
    and
    \begin{equation*}
        \|\widehat{ P_{2,2,>p+1}(u_0 + v(\boldsymbol\omega,\boldsymbol r) })\|_{\sigma',s_1,s_2-1}\leq f_{2,2,>p+1}(C_{\sigma', s_1,s_2} (\|\hat{u}_0\|_{\sigma',s_1,s_2} + M_0|\boldsymbol r|_{\infty}^{p})).
    \end{equation*}
    With these estimates and the calculations in the Appendix \ref{Appendix}, if $p$ is odd,
    \[
        \widehat{P_2(u_0+v)}(1,0) = \mathscr{M}_{1,odd}(\boldsymbol r) \cdot \hat{u}_0(1,0) + O( \boldsymbol r^{p+1}),\\
        \widehat{P_2(u_0+v)}(0,1) = \mathscr{M}_{2,odd}(\boldsymbol r) \cdot \hat{u}_0(0,1) + O( \boldsymbol r^{p+1}),
    \]
    whose leading orders are $p$.
    When $p$ is even, they can be rewritten into
    \begin{align*}
        &\widehat{P_2(u_0+v)}(1,0) =  \mathscr{M}_{1,even}(\boldsymbol r)\cdot \hat{u}_0(1,0) + O(\boldsymbol r^{p+2}),\\
        &\widehat{P_2(u_0+v)}(0,1) = \mathscr{M}_{2,even}(\boldsymbol r)\cdot \hat{u}_0(0,1)+ O(\boldsymbol r^{p+2}),
    \end{align*}
    whose leading orders are $p+1$.

    When $p$ is odd, a simple calculation implies that
    \begin{align*}
        &|\widehat{u_0v}(1,0)|
        \leq \sum_{(k,l)\in\mathbb{Z}^2\setminus S_0} \left|\hat{u}_0(k,l)\hat{v}(1-k,-l)\right| \lesssim \sum_{|\boldsymbol k|_1\geq 1} |\boldsymbol r|_{\infty}^{|\boldsymbol k|_1+p} \lesssim |\boldsymbol r|_{\infty}^{p+1},\\
        &|\widehat{u_0v}(0,1)|
        \leq \sum_{(k,l)\in\mathbb{Z}^2\setminus S_0} \left| \hat{u}_0(k,l)\hat{v}(-k,1-l)\right| \lesssim \sum_{|\boldsymbol k|_1\geq 1} |\boldsymbol r|_{\infty}^{|\boldsymbol k|_1+p} \lesssim |\boldsymbol r|_{\infty}^{p+1},\\
        &|\widehat{v^2}(1,0)|\leq \sum_{(k,l)\in\mathbb{Z}^2\setminus S_0} \left| \hat{v}(k,l)\hat{v}(1-k,-l) \right|\lesssim \sum_{|\boldsymbol k|_1\geq 1} e^{-\sigma\nu_2|\boldsymbol k|_1}|\boldsymbol r|_{\infty}^{2p} \lesssim |\boldsymbol r|_{\infty}^{2p},\\
        &|\widehat{v^2}(0,1)|\leq \sum_{(k,l)\in\mathbb{Z}^2\setminus S_0} \left| \hat{v}(k,l)\hat{v}(-k,1-l) \right| \lesssim \sum_{|\boldsymbol k|_1\geq 1} e^{-\sigma\nu_2|\boldsymbol k|_1}|\boldsymbol r|_{\infty}^{2p} \lesssim |\boldsymbol r|_{\infty}^{2p}.
    \end{align*}
    Notice that
    \[
        |\hat{u}_0(1,0)| \gtrsim |\boldsymbol r|_{\infty}, \quad |\hat{u}_0(0,1)| \gtrsim |\boldsymbol r|_{\infty},\quad \textup{for } \boldsymbol r \in \tilde{\mathcal{U}}_{\rho_3}.
    \]
    Thus, when $p$ is odd, we can obtain the bifurcation equation in the following form:
    \begin{align}
        (\omega_1-\omega_{1,0}(\boldsymbol r)) &+ m_1\cdot \mathscr{M}_{1,odd}(\boldsymbol r) + \mathscr{R}_{1,odd}(\boldsymbol \omega, \boldsymbol r) = 0,\label{reducedbifurcation3.1}\\
        (\omega_2-\omega_{2,0}(\boldsymbol r)) &+ m_2\cdot \mathscr{M}_{2,odd}(\boldsymbol r) + \mathscr{R}_{2,odd}(\boldsymbol \omega, \boldsymbol r) = 0,\label{reducedbifurcation3.2}
    \end{align}
    where $\mathscr{R}_{j,odd}(\boldsymbol \omega, \boldsymbol r) = O(\boldsymbol r^{p}),\ j=1,2$.
    When $p$ is even, we have
    \begin{align*}
        &|\widehat{u_0v}(1,0) - \hat{u}_0(-1,0) \hat{v}(2,0) - \hat{u}_0(0,1)\hat{v}(1,-1) - \hat{u}_0(0,-1)\hat{v}(1,1)| \\
        \leq &\sum_{(k,l)\in\mathbb{Z}^2\setminus (S_0\cup S_2)} \left|\hat{u}_0(k,l)\hat{v}(1-k,-l)\right| \lesssim \sum_{|\boldsymbol k|_1\geq 2} |\boldsymbol r|_{\infty}^{|\boldsymbol k|_1+p} \lesssim |\boldsymbol r|_{\infty}^{p+2},\\
        &|\widehat{u_0v}(0,1) - \hat{u}_0(0,-1) \hat{v}(0,2) - \hat{u}_0(-1,0)\hat{v}(1,1) - \hat{u}_0(1,0)\hat{v}(-1,1)| \\
        \leq &\sum_{(k,l)\in\mathbb{Z}^2\setminus (S_0\cup S_2)} \left| \hat{u}_0(k,l)\hat{v}(-k,1-l)\right| \lesssim \sum_{|\boldsymbol k|_1\geq 2} |\boldsymbol r|_{\infty}^{|\boldsymbol k|_1+p} \lesssim |\boldsymbol r|_{\infty}^{p+2},               
    \end{align*}
    and
    \begin{align*}
        &|\widehat{v^2}(1,0)|\leq \sum_{(k,l)\in\mathbb{Z}^2\setminus S_0} \left| \hat{v}(k,l)\hat{v}(1-k,-l) \right|\lesssim \sum_{|\boldsymbol k|_1\geq 1} e^{-\sigma\nu_2|\boldsymbol k|_1}|\boldsymbol r|_{\infty}^{2p} \lesssim |\boldsymbol r|_{\infty}^{2p},\\
        &|\widehat{v^2}(0,1)|\leq \sum_{(k,l)\in\mathbb{Z}^2\setminus S_0} \left| \hat{v}(k,l)\hat{v}(-k,1-l) \right| \lesssim \sum_{|\boldsymbol k|_1\geq 1} e^{-\sigma\nu_2|\boldsymbol k|_1}|\boldsymbol r|_{\infty}^{2p} \lesssim |\boldsymbol r|_{\infty}^{2p}.
    \end{align*}    
    We therefore need to calculate the leading-order term of $\hat{v}(2,0), \hat{v}(0,2),\hat{v}(1,1), \hat{v}(1,-1), \hat{v}(-1,1)$ when $p$ is even.
    
    Recall the Picard iteration sequence for the range equation \eqref{banachform2}:
    \begin{equation*}
        \hat{v}_0(\boldsymbol \omega, \boldsymbol r)=0,\hat{v}_{n+1}(\boldsymbol \omega, \boldsymbol r)= -L^{-1}(\boldsymbol \omega, \boldsymbol r)G(\boldsymbol \omega, \boldsymbol r)(\hat{v}_n(\boldsymbol \omega, \boldsymbol r)),n\geq 0.
    \end{equation*}
    The first iterate is
    \begin{equation*}
        \hat{v}_1(\boldsymbol\omega,\boldsymbol r) = -L^{-1}(\boldsymbol \omega, \boldsymbol r)G(\boldsymbol \omega, \boldsymbol r)(0) = \sum_{l=0}^{\infty} (-L_0^{-1}(\boldsymbol \omega)L_1(\boldsymbol r))^l(-L_0^{-1}(\boldsymbol \omega)G(\boldsymbol \omega, \boldsymbol r)(0)).
    \end{equation*}
    Notice that for $\boldsymbol k=(2,0), (0,2), (1,1), (1,-1), (-1,1)$, $\left|\boldsymbol k\cdot (\boldsymbol\omega-\boldsymbol \omega_{0})\hat{u}_0(\boldsymbol k)\right| \lesssim |\boldsymbol r|_{\infty}^{p+1}$.
    With the preceding estimate and \eqref{omega10line1}--\eqref{omega20line2}, we have
    \begin{align*}
        L_0(\boldsymbol \omega)(\boldsymbol k, \boldsymbol k) &= L_0(\boldsymbol \omega_0(\boldsymbol 0))(\boldsymbol k, \boldsymbol k)+O(\boldsymbol r^{2}),\\
        \left(G(\boldsymbol \omega, \boldsymbol r)(0)\right)(\boldsymbol k) &= (a_p+b_p) \widehat{u_0^{p}}(\boldsymbol k) + O(\boldsymbol r^{p+1}).
    \end{align*}
    Hence,
    \begin{align*}
        \left(-L_0^{-1}(\boldsymbol \omega) G(\boldsymbol \omega, \boldsymbol r)(0)\right)(\boldsymbol k)=& -(a_p+b_p)L_0(\boldsymbol \omega_0(\boldsymbol 0))(\boldsymbol k,\boldsymbol k)^{-1} \widehat{u_0^{p}}(\boldsymbol k) + O(\boldsymbol r^{p+1}),
    \end{align*}
    and for all $l>0$,
    \begin{align*}
        &\left| \left( (-L_0^{-1}(\boldsymbol \omega)L_1(\boldsymbol r))^l (-L_0^{-1}(\boldsymbol \omega) G(\boldsymbol \omega, \boldsymbol r)(0)) \right)(\boldsymbol k) \right|\\
        = &\left| \sum_{\boldsymbol k'\in \mathbb{Z}^2} (-L_0^{-1}(\boldsymbol \omega) L_1(\boldsymbol r))^l(\boldsymbol k,\boldsymbol k')\cdot -L_0^{-1}(\boldsymbol \omega)(\boldsymbol k',\boldsymbol k')\cdot\left(G(\boldsymbol \omega, \boldsymbol r)(0)\right)(\boldsymbol k') \right|\\
        \lesssim &|\boldsymbol r|_{\infty}^{p+l} \sum_{\boldsymbol k' \in \mathbb{Z}^2}|\boldsymbol r|_{\infty}^{2\max\{|\boldsymbol k-\boldsymbol k'|_1-l,0\}} \lesssim |\boldsymbol r|_{\infty}^{p+l}.
    \end{align*}  
    Thus
    \begin{align*}
        &\hat{v}_1(\boldsymbol k)=-(a_p+b_p)L_0(\boldsymbol \omega_0(\boldsymbol 0))(\boldsymbol k,\boldsymbol k)^{-1} \widehat{u_0^{p}}(\boldsymbol k) + O(\boldsymbol r^{p+1}).
    \end{align*}
    With Theorem \ref{Banachfixed2}, we have
    \begin{align*}
        \|\hat{v}_{n} - \hat{v}_1\|_{\sigma',s_1,s_2,P}& \leq \sum_{j=1}^{n-1} |\boldsymbol r|_{\infty}^{(p-2)j}  \|\hat{v}_1 - \hat{v}_{0}\|_{\sigma',s_1,s_2,P} \leq \frac{M_0|\boldsymbol r|_{\infty}^{2p-2}}{1-|\boldsymbol r|_{\infty}^{p-2}},
    \end{align*}
    which implies that
    \begin{equation*}
        \| \widehat{v(\boldsymbol \omega,\boldsymbol r)}-\widehat{v_1(\boldsymbol \omega,\boldsymbol r)} \|_{\sigma',s_1,s_2,P} \lesssim |\boldsymbol r|_{\infty}^{2p-2}.
    \end{equation*}
    Thus,
    \begin{equation}\label{approximationv}
        \widehat{v(\boldsymbol \omega,\boldsymbol r)}(\boldsymbol k)=-(a_p+b_p)L_0(\boldsymbol \omega_0(\boldsymbol 0))(\boldsymbol k,\boldsymbol k)^{-1} \widehat{u_0^{p}}(\boldsymbol k) + O(\boldsymbol r^{p+1}).
    \end{equation}
    Notice that if $m_1 + m_2 = 0$ or $m_1 - m_2 = 0$, the leading-order term in \eqref{approximationv} at $(1,-1),(-1,1)$ or $(1,1)$ vanishes.
    When $p$ is even, a simple calculation yields
    \begin{align*}
        L_0(\boldsymbol \omega_0(\boldsymbol 0))((2,0),(2,0))^{-1} &=  -\frac{\gamma^2}{12\pi^2m_1^2} ,\quad
        L_0(\boldsymbol \omega_0(\boldsymbol 0))((0,2),(0,2))^{-1} =  -\frac{\gamma^2}{12\pi^2m_2^2},\\
        L_0(\boldsymbol \omega_0(\boldsymbol 0))((1,1),(1,1))^{-1} &= \left(-\frac{12\pi^2}{\gamma^2} m_1m_2 + \frac{(m_1n_2-m_2n_1)^2}{m_1m_2(m_1+m_2)^2}\right)^{-1},\textup{when }m_1+m_2 \neq 0,\\
        L_0(\boldsymbol \omega_0(\boldsymbol 0))((1,-1),(1,-1))^{-1} &= L_0(\boldsymbol \omega_0(\boldsymbol 0))((-1,1),(-1,1))^{-1}\\
        &= \left(\frac{12\pi^2}{\gamma^2} m_1m_2 - \frac{(m_1n_2-m_2n_1)^2}{m_1m_2(m_1-m_2)^2}\right)^{-1},\textup{when }m_1 - m_2 \neq 0.
    \end{align*}
    Moreover,
    \begin{align*}
        \widehat{u_0^{p}}(2,0) = \sum_{\substack{ \boldsymbol k_1+\cdots+\boldsymbol k_p = (2,0) \\ |\boldsymbol k_1|_1+\cdots +|\boldsymbol k_p|_1 = p}} &\hat{u}_0(\boldsymbol k_1)\cdots \hat{u}_0(\boldsymbol k_p) +  \sum_{L\geq p+2}\sum_{\substack{ \boldsymbol k_1+\cdots+\boldsymbol k_p = (2,0) \\ |\boldsymbol k_1|_1+\cdots +|\boldsymbol k_p|_1 = L}} \hat{u}_0(\boldsymbol k_1)\cdots \hat{u}_0(\boldsymbol k_p), \\
        \sum_{L\geq p+2}\sum_{\substack{ \boldsymbol k_1+\cdots+\boldsymbol k_p = (2,0) \\ |\boldsymbol k_1|_1+\cdots +|\boldsymbol k_p|_1 = L}} &|\hat{u}_0(\boldsymbol k_1)\cdots \hat{u}_0(\boldsymbol k_p)| \leq \sum_{L\geq p+2} L^{2p}|\boldsymbol r|_{\infty}^{L} \lesssim |\boldsymbol r|_{\infty}^{p+2},\\
        \sum_{\substack{ \boldsymbol k_1+\cdots+\boldsymbol k_p = (2,0) \\ |\boldsymbol k_1|_1+\cdots +|\boldsymbol k_p|_1 = p}} \hat{u}_0(\boldsymbol k_1)\cdots \hat{u}_0(\boldsymbol k_p) &= \sum_{l=0}^{p/2-1} \binom{p}{l}\binom{p-l}{l}\binom{p-2l}{p/2-l-1} \mathcal{A}_2^{l} \mathcal{A}_1^{p/2-l-1} \cdot \hat{u}_0(1,0)^{2}\\
        &+ O(\boldsymbol r^{p+2}).
    \end{align*}
    Similarly, we have
    \begin{align*}
        \widehat{u_0^{p}}(0,2) =& \sum_{l=0}^{p/2-1} \binom{p}{l} \binom{p-l}{l} \binom{p-2l}{p/2-l-1} \mathcal{A}_1^{l} \mathcal{A}_2^{p/2-l-1} \hat{u}_0(0,1)^{2} + O(\boldsymbol r^{p+2}),\\
        \widehat{u_0^{p}}(1,1) =& \sum_{l=0}^{p/2-1} \binom{p}{l+1} \binom{p-l-1}{l} \binom{p-2l-1}{p/2-l} \mathcal{A}_1^{l} \mathcal{A}_2^{p/2-l-1} \hat{u}_0(1,0) \hat{u}_0(0,1) + O(\boldsymbol r^{p+2}),\\
        \widehat{u_0^{p}}(1,-1) =& \sum_{l=0}^{p/2-1} \binom{p}{l+1} \binom{p-l-1}{l} \binom{p-2l-1}{p/2-l}\mathcal{A}_1^{l} \mathcal{A}_2^{p/2-l-1} \hat{u}_0(1,0) \hat{u}_0(0,-1) + O(\boldsymbol r^{p+2}),\\
        \widehat{u_0^{p}}(-1,1) =& \sum_{l=0}^{p/2-1} \binom{p}{l+1} \binom{p-l-1}{l} \binom{p-2l-1}{p/2-l} \mathcal{A}_1^{l} \mathcal{A}_2^{p/2-l-1} \hat{u}_0(-1,0) \hat{u}_0(0,1) + O(\boldsymbol r^{p+2}).
    \end{align*}
    With the above calculations, we can obtain the bifurcation equation when $p$ is even as follows:
    \begin{align}
            &(\omega_1-\omega_{1,0}(\boldsymbol r)) + 
        \Gamma_{1,1}(\boldsymbol r) + \Gamma_{1,2}(\boldsymbol r) + m_1 \mathscr{M}_{1,even}(\boldsymbol r) + \mathscr{R}_{1,even}(\boldsymbol\omega, \boldsymbol r)=0,\label{reducedbifurcation4.1}\\
        &(\omega_2-\omega_{2,0}(\boldsymbol r)) + \Gamma_{2,1}(\boldsymbol r) + \Gamma_{2,2}(\boldsymbol r) + m_2 \mathscr{M}_{2,even}(\boldsymbol r) + \mathscr{R}_{2,even}(\boldsymbol\omega, \boldsymbol r) =0.\label{reducedbifurcation4.2}
    \end{align}
    where $ \mathscr{R}_{j,even}(\boldsymbol\omega, \boldsymbol r) = O( \boldsymbol r^{p+1}),\ j=1,2$.
    
    Denote the left-hand sides of \eqref{reducedbifurcation3.1}--\eqref{reducedbifurcation3.2} and \eqref{reducedbifurcation4.1}--\eqref{reducedbifurcation4.2} by
    \[
    F_{odd}(\boldsymbol\omega, \boldsymbol r) = (F_{1,odd}(\boldsymbol\omega, \boldsymbol r),F_{2,odd}(\boldsymbol\omega, \boldsymbol r))^{T},\quad F_{even}(\boldsymbol \omega, \boldsymbol r) = (F_{1,even}(\boldsymbol \omega, \boldsymbol r), F_{2,even}(\boldsymbol \omega, \boldsymbol r))^{T},
    \]
    respectively.
    Let
    \begin{align}
        &\omega_{1,0,odd}(\boldsymbol r) = \omega_{1,0}(\boldsymbol r) - m_1\cdot \mathscr{M}_{1,odd}(\boldsymbol r),\label{zhubuodd1.1}\\
        &\omega_{2,0,odd}(\boldsymbol r) = \omega_{2,0}(\boldsymbol r) - m_2\cdot \mathscr{M}_{2,odd}(\boldsymbol r),\label{zhubuodd1.2}
    \end{align}
    and
    \begin{align}
        &\omega_{1,0,even}(\boldsymbol r) = \omega_{1,0}(\boldsymbol r) - 
        \Gamma_{1,1}(\boldsymbol r) - \Gamma_{1,2}(\boldsymbol r)
        - m_1\cdot \mathscr{M}_{1,even}(\boldsymbol r),\label{zhubueven2.1}\\
        &\omega_{2,0,even}(\boldsymbol r) = \omega_{2,0}(\boldsymbol r) - 
         \Gamma_{2,1}(\boldsymbol r) - \Gamma_{2,2}(\boldsymbol r) - m_2\cdot \mathscr{M}_{2,even}(\boldsymbol r).\label{zhubueven2.2}
    \end{align}
    Denote $\boldsymbol\omega_{0,odd}(\boldsymbol r) = (\omega_{1,0,odd}(\boldsymbol r), \omega_{2,0,odd}(\boldsymbol r)), \boldsymbol\omega_{0,even}(\boldsymbol r) = (\omega_{1,0,even}(\boldsymbol r), \omega_{2,0,even}(\boldsymbol r))$.
    Then,
    \begin{align*}
        F_{k,odd}(\boldsymbol\omega_{0,odd}(\boldsymbol r), \boldsymbol r) = O(\boldsymbol r^{p}),\quad
        F_{k,even}(\boldsymbol\omega_{0,even}(\boldsymbol r), \boldsymbol r) = O(\boldsymbol r^{p+1}),\quad k=1,2.
    \end{align*}
    We may suppose
    \[
        |F_{k,odd}(\boldsymbol\omega_{0,odd}(\boldsymbol r), \boldsymbol r)|\leq \mathscr E_0 |\boldsymbol r|_{\infty}^p,\quad |F_{k,even}(\boldsymbol\omega_{0,even}(\boldsymbol r), \boldsymbol r)|\leq \mathscr E_0 |\boldsymbol r|_{\infty}^{p+1},
    \]
    for a constant $\mathscr E_0>0$ independent of $\rho_3$.
    With \eqref{Lipchitz2.1line1}, \eqref{Lipchitz2.1line2}, \eqref{Lipchitz2.2} and \eqref{qiudaoguji2}, we have
    \begin{align}
        \dfrac{\partial F_{k,\cdot}}{\partial \omega_j}(\boldsymbol\omega,\boldsymbol r) &= \delta_{k,j} + \frac{m_k}{\hat{u}_0(e_k)} \left[ \left(\hat{u}_0*\frac{\partial \hat{v}}{\partial \omega_j}\right)(e_k) + \left(\hat{v}*\frac{\partial \hat{v}}{\partial \omega_j}\right)(e_k) \right] \notag\\
        &\qquad+ \frac{m_k}{\hat{u}_0(e_k)} \widehat{DP_2(u_0+v)[\partial_{\omega_j}v]}(e_k) \notag\\
        &= \delta_{k,j} + O(\boldsymbol r^{2})+O(\boldsymbol r^{p-1}),\label{setMinv}
    \end{align}
    and
    \begin{align*}
        \left| \dfrac{\partial^2 F_{k,\cdot}}{\partial \omega_i \partial \omega_j}(\boldsymbol\omega,\boldsymbol r) \right| &= \bigg| \frac{m_k}{\hat{u}_0(e_k)} \left[ \left( \hat{u}_0*\frac{\partial^2 \hat{v}}{\partial \omega_i\partial \omega_j}\right)(e_k) + \left( \frac{\partial \hat{v}}{\partial \omega_i} * \frac{\partial \hat{v}}{\partial \omega_j}\right)(e_k) + \left( \hat{v}*\frac{\partial^2 \hat{v}}{\partial \omega_i\partial \omega_j}\right)(e_k) \right] \\
        &\qquad+ \frac{m_k}{\hat{u}_0(e_k)} \left[ \widehat{D^2P_2(u_0+v)[\partial_{\omega_i}v, \partial_{\omega_j}v]}(e_k) + \widehat{DP_2(u_0+v)[\partial_{\omega_i\omega_j}v]}(e_k) \right] \bigg| \\
        &\lesssim |\boldsymbol r|^2,
    \end{align*}
    for $\cdot = \textup{odd, even}$, $e_1 = (1,0), e_2 = (0,1)$, and $i,j,k\in \{1,2\}$.
    Thus, the Jacobian matrix of $F_{odd}$ admits a uniform inverse on $\overline{\tilde{\Omega}}_{\rho_3,C}$.
    With the uniform estimate \eqref{setMinv}, we set
    \[
        M_{inv}:= \sup_{(\boldsymbol\omega, \boldsymbol r)\in\overline{\tilde \Omega}_{\rho_3,C}}\frac{\max_{j,k\in \{1,2\}}\left| \partial_{\omega_j} F_{k,odd} (\boldsymbol\omega, \boldsymbol r) \right|}{|\det [\partial_{\boldsymbol\omega} F_{odd} (\boldsymbol\omega, \boldsymbol r)]|}< \infty.
    \]
    Moreover, by the preceding estimate for the second derivative, there
    exists a constant $M>0$ such that
    \[
        \left| \partial_{\omega_i \omega_j}F_{k,odd}(\boldsymbol \omega, \boldsymbol r) \right| \leq M|\boldsymbol r|_{\infty}^2, \quad (\boldsymbol \omega, \boldsymbol r)\in\overline{\tilde\Omega}_{\rho_3,C},\quad i,j,k \in \{1,2\}.
    \]    
    Notice that with \eqref{zhubuodd1.1} and \eqref{zhubuodd1.2}, there exists a constant $C_{\mathrm{odd}}>0$ such that
    \[
        \left| \boldsymbol\omega_{0,\mathrm{odd}}(\boldsymbol r) - \boldsymbol\omega_{0}(\boldsymbol r) \right|_{\infty} \leq C_{\mathrm{odd}}|\boldsymbol r|_{\infty}^{p-1}.
    \]
    Choose $C>2C_{\mathrm{odd}}$. After reducing $\rho_3$ if necessary, we may
    assume that
    \[
        M_{inv}\mathscr E_0 \sum_{l\geq1} |\boldsymbol r|_{\infty}^{2^{l-1}p} \leq \left( \frac C2-C_{\mathrm{odd}} \right) |\boldsymbol r|_{\infty}^{p-1}.
    \]
    For $l\geq 1$, let
    \begin{align*}
        E_{l-1}(\boldsymbol r) &:= |F_{odd}(\boldsymbol\omega_{l-1,odd}(\boldsymbol r),\boldsymbol r)|_{\infty},\\
        \Delta_{l,odd}(\boldsymbol{r}) &:= -\partial _{\boldsymbol \omega} F_{odd}(\boldsymbol \omega_{l-1,odd}(\boldsymbol{r}),\boldsymbol{r})^{-1} F_{odd}(\boldsymbol \omega_{l-1,odd}(\boldsymbol{r}),\boldsymbol{r}),\\
        \boldsymbol\omega_{l,odd}(\boldsymbol r) &:= \boldsymbol\omega_{0,odd}(\boldsymbol r)+\sum_{l'=1}^{l}\Delta_{l',odd}(\boldsymbol r).
    \end{align*}
    Fix $l=1$. Notice that $\boldsymbol\omega_{0,odd}(\boldsymbol r) \in I_C(\boldsymbol r)$. Since $\boldsymbol \omega_{0,odd}(\boldsymbol r),F_{odd}(\boldsymbol \omega, \boldsymbol r)$ are real analytic on $\Omega_{\rho_3,C}$ and holomorphic on $\tilde{\Omega}_{\rho_3,C}$, we have
    \begin{align*}
        &\boldsymbol\omega_{0,odd}(\overline {\boldsymbol r}) = \overline{\boldsymbol\omega_{0,odd}(\boldsymbol r)},\\
        &\Delta_{1,odd}(\overline {\boldsymbol r}) = -\partial _{\boldsymbol\omega} F_{odd}(\overline{\boldsymbol\omega_{0,odd}(\boldsymbol r)},\overline {\boldsymbol r})^{-1} F_{odd}(\overline{\boldsymbol\omega_{0,odd}(\boldsymbol r)},\overline{\boldsymbol r}) = \overline{\Delta_{1,odd}(\boldsymbol r)}.
    \end{align*}
    Hence $\Delta_{1,odd}(\boldsymbol{r})$ is holomorphic in $\tilde{\mathcal{U}}_{\rho_3}$ and $|\Delta_{1,odd}(\boldsymbol{r})|_{\infty} \leq M_{inv}E_0(\boldsymbol r)$. Moreover, since
    \[
        |\boldsymbol\omega_{0,odd}(\boldsymbol r)-\boldsymbol\omega_{0}(\boldsymbol r)|_{\infty} + |\Delta_{1,odd}(\boldsymbol r)|_{\infty} < C|\boldsymbol r|_{\infty}^{p-1},
    \]
    the function $\boldsymbol\omega_{0,odd}(\boldsymbol r) + t \Delta_{1,odd}(\boldsymbol r) \in I_C(\boldsymbol r)$ for $t\in [0,1]$. This implies that
    \begin{align*}
        &|F_{odd}(\boldsymbol \omega_{0,odd}(\boldsymbol{r})+\Delta_{1,odd}(\boldsymbol{r}),\boldsymbol{r})|_{\infty} \\
        \leq&|F_{odd}(\boldsymbol \omega_{0,odd}(\boldsymbol{r}),\boldsymbol{r})+\partial _{\boldsymbol \omega} F_{odd}(\boldsymbol \omega_{0,odd}(\boldsymbol{r}),\boldsymbol{r})\Delta_{1,odd}(\boldsymbol{r})|_{\infty}\\
        &+ \left|\Delta_{1,odd}(\boldsymbol{r})^{T}\left[ \int_0^1 \partial_{\boldsymbol \omega}^2F_{odd}(\boldsymbol \omega_{0,odd}(\boldsymbol{r})+t\Delta_{1,odd}(\boldsymbol{r}),\boldsymbol{r})(1-t) dt\right] \Delta_{1,odd}(\boldsymbol{r})\right|_{\infty} \\
        \leq &M |\boldsymbol{r}|_{\infty}^{2} \cdot |\Delta_{1,odd}(\boldsymbol{r})|_{\infty}^2 \leq MM_{inv}^2|\boldsymbol{r}|_{\infty}^{2}E_0(\boldsymbol r)^2.
    \end{align*}
    Suppose the preceding results hold for $l-1$. Then, $\Delta_{l,odd}(\boldsymbol{r})$ is holomorphic in $\tilde{\mathcal{U}}_{\rho_3}$ and $|\Delta_{l,odd}(\boldsymbol{r})|_{\infty} \leq M_{inv}E_{l-1}(\boldsymbol r)$. If $MM_{inv}^2\mathscr E_0 \rho_3^2<1$, we claim that $E_{l'}(\boldsymbol r) \leq \mathscr E_0 |\boldsymbol r|_{\infty}^{2^{l'}p}$ for $l'<l$. It is easy to see that $E_0(\boldsymbol r)$ satisfies the estimate. If it holds for $E_{l'-1}(\boldsymbol r)$, then
    \begin{align*}
        E_{l'}(\boldsymbol r) &\leq M|\boldsymbol r|_{\infty}^2|\Delta_{l'}(\boldsymbol r)|^2 \leq MM_{inv}^2|\boldsymbol r|_{\infty}^2E_{l'-1}(\boldsymbol r)^2 \\
        &\leq MM_{inv}^2\mathscr E_0 |\boldsymbol r|_{\infty}^2\cdot \mathscr E_0|\boldsymbol r|_{\infty}^{2^{l'}p}\leq \mathscr E_0 |\boldsymbol r|_{\infty}^{2^{l'}p}.
    \end{align*}
    Thus, for $t \in [0,1]$,
    \begin{align*}
        \left| \boldsymbol\omega_{l-1,\mathrm{odd}}(\boldsymbol r) + t\Delta_{l,odd}(\boldsymbol r) - \boldsymbol\omega_{0}(\boldsymbol r) \right|_{\infty} &\leq \left| \boldsymbol\omega_{0,\mathrm{odd}}(\boldsymbol r) - \boldsymbol\omega_{0}(\boldsymbol r) \right|_{\infty} + \sum_{l'=1}^{l}|\Delta_{l',odd}(\boldsymbol r)|_{\infty}\\
        &\leq C_{\mathrm{odd}}|\boldsymbol r|_{\infty}^{p-1} + M_{inv}\mathscr E_0\sum_{l\geq 1} |\boldsymbol r|_{\infty}^{2^{l-1}p} < C|\boldsymbol r|_{\infty}^{p-1},
    \end{align*}
    which implies that $\boldsymbol\omega_{l-1,\mathrm{odd}}(\boldsymbol r) + t\Delta_{l,odd}(\boldsymbol r) \in I_C(\boldsymbol r)$.
    As a consequence,
    \begin{align*}
        &|F_{odd}(\boldsymbol \omega_{l-1,odd}(\boldsymbol r)+\Delta_{l,odd}(\boldsymbol r),\boldsymbol r)|_{\infty} \\
        = &\left|\Delta_{l,odd}(\boldsymbol r)^T \left[ \int_0^1 \partial_{\boldsymbol \omega}^2F_{odd}(\boldsymbol \omega_{l-1,odd}(\boldsymbol r)+t\Delta_{l,odd}(\boldsymbol r),\boldsymbol r)(1-t) dt\right] \Delta_{l,odd}(\boldsymbol r)\right|_{\infty}\\
        \leq &MM_{inv}^2|\boldsymbol r|_{\infty}^2E_{l-1}(\boldsymbol r)^2 \leq MM_{inv}^2\mathscr E_0|\boldsymbol r|_{\infty}^2 \cdot \mathscr E_0|\boldsymbol r|_{\infty}^{2^lp} \leq \mathscr E_0|\boldsymbol r|_{\infty}^{2^lp}.
    \end{align*}
    It also follows inductively that
    \[
        \boldsymbol \omega_{l,odd}(\overline{\boldsymbol r}) = \overline{\boldsymbol \omega_{l,odd}(\boldsymbol r)}, \quad \Delta_{l,odd}(\overline{\boldsymbol r}) = \overline{\Delta_{l,odd}(\boldsymbol r)}.
    \]
    Thus, for $j=1,2$, the series
    \begin{align*}
        \sum_{l\geq 1} |\Delta_{l,odd}(\boldsymbol{r})|_{\infty} \leq M_{inv}\mathscr E_0 \sum_{l\geq 1} |\boldsymbol r|_{\infty}^{2^{l-1}p} \lesssim |\boldsymbol{r}|_{\infty}^{p} 
    \end{align*}
    converges absolutely and locally uniformly on $\tilde{\mathcal{U}}_{\rho_3}$.
    Hence we obtain a holomorphic function
    \begin{equation}
        \boldsymbol \omega_{odd}(\boldsymbol{r}) = \boldsymbol \omega_{0,odd}(\boldsymbol{r}) + \sum_{l=1}^{\infty}\Delta_{l,odd}(\boldsymbol{r}). \label{jieoddbeq2}
    \end{equation}
    Since $F_{odd}$ is continuous, $|F_{odd}(\boldsymbol\omega_{odd}(\boldsymbol r),\boldsymbol r)|_{\infty} = \lim_{l\to\infty} |F_{\mathrm{odd}}(\boldsymbol\omega_{l,odd}(\boldsymbol r),\boldsymbol r)|_{\infty} =0$. Thus, $\boldsymbol \omega_{odd}(\boldsymbol{r})$ is a holomorphic solution of the equation \eqref{reducedbifurcation3.1}--\eqref{reducedbifurcation3.2} on $\tilde{\mathcal{U}}_{\rho_3}$.
    Moreover, since $\boldsymbol\omega_{l,odd}(\overline{\boldsymbol r}) = \overline{\boldsymbol \omega_{l,odd}(\boldsymbol r)}$ hold for any $l\geq 1$, the solution satisfies $\boldsymbol \omega_{odd}(\overline{\boldsymbol r}) = \overline{\boldsymbol \omega_{odd}(\boldsymbol r)}$.
    Similarly, we can obtain a real-analytic solution
    \begin{equation}
        \boldsymbol \omega_{even}(\boldsymbol{r}) = \boldsymbol\omega_{0,even}(\boldsymbol{r}) + \sum_{l=1}^{\infty}\Delta_{l,even}(\boldsymbol{r}) \label{jieevenbeq2}
    \end{equation}
    of the equation \eqref{reducedbifurcation4.1}--\eqref{reducedbifurcation4.2} on $(\rho_3/2,\rho_3)^2$.
    
    Therefore, after first choosing $C$ sufficiently large and then reducing $\rho_3$, the a priori restriction \eqref{xianyanguji2} is satisfied. Thus, the functions \eqref{jieoddbeq2}--\eqref{jieevenbeq2} are real-analytic solutions to the bifurcation equation \eqref{solvebifurcation2} on $(\rho_3/2,\rho_3)^2$.
\end{proof}

\appendix
\section{Proof of Corollary \ref{cor:finite-gap-KP}}
\begin{proof}
Throughout this paper, we use the  normalization \eqref{kp} of the KP equation.
However, the finite-gap solutions which we invoke are usually stated in the literature for \eqref{transformedKP2} and \eqref{transformedKP1}.
These two forms are equivalent up to elementary changes of variables and a rescaling of the dependent variable. 
We now transform the Its--Matveev formula to the form used in
\eqref{kp}. 
Take KP-II equation for example.
Let \(u(T,X,Y)\) be a solution of
\[
    \partial_X\left(
    u_T-\frac14(6uu_X+u_{XXX})
    \right)
    -\frac34u_{YY}=0.
\]
Equivalently, after applying \(\partial_X^{-1}\), this equation can be
written as
\[
    u_T-\frac14u_{XXX}-\frac32uu_X
    -\frac34\partial_X^{-1}u_{YY}=0.
\]
Set
\[
    \tilde u(t,x,y)=A\,u(a_3t,a_1x,a_2y).
\]
Then
\[
    \tilde u_t=Aa_3u_T,\qquad
    \tilde u_{xxx}=Aa_1^3u_{XXX},\qquad
    \tilde u\tilde u_x=A^2a_1uu_X,
\]
and
\[
    \partial_x^{-1}\tilde u_{yy}
    =
    \frac{Aa_2^2}{a_1}\partial_X^{-1}u_{YY}.
\]
Therefore, in order that \(\tilde u\) solve
\[
    \tilde u_t+\tilde u_{xxx}
    +\partial_x^{-1}\tilde u_{yy}
    +\tilde u\tilde u_x=0,
\]
it is enough to require $a_3=-4a_1^3,\ A=6a_1^2,\ a_2^2=3 a_1^4$.
Taking \(a_1=1\), we obtain
\[
    A=6,\qquad a_3=-4,\qquad a_2^2=3.
\]

We choose $a_2=\sqrt3$.
Hence, the function $\tilde u(t,x,y)=6u(-4t,x,\sqrt3 y)$ solves \eqref{kp} with \(\lambda=-1\). Consequently, from the Its--Matveev formula, we obtain the solution of the equation \eqref{kp} with $\lambda=-1$:
\[
    \tilde u(t,x,y)
    =
    12\frac{\partial^2}{\partial x^2}
    \log\theta\left(
        \boldsymbol Ux+\sqrt3\boldsymbol Vy
        -4\boldsymbol Wt+\boldsymbol D;B
    \right)
    +12c.
\]

Using the Galilean symmetry of \eqref{kp}, namely that if \(v(t,x,y)\) is
a solution then
\[
    v(t,x+st,y)-s
\]
is also a solution, the additive constant can be removed. Taking
\(s=12c\), we obtain the solution of the equation \eqref{kp} with $\lambda=-1$:
\begin{equation*}
    \tilde u(t,x,y)
    =
    12\frac{\partial^2}{\partial x^2}
    \log\theta\left(
        \boldsymbol Ux+\sqrt3\boldsymbol Vy
        +\boldsymbol W't+\boldsymbol D;B
    \right),\quad \boldsymbol W'=-4\boldsymbol W+12c\,\boldsymbol U.
\end{equation*}

Similarly, for the KP-I equation, applying the complex transformation
from KP-II to KP-I in Remark \ref{transformKP12} gives the finite-gap solution
\begin{equation*}
    \tilde u(t,x,y)
    =
    12\frac{\partial^2}{\partial x^2}
    \log\theta\left(
        i\bigl(-\boldsymbol Ux+\sqrt3\boldsymbol Vy
        +\boldsymbol W't\bigr)+\boldsymbol D;B
    \right)
\end{equation*}
of the equation \eqref{kp} with $\lambda = 1$.
\end{proof}

\section{Expansion of \texorpdfstring{$\boldsymbol{\omega}$}{omega} in Theorem \ref{solvetwogapbifurcation}}\label{Appendix}

Set
\begin{align*}
    &\mathscr{M}_{1,odd}(\boldsymbol r) = (a_p+b_p)\sum_{l=0}^{(p-1)/2} \binom{p}{l+1} \binom{p-l-1}{l} \binom{p-2l-1}{(p-1)/2-l} \mathcal{A}_1^{l} \mathcal{A}_2^{(p-1)/2-l},\\
    &\mathscr{M}_{2,odd}(\boldsymbol r) = (a_p+b_p)\sum_{l=0}^{(p-1)/2} \binom{p}{l+1} \binom{p-l-1}{l} \binom{p-2l-1}{(p-1)/2-l} \mathcal{A}_2^{l} \mathcal{A}_1^{(p-1)/2-l},
\end{align*}
and
\begin{align*}
    &C_{\Gamma,1} :=
    \begin{cases}
        D_{1,+}+D_{1,-},\ m_1\pm m_2 \neq 0,\\
        D_{1,+},\ m_1- m_2 = 0,\\
        D_{1,-},\ m_1 + m_2 =0,          
    \end{cases}\quad
    C_{\Gamma,2} :=
    \begin{cases}
        D_{2,+}+D_{2,-},\ m_1\pm m_2 \neq 0,\\
        D_{2,+},\ m_1 - m_2 = 0,\\
        D_{2,-},\ m_1 + m_2 =0,            
    \end{cases}\\
    &D_{1,+} = \bigg( \frac{12\pi^2 }{\gamma^2}m_2  - \frac{(m_1n_2-m_2n_1)^2}{m_1^2 m_2 (m_1+m_2)^2}\bigg)^{-1},\quad D_{1,-} = -\bigg( \frac{12\pi^2}{\gamma^2} m_2 -\frac{(m_1n_2-m_2n_1)^2}{m_1^2 m_2 (m_1-m_2)^2}\bigg)^{-1},\\
    &D_{2,+} = \bigg( \frac{12\pi^2 }{\gamma^2}m_1  - \frac{(m_1n_2-m_2n_1)^2}{m_1 m_2^2 (m_1+m_2)^2}\bigg)^{-1},\quad D_{2,-} = -\bigg( \frac{12\pi^2}{\gamma^2} m_1 -\frac{(m_1n_2-m_2n_1)^2}{m_1 m_2^2 (m_2-m_1)^2}\bigg)^{-1},\\
    &\Gamma_{1,1}(\boldsymbol r) := \frac{\gamma^2(a_p+b_p)}{12 \pi^2 m_1}\sum_{l=0}^{p/2-1} \binom{p}{l}\binom{p-l}{l}\binom{p-2l}{p/2-l-1} \mathcal{A}_1^{p/2-l} \mathcal{A}_2^{l},\\
    &\Gamma_{1,2}(\boldsymbol r) := C_{\Gamma,1}(a_p+b_p)\sum_{l=0}^{p/2-1} \binom{p}{l+1} \binom{p-l-1}{l} \binom{p-2l-1}{p/2-l}\mathcal{A}_1^{l} \mathcal{A}_2^{p/2-l},\\
    &\Gamma_{2,1}(\boldsymbol r) := \frac{\gamma^2(a_p+b_p)}{12 \pi^2 m_2}\sum_{l=0}^{p/2-1} \binom{p}{l} \binom{p-l}{l} \binom{p-2l}{p/2-l-1} \mathcal{A}_1^{l} \mathcal{A}_2^{p/2-l},\\
    &\Gamma_{2,2}(\boldsymbol r) := C_{\Gamma,2}(a_p+b_p)\sum_{l=0}^{p/2-1} \binom{p}{l+1} \binom{p-l-1}{l} \binom{p-2l-1}{p/2-l}\mathcal{A}_1^{l+1} \mathcal{A}_2^{p/2-l-1}.
\end{align*}
Moreover,
\begin{align*}
    &\mathscr{M}_{1,even}(\boldsymbol r):= (a_p+b_p)\sum_{l=0}^{p/2-1} \binom{p}1 \binom{p-1}{l} \binom{p-l-1}{l} \binom{p-2l-1}{p/2-l-1} \mathcal{A}_1^{l} \mathcal{A}_2^{p/2-l-1} \\
    &\quad\cdot \frac{\hat{u}_0(1,1) \hat{u}_0(0,-1) + \hat{u}_0(1,-1) \hat{u}_0(0,1)}{ \hat{u}_0(1,0)}\\
    &+ (a_p+b_p) \sum_{l=0}^{p/2-1} \binom{p}1 \binom{p-1}{l} \binom{p-l-1}{l+1} \binom{p-2l-2}{p/2-l-1} \mathcal{A}_1^{l} \mathcal{A}_2^{p/2-l-1} \cdot \frac{\hat{u}_0(2,0) \hat{u}_0(-1,0)}{\hat{u}_0(1,0)} \\
    &+ (a_p+b_p)\sum_{l=0}^{p/2-2} \binom{p}1 \binom{p-1}{l} \binom{p-l-1}{l+3} \binom{p-2l-4}{p/2-l-2} \mathcal{A}_1^{l} \mathcal{A}_2^{p/2-l-2} \cdot \hat{u}_0(-2,0) \hat{u}_0(1,0)^{2}\\
    &+ (a_p+b_p)\sum_{l=0}^{p/2-2} \binom{p}1 \binom{p-1}{l+2} \binom{p-l-3}{l} \binom{p-2l-3}{p/2-l-2} \mathcal{A}_1^{l} \mathcal{A}_2^{p/2-l-2} \\
    &\quad \cdot [\hat{u}_0(-1,1) \hat{u}_0(0,-1)\hat{u}_0(1,0) + \hat{u}_0(-1,-1) \hat{u}_0(0,1)\hat{u}_0(1,0)]\\
    &+ (a_p+b_p) \sum_{l=0}^{p/2-2} \binom{p}1 \binom{p-1}{l+1} \binom{p-l-2}{l} \binom{p-2l-2}{p/2-l-2} \mathcal{A}_1^{l} \mathcal{A}_2^{p/2-l-2} \\
    &\quad \cdot [\hat{u}_0(0,2) \hat{u}_0(0,-1)^2 + \hat{u}_0(0,-2) \hat{u}_0(0,1)^2] \\
    &+ (a_{p+1}+b_{p+1})\sum_{l=0}^{p/2} \binom{p+1}{l+1} \binom{p-l}{l} \binom{p-2l}{p/2-l} \mathcal{A}_1^{l} \mathcal{A}_2^{p/2-l},\\
    &\mathscr{M}_{2,even}(\boldsymbol r) := (a_p+b_p)\sum_{l=0}^{p/2-1} \binom{p}1 \binom{p-1}{l} \binom{p-l-1}{l} \binom{p-2l-1}{p/2-l-1} \mathcal{A}_2^{l} \mathcal{A}_1^{p/2-l-1} \\
    &\quad\cdot \frac{\hat{u}_0(1,1) \hat{u}_0(-1,0) + \hat{u}_0(-1,1) \hat{u}_0(1,0)} {\hat{u}_0(0,1)}\\
    &+ (a_p+b_p) \sum_{l=0}^{p/2-1} \binom{p}1 \binom{p-1}{l} \binom{p-l-1}{l+1} \binom{p-2l-2}{p/2-l-1} \mathcal{A}_2^{l} \mathcal{A}_1^{p/2-l-1} \cdot \frac{\hat{u}_0(0,2) \hat{u}_0(0,-1)}{\hat{u}_0(0,1)} \\
    &+ (a_p+b_p)\sum_{l=0}^{p/2-2} \binom{p}1 \binom{p-1}{l} \binom{p-l-1}{l+3} \binom{p-2l-4}{p/2-l-2} \mathcal{A}_2^{l} \mathcal{A}_1^{p/2-l-2} \cdot \hat{u}_0(0,-2) \hat{u}_0(0,1)^{2}\\
    &+ (a_p+b_p)\sum_{l=0}^{p/2-2} \binom{p}1 \binom{p-1}{l+2} \binom{p-l-3}{l} \binom{p-2l-3}{p/2-l-2} \mathcal{A}_2^{l} \mathcal{A}_1^{p/2-l-2} \\
    &\quad \cdot [\hat{u}_0(1,-1) \hat{u}_0(-1,0)\hat{u}_0(0,1) + \hat{u}_0(-1,-1) \hat{u}_0(1,0)\hat{u}_0(0,1)]\\
    &+ (a_p+b_p) \sum_{l=0}^{p/2-2} \binom{p}1 \binom{p-1}{l+1} \binom{p-l-2}{l} \binom{p-2l-2}{p/2-l-2} \mathcal{A}_2^{l} \mathcal{A}_1^{p/2-l-2} \\
    &\quad \cdot [\hat{u}_0(2,0) \hat{u}_0(-1,0)^2 + \hat{u}_0(-2,0) \hat{u}_0(1,0)^2] \\
    &+ (a_{p+1}+b_{p+1})\sum_{l=0}^{p/2} \binom{p+1}{l+1} \binom{p-l}{l} \binom{p-2l}{p/2-l} \mathcal{A}_2^{l} \mathcal{A}_1^{p/2-l}.
\end{align*}
If $p$ is odd, we have
\begin{align*}
    \widehat{P_2(u_0+v)}(1,0)
    =&\widehat{P_2(u_0)}(1,0)+O(\boldsymbol r^{2p-1})\\
    =& (a_p+b_p)\sum_{\boldsymbol k_1+\cdots+\boldsymbol k_p = (1,0)} \hat{u}_0(\boldsymbol k_1)\cdots \hat{u}_0(\boldsymbol k_p) 
    + (a_{p+1} + b_{p+1}) \widehat{u_0^{p+1}}(1,0)\\ 
    &+ \widehat{P_{2,1,>p+1}(u_0)}(1,0) + \widehat{P_{2,2,>p+1}(u_0)}(1,0) + O(\boldsymbol r^{2p-1})\\
    =& (a_p+b_p)\sum_{\substack{ \boldsymbol k_1+\cdots+\boldsymbol k_p = (1,0) \\ |\boldsymbol k_1|_1+\cdots +|\boldsymbol k_p|_1 = p}} \hat{u}_0(\boldsymbol k_1)\cdots \hat{u}_0(\boldsymbol k_p)\\
    &+ (a_p+b_p)\sum_{L\geq p+2}\sum_{\substack{ \boldsymbol k_1+\cdots+\boldsymbol k_p = (1,0) \\ |\boldsymbol k_1|_1+\cdots +|\boldsymbol k_p|_1 = L}} \hat{u}_0(\boldsymbol k_1)\cdots \hat{u}_0(\boldsymbol k_p) + O( \boldsymbol r^{p+1})\\
    =& (a_p+b_p)\sum_{\substack{ \boldsymbol k_1+\cdots+\boldsymbol k_p = (1,0) \\ |\boldsymbol k_1|_1+\cdots +|\boldsymbol k_p|_1 = p}} \hat{u}_0(\boldsymbol k_1)\cdots \hat{u}_0(\boldsymbol k_p) + O(\boldsymbol r^{p+1})\\
    =& \mathscr{M}_{1,odd}(\boldsymbol r) \cdot \hat{u}_0(1,0) + O( \boldsymbol r^{p+1}) ,
\end{align*}
and
\begin{align*}
    \widehat{P_2(u_0+v)}(0,1)  = \widehat{P_2(u_0)}(0,1) + O(\boldsymbol r^{2p-1}) = \mathscr{M}_{2,odd}(\boldsymbol r) \cdot \hat{u}_0(0,1) + O(\boldsymbol r^{p+1}).
\end{align*}

If $p$ is even, we have
\begin{align*}
    &\widehat{P_2(u_0+v)}(1,0)\\
    =&\widehat{P_2(u_0)}(1,0)+O(\boldsymbol r^{2p-1})\\
    =& (a_p+b_p)\sum_{\boldsymbol k_1+\cdots+\boldsymbol k_p = (1,0)} \hat{u}_0(\boldsymbol k_1)\cdots \hat{u}_0(\boldsymbol k_p) 
    + (a_{p+1}+b_{p+1}) \sum_{\boldsymbol k_1+\cdots+\boldsymbol k_{p+1} = (1,0)} \hat{u}_0(\boldsymbol k_1)\cdots \hat{u}_0(\boldsymbol k_{p+1})\\
    &+ \widehat{P_{2,1,>p+1}(u_0)}(1,0) +
    \widehat{P_{2,2,>p+1}(u_0)}(1,0) + O(\boldsymbol r^{2p-1})\\
    =&(a_p+b_p) \sum_{\substack{ \boldsymbol k_1+\cdots+\boldsymbol k_p = (1,0) \\ |\boldsymbol k_1|_1+\cdots +|\boldsymbol k_p|_1 = p+1}} \hat{u}_0(\boldsymbol k_1)\cdots \hat{u}_0(\boldsymbol k_p)
    + (a_p+b_p) \sum_{L\geq p+2}\sum_{\substack{ \boldsymbol k_1+\cdots+\boldsymbol k_p = (1,0) \\ |\boldsymbol k_1|_1+\cdots +|\boldsymbol k_p|_1 = L}} \hat{u}_0(\boldsymbol k_1)\cdots \hat{u}_0(\boldsymbol k_p) \\
    &+ (a_{p+1}+b_{p+1}) \sum_{\substack{ \boldsymbol k_1+\cdots+\boldsymbol k_{p+1} = (1,0) \\ |\boldsymbol k_1|_1+\cdots +|\boldsymbol k_{p+1}|_1 = p+1}} \hat{u}_0(\boldsymbol k_1)\cdots \hat{u}_0(\boldsymbol k_{p+1})\\
    &+ (a_{p+1}+b_{p+1}) \sum_{L\geq p+2}\sum_{\substack{ \boldsymbol k_1+\cdots+\boldsymbol k_{p+1} = (1,0) \\ |\boldsymbol k_1|_1+\cdots +|\boldsymbol k_{p+1}|_1 = L}} \hat{u}_0(\boldsymbol k_1)\cdots \hat{u}_0(\boldsymbol k_{p+1}) + O(\boldsymbol r^{p+2})\\
    =&(a_p+b_p) \sum_{\substack{ \boldsymbol k_1+\cdots+\boldsymbol k_p = (1,0) \\ |\boldsymbol k_1|_1+\cdots +|\boldsymbol k_p|_1 = p+1}} \hat{u}_0(\boldsymbol k_1)\cdots \hat{u}_0(\boldsymbol k_p)\\
    &+ (a_{p+1}+b_{p+1}) \sum_{\substack{ \boldsymbol k_1+\cdots+\boldsymbol k_{p+1} = (1,0) \\ |\boldsymbol k_1|_1+\cdots +|\boldsymbol k_{p+1}|_1 = p+1}} \hat{u}_0(\boldsymbol k_1)\cdots \hat{u}_0(\boldsymbol k_{p+1})+ O(\boldsymbol r^{p+2})\\
    =  & \mathscr{M}_{1,even}(\boldsymbol r)\cdot \hat{u}_0(1,0) + O(\boldsymbol r^{p+2}),
\end{align*}
and
\begin{align*}
    &\widehat{P_2(u_0+v)}(0,1) = \mathscr{M}_{2,even}(\boldsymbol r)\cdot \hat{u}_0(0,1)+ O(\boldsymbol r^{p+2}),
\end{align*}

\bibliographystyle{alpha}
\bibliography{sample2}
\end{document}